\documentclass[12pt,reqno]{amsart}

\usepackage[margin=1.05in]{geometry}
\usepackage{
amsfonts,
amsmath,amssymb,amsbsy,amstext,amsthm,accents,enumerate,float,
verbatim, mathtools, url, 
bm,
mathbbol, 
cases}
\usepackage[dvipsnames]{xcolor}

\usepackage[shortlabels]{enumitem}
\usepackage[colorlinks=true,hyperindex, linkcolor=blue, pagebackref=false, citecolor=cyan]{hyperref}
\usepackage[alphabetic,lite,backrefs]{amsrefs}
\usepackage{lscape}
\usepackage{cleveref}

\usepackage{graphicx,pstricks}
\usepackage{graphics}
\usepackage{subcaption}
\usepackage{import}
\graphicspath{ {images/} }
\usepackage{subfiles}
\usepackage{float, pdfpages}

\usepackage[all,cmtip]{xy}
\usepackage{tikz}
\usetikzlibrary{arrows,decorations.markings, cd}
\tikzstyle arrowstyle=[scale=1]
\tikzstyle directed=[postaction={decorate,
decoration={markings,mark=at position .65 with {\arrow[arrowstyle]{stealth}}}}]

\usepackage{mathdots}
\usepackage{bbold} 
\usepackage{tikz-cd}  

\usepackage[shortlabels]{enumitem}
\usepackage{ytableau}
\usepackage[ruled,vlined]{algorithm2e}
\usepackage[all]{nowidow}
\usepackage{genyoungtabtikz}

\newtheorem{lemma}{Lemma}[section]
\newtheorem{theorem}[lemma]{Theorem}

\newtheorem{prop}[lemma]{Proposition}
\newtheorem{cor}[lemma]{Corollary}

\newtheorem{claim}[lemma]{Claim}

\newtheorem{lem-def}[lemma]{Lemma/Definition}
 \newtheorem*{thmA}{Theorem A}
 \newtheorem*{thmA'}{Theorem A'}
 \newtheorem*{thmB}{Theorem B}
 \newtheorem*{thmC}{Theorem C}
 \newtheorem*{thmD}{Theorem D}

\theoremstyle{definition}
\newtheorem{definition}[lemma]{Definition}
\newtheorem{example}[lemma]{Example}

\newtheorem{notation}[lemma]{Notation} 
\newtheorem{construction}[lemma]{Construction}

\newtheorem{remark}[lemma]{Remark}

\newtheorem*{conventions}{Conventions}

\newcommand{\Hilb}{\operatorname{Hilb}}

\newcommand{\Span}{\operatorname{span}}

\newcommand{\inn}{\operatorname{In}}

\newcommand{\evac}
{\operatorname{evac}}

\newcommand{\length}{\operatorname{length}}

\newcommand{\id}{\operatorname{id}}

\newcommand{\antidiag}{\operatorname{antidiag}}
\newcommand{\lad}{\operatorname{lad}}

\newcommand{\rank}{\operatorname{rank}}

\newcommand{\GL}{\mathrm{GL}}

\newcommand{\RNum}[1]{\uppercase\expandafter{\romannumeral #1\relax}}

\newcommand{\sgn}{\text{sign}}

\renewcommand{\aa}{\mathbf a}
\newcommand{\bb}{\mathbf b}
\newcommand{\cc}{\mathbf c}

\newcommand{\mm}{\mathbf m}

\renewcommand{\k}{\mathbb{k}}

\newcommand{\cI}{\mathcal{I}}
\newcommand{\cJ}{\mathcal{J}}

\newcommand{\cP}{\mathcal{P}}

\newcommand{\NN}{\mathbb{N}}
\newcommand{\ZZ}{\mathbb{Z}}

\newcommand{\rk}{\textrm{rk}}

\DeclarePairedDelimiter\abs{\lvert}{\rvert}%

\newcommand{\Mon}{\text{Mon}}

\newcommand{\res}{\mathrm{res}}

\def\ZZ{{\mathbb Z}}
\def\NN{{\mathbb N}}

\newcommand{\ch}{\mathrm{ch}}

\newcommand{\symm}{\mathcal{S}}

\newcommand{\gr}{\mathrm{Gr}}

\def\yshort{\ytableaushort}
\newcommand{\rr}[1]{\textcolor{red}{\textbf{{#1}}}}

\newcommand{\topp}{\operatorname{top}}
\newcommand{\bound}{\operatorname{Bound}}
\newcommand{\content}{\operatorname{content}}
\newcommand{\lub}{\operatorname{lub}}
\newcommand{\dem}{\operatorname{Dem}}
\newcommand{\Red}{\operatorname{Red}}

\newcommand{\hilb}{\operatorname{Hilb}}
\newcommand{\Gr}{\operatorname{Gr}}
\newcommand{\G}{\mathfrak{G}}
\newcommand{\genP}{\mathfrak{Y}}
\newcommand{\Lenart}{\mathcal{L}}

\begin{document}

\title[Standard Monomials for Positroid Varieties]{Standard Monomials for Positroid Varieties}
\author{Ayah Almousa}
\author{Shiliang Gao}
\author{Daoji Huang}


\date{\today}
\begin{abstract}
We give an explicit characterization of the standard monomials for positroid varieties with respect to the Hodge degeneration and give a Gr\"obner basis. 
 Furthermore, we 
 show that promotion and evacuation 
 biject standard monomials of a positroid variety with those of its cyclic shifts and $w_0$-reflection, respectively.
The connection to promotion allows us to identify standard monomials of a positroid variety with Lam's cyclic Demazure crystal. Using a recurrence on the Hilbert series, we give an inductive formula for the character of cyclic Demazure modules.
\end{abstract}
\maketitle
 \tableofcontents
\section{Introduction}
\label{sec:intro}
Classical work of Hodge \cite{Hodge} described a particular set of bases for the homogeneous coordinate rings of the Grassmannian $\gr(k,n)$ and its Schubert varieties under the Pl\"ucker embedding.
Building on Hodge's ideas, Seshadri initiated the study of standard monomial theory (SMT), with the aim of giving standard monomial bases for the space of global sections of line bundles over a (generalized) flag variety $G/P$.
The foundation of SMT was built in the works of Lakshmibai, Musili, and Seshadri,  \cites{GmodP1, GmodP2, GmodP3, GmodP4, GmodP5}.
The tools developed in SMT yield a wide range of applications, such as derivations of geometric properties (Cohen-Macaulayness, normality, singularities, cohomological vanishings, etc.)\! of Schubert varieties and character formulas for Demazure modules.

The main goal of our paper is to extend the work of Hodge to \emph{positroid varieties} in $\gr(k,n)$, based on works of Knutson--Lam--Speyer. 
Positroid varieties are the closed strata of the positroid stratification, the common refinement of the cyclically permuted Bruhat decompositions.
These are also exactly the projections of Richardson varieties from the complete flag variety to $\gr(k,n)$, and are
 the only compatibly split subvarieties with respect to the standard Frobenius splitting on the Grassmannian \cite{KLS2014projections}.
Since their introduction, positroid varieties have been of great interest in the study of Schubert calculus \cite{K14interval}, cluster algebras \cite{GL23cluster}, and positive geometries \cite{AHBL17}.

Using the standard monomial theory of Richardson varieties \cites{ lakshmibai2003richardson,BrionLak},
Knutson--Lam--Speyer \cite{KLS2014projections} described the Stanley--Reisner complexes of positroid varieties 
under the \emph{Hodge degeneration}, a special kind of Gr\"obner degeneration.
However, they did not give an explicit description (i.e., without referring to the chains in the Bruhat order of $G/B$) of the standard monomials. One primary goal of our paper is to give an explicit description for positroid varieties analogous to Hodge's original description based on semistandard Young tableaux. As we shall see, this approach yields several applications in algebra, combinatorics, and representation theory. More specifically:
\begin{itemize}
    \item We give an explicit Gr\"obner basis for positroid varieties with respect to the Hodge degeneration, analogous to the classical ``straightening relations'' for Hodge algebras. 
    \item We establish a connection between the \emph{promotion} (resp., \emph{evacuation}) operation on rectangular semistandard Young tableaux and  rotations  (resp., reflections) of positroid varieties.
    \item We give a formula for characters of \emph{cyclic Demazure modules}, using a recurrence on multigraded Hilbert series of positroid varieties. This resolves a problem of Lam's \cite{Lam19}.
\end{itemize}

\subsection{Hodge degeneration and standard monomials for positroid varieties}
We work throughout in the Pl\"ucker embedding of the Grassmannian $\gr(k,n)$.
For $\aa\in\binom{[n]}{k}$, let $[\aa]$ denote the corresponding Pl\"ucker coordinate, and set
\[
R(k,n):=\k[[\aa]:\aa\in\binom{[n]}{k}].
\]
Fix a degree reverse lexicographic term order $<_\omega$ on $R(k,n)$, in which the Pl\"ucker variables are ordered by a linear extension of the Pl\"ucker poset.
This term order defines the \emph{Hodge degeneration} of $\gr(k,n)$ and its subvarieties; see \Cref{section: initial-ideal}.
Under this degeneration, degree-$d$ standard monomials of $\gr(k,n)$ are naturally indexed by
$B(k,n,d)$, the set of rectangular semistandard Young tableaux of shape $k\times d$ with entries in $[n]$
(reviewed in \Cref{section: hodge-algebra-background}).

Positroid varieties in $\gr(k,n)$ admit several equivalent parametrizations.
Throughout the paper we index them by \emph{bounded affine permutations} $f$ of type $(k,n)$ and write
$\Pi_f\subseteq \gr(k,n)$ for the corresponding positroid variety; see \Cref{subsec:basic-positroid}.
Write $\ell(f)$ for the affine Coxeter length.
Every positroid variety admits a canonical irredundant presentation as an intersection of basic cyclic rank varieties
$\Pi_{S\le r}$, indexed by cyclic intervals $S$ (see \Cref{subsec:basic-positroid} and \eqref{eqn:basic-int}).

A \emph{basic positroid variety} $\Pi_{S\le r}\subseteq \gr(k,n)$ is defined by a single cyclic rank condition $S\le r$ where $S$ is a cyclic interval (Section~\ref{subsec:basic-positroid}). 

In \Cref{section: initial-ideal}, we prove an explicit combinatorial description for standard monomials of positroid varieties. In the wrapped-around case we apply a natural duality $(\cdot)^\vee$
that converts wrapped-around cyclic rank conditions into interval-type ones;
see Construction~\ref{const: complement-interval}.

\begin{thmA}[Theorem~\ref{thm:basic-std-mon} $+$ Theorem~\ref{thm:std-mon-intersect}]
    A monomial $\mm = \prod_{i = 1}^d [\aa^{(i)}]$ is a standard monomial for a positroid variety $\Pi_f$ if and only if $\mm$ is a standard monomial for $\gr(k,n)$ and 
    \begin{enumerate}
        \item when $\Pi_f = \Pi_{S\leq r}$ for some  interval $S$, $\mm$ does not contain a \textbf{generalized antidiagonal} for $S\leq r$ (Definition~\ref{def:antidiag});
        \item when $\Pi_f = \Pi_{S\leq r}$ for some wrapped-around interval $S$, $\mm^\vee$ does not contain a generalized antidiagonal for $S^\vee\leq r^\vee$ (Construction~\ref{const: complement-interval});
    \item when $\Pi_f=\bigcap_i \Pi_{S_i\le r_i}$ is a presentation as an intersection of basic positroid varieties,
$\mm$ is a standard monomial for all $\Pi_{S_i\le r_i}$\footnote{We note that there are possibly more than one way of writing $\Pi_f$ as an intersection of basic positroid varieties. One explicit construction is given in \Cref{def:essential-rank}. Nonetheless, the set of standard monomials for $\Pi_f$ is independent of the choice.}.
    \end{enumerate}
\end{thmA}

Using straightening relations and our characterizations for standard monomials, we obtain an explicit Gr\"obner basis for positroid varieties.

\begin{thmB}[Theorem~\Cref{thm: gb-basic-positroid} $+$ Proposition~\ref{prop:dual-gb} $+$ Theorem~\ref{thm:concat-gb}]\label{thmB}
Let $\cJ_f$ be the defining ideal for the positroid 
variety $\Pi_f$. $\cJ_f$ has a Gr\"{o}bner basis with respect to $<_\omega$ consisting of straightening relations (Definition~\ref{def:Pluckerrelations}), and
\begin{enumerate}[(1)]
    \item when $\Pi_f=\Pi_{S\le r}$ for some interval $S$, 
     $\sum_{w\in \symm_{r+1}}(-1)^{\ell(w)} w\cdot \mm$ where $\mm$ minimally contains a generalized antidiagonal for $S\leq r$ and $\symm_{r+1}$ (the symmetric group on $r+1$ letters) acts on $\mm$ by permuting entries in the generalized antidiagonal;

    \item when $\Pi_f=\Pi_{S\le r}$ for some wrapped-around interval $S$, $\sum_{w\in \symm_{r^\vee+1}} (-1)^{\ell(w)} (w\cdot \mm^\vee)^\vee$ where $\mm^\vee$ minimally contains a generalized antidiagonal for $S^\vee\leq r^\vee$; 

    \item when $\Pi_f = \bigcap_{i}\Pi_{S_i\leq r_i}$, the union of the Gr\"obner basis for each $\Pi_{S_i\le r_i}$.
\end{enumerate}
 
\end{thmB}

\begin{example}
    We illustrate the first case of Theorems A and B with a small example. Let $n=5$, $k=3$, $S=[2,4]$, $r=2$, and consider the interval positroid variety $\Pi_{[2,4]\le 2}$. The Gr\"obner basis of $\cJ_{[2,4]\le 2}$ consists of straightening relations for $\gr(3,5)$ and the following:
   
\[\ytableausetup{centertableaux, boxsize=1.1em}
\yshort{{\rr{2}},{\rr{3}},{\rr{4}}} \; , \hskip 2em
\yshort{1{\rr{2}},2{\rr{3}},{\rr{4}}5}-\yshort{1{\rr{2}},2{\rr{4}},{\rr{3}}5}\; ,\hskip 2em 
    \yshort{1{\rr{2}},3{\rr{3}},{\rr{4}}5}-\yshort{1{\rr{3}},{\rr{2}}{\rr{4}},35} \; , \hskip 2em
    \yshort{1{\rr{2}},{\rr{3}}4,{\rr{4}}5}-\yshort{1{\rr{3}},{\rr{2}}4,{\rr{4}}5} \; ,
    \]

 \[\ytableausetup{centertableaux}
\yshort{11{\rr{2}},2{\rr{3}}4,{\rr{4}}55}-\yshort{11{\rr{2}},2{\rr{4}}4,{\rr{3}}55}-\yshort{11{\rr{3}},2{\rr{2}}4,{\rr{4}}55}
 \; . \]
    In each term, the entries in the permuted generalized antidiagonal are highlighted. 
\end{example}

\subsection{Tableau symmetries and positroid varieties}
Our next application concerns \emph{promotion} and \emph{evacuation} introduced by Sch\"utzenberger \cites{evacuation,promotion}. 
Evacuation appears in the theory of canonical bases where the bijection by Berenstein-Zelevinsky \cite{BZ96} between canonical bases and semistandard tableaux sends Lusztig's involution to evacuation.
Stembridge \cites{Ste94,Ste96} identified a recurring \emph{$q=-1$ phenomenon}: in a number of representation-theoretic and tableau-enumeration settings, a
natural $q$-analogue $F(q)$ (typically a graded generating function) has the property that the specialization $F(-1)$ counts the fixed points of a distinguished involution.
Motivated by this perspective, Reiner--Stanton--White \cite{RSW04} introduced the \emph{cyclic sieving phenomenon}
for the enumeration of fixed points under cyclic group actions.
Rhoades \cite{rhoades2010cyclic} showed that promotion on $B(k,n,d)$ exhibits the cyclic sieving phenomenon, where the number of fixed tableaux after applying promotion $a$ times is the value of a polynomial at the  $n$th roots of unity $\xi_n^a$.
More recently, these operations are of interest in the study of dynamical algebraic combinatorics (see e.g., the survey
\cite{Roby} and the references therein). 
 
Let $\chi$ denote the cyclic rotation automorphism of $\gr(k,n)$ induced by the permutation $i\mapsto i+1\pmod n$ on coordinate vectors; it sends any positroid variety to a (generally different) positroid variety, denoted $\chi(\Pi_f)$.
Let $w_0\in S_n$ be the longest element (acting by reversing coordinates), and write $\Pi_{f^*}:=w_0\cdot \Pi_f$ for the reflected positroid variety; see \Cref{subsec:basic-positroid} for the corresponding action on the indexing data.

We relate promotion and evacuation to the cyclic and reflection symmetry of positroid varieties, respectively. 
 
\begin{thmC}[Theorem~\ref{thm:promotion}$+$Theorem~\ref{thm:evac}]

    Promotion (resp. evacuation) gives a bijection between the set of standard monomials of a positroid variety $\Pi_f$ and those of its cyclic shift $\chi(\Pi_f)$ (resp. its reflection $\Pi_{f^*}=w_0\cdot \Pi_f$).
\end{thmC}
The promotion statement in Theorem C allows us to identify standard monomials of a positroid variety with Lam's cyclic Demazure crystals \cite{Lam19}. This identification is a priori not apparent. It would be interesting to investigate a cyclic sieving/$q=-1$ phenomenon for standard monomials of positroid varieties.

\subsection{Cyclic Demazure modules and character formulas}
Motivated by the cyclic symmetry of the Grassmannian and its positroid varieties, Lam \cite{Lam19} associates to each bounded affine permutation $f\in\bound(k,n)$
and each integer $d\ge 0$ a finite-dimensional \emph{cyclic Demazure module} $V_f(d\omega_k)$ (see Definition~\ref{def:cyclicdemazure}).
These are naturally representations of the diagonal torus $T\subset GL_n$ (equivalently, objects in the representation ring $R(T)$), and may be realized as the
$T$-module of sections obtained by restricting $H^0(\gr(k,n),\mathcal{L}_{d\omega_k})$ to $\Pi_f$.
We write $\ch(\cdot)$ for the $T$-character in $\mathbb{Z}[\mathbf{t}^{\pm1}]$.

We establish a recurrence for the multigraded Hilbert series of positroid varieties (with respect to the standard $T$-action), which yields an inductive formula for $\ch(V_f(d\omega_k))$ and answers \cite{Lam19}*{Problem~24}.
Concretely, the recurrence reduces either the degree $d$ (via $V_f((d-1)\omega_k)$) or the dimension of the positroid variety (via terms indexed by $f'$ with larger length), so the character is determined by induction starting from $d=0$.

Write $\mathbf{t}=(t_1,\dots,t_n)$ for the character variables of $T$.
For $\aa\in\binom{[n]}{k}$ set $\mathbf{t}^{\aa}:=\prod_{i\in\aa} t_i$. Here $\topp(f)\in\binom{[n]}{k}$ corresponds to the smallest Pl\"ucker coordinate that does not vanish on $\Pi_f$, and $\Lenart_0(f)$ is the set of endpoints of saturated chains in Bruhat order. The precise definition is given in Section~\ref{subsec:hilbert-series-positroid}.

\begin{thmD}[Theorem~\ref{thm:character}]
For $f\in \bound(k,n)$, if $d=0$, then
\[
    \ch(V_f(d\omega_k)) = 1;
\]
otherwise,
\begin{equation}
\ch(V_f(d\omega_k))
=
\mathbf t^{\topp(f)}\left(
\ch(V_f((d-1)\omega_k))
+
\sum_{g\in \Lenart_0(f)\cap \mathrm{Bound}(k,n)}
(-1)^{\ell(g)-\ell(f)+1}\,
\ch(V_g(d\omega_k))
\right).
\end{equation}
\end{thmD}

Since the right-hand side only involves $(d-1)$ or bounded affine permutations of strictly larger length (hence positroid varieties of strictly smaller dimension),
Theorem~\ref{thm:character} computes $\ch(V_f(d\omega_k))$ by induction on $(d,\dim \Pi_f)$.
In particular, this recursion determines the coefficient of each monomial \(\mathbf t^\alpha\)
in \(\ch(V_f(d\omega_k))\) by a purely combinatorial rule, indexed by saturated \(0\)-Bruhat chains
starting at \(f\).

To establish the results in Section~\ref{sec:cyclic_demazure}, we prove a few auxiliary statements that may be of independent interest. In Proposition~\ref{prop:geometric-monk}, we give a geometric interpretation of Monk's rule in terms of matrix Schubert varieties. In the Appendix, we show that in an arbitrary Coxeter group, the least upper bounds of any set of elements in the same coset of a parabolic subgroup remain in the coset.

\vskip 1em

\noindent\textbf{Acknowledgments.}
We thank Allen Knutson, Thomas Lam, and David Speyer for inspiring conversations and helpful comments. We also thank Anders Buch, Sean Griffin, Matt Larson, Leonardo Mihalcea, Vic Reiner, Brendon Rhoades, Melissa Sherman-Bennett, Jessica Striker, Keller VandeBogert, Anna Weigandt and Alexander Yong for helpful conversations. DH would like to thank ICERM for the Combinatorial Algebraic Geometry Reunion Event, during which many fruitful conversations happened.
We also thank the anonymous referees for their helpful suggestions.
SG was supported by NSF Graduate Research Fellowship under grant No. DGE-1746047 and by an NSF RTG in Combinatorics DMS-1937241. DH was supported by NSF grant DMS-2202900.

%%%%%%%%%%%%%%%%%%%%%%%%%%%%
\section{Pl\"ucker embeddings and Hodge algebra structure of Pl\"{u}cker coordinates}\label{section: hodge-algebra-background}
%%%%%%%%%%%%%%%%%%%%%%%%%%%%

\subsection{Grassmannians and Pl\"ucker embeddings}

For $k\leq n\in \mathbb{Z}_{>0}$ and an algebraically closed field $\k$, the \textbf{Grassmannian} $\gr(k,n)$ is 
\[\gr(k,n) = \{W\subseteq \k^n: \dim(W) = k\}.\]
For $W\in \gr(k,n)$ and $\{w_1,\dots,w_k\}$ a choice basis of $W$, the \textbf{Pl\"ucker embedding} is the map
\begin{align*}
    \iota:\gr(k,n)&\rightarrow \mathbb{P}(\Lambda^k(\k^n))\\
    W &\mapsto [w_1\wedge \dots \wedge w_k].
\end{align*}

Set $R(k,n):= \k[[\aa]:\aa\in {[n]\choose k}]$ to be the homogeneous coordinate ring of the projective space $\mathbb{P}(\Lambda^k(\k^n))$. 
We extend the notation of $[\aa]$ to all sequences 
\[\aa = (a_1,\dots,a_k)\in [n]^k\]
where we use the convention that for any permutation $\sigma\in \symm_k$, 
\begin{equation}\label{eqn:extendedplucker}
    [a_{\sigma(1)},a_{\sigma(2)},\dots, a_{\sigma(k)}] = \sgn(\sigma)\cdot [a_1,\dots, a_k].
\end{equation}
In particular, this implies that $[\aa] = 0$ if $a_i = a_j$ for some $i\neq j\in [k]$.

\begin{conventions}
We write \([n]=\{1,\dots,n\}\) and \([j]=\{1,\dots,j\}\).
For \(\aa=(a_1,\dots,a_k)\in [n]^k\), the symbol \([\aa]\) denotes the corresponding
Pl\"ucker coordinate in \(R(k,n)\), interpreted using \eqref{eqn:extendedplucker}.
\end{conventions}

Define a partial order ``$\leq$" on ${[n]\choose k}$ by 
$$
[c_1,\dots, c_k] \leq [d_1,\dots, d_k] \text{ if and only if } c_i\leq d_i \text{ for all } i=1,\dots, k.
$$
We call this poset $\cP = ({[n]\choose k}, \leq)$ the \textbf{Pl\"{u}cker poset}.

The defining ideal $\cJ$ of $\gr(k,n)$ as a projective subvariety of $\mathbb{P}(\Lambda^k(\k^n))$ is generated by the following \textbf{straightening relations}:

\begin{definition}\cite{SturmfelsWhite}*{Section~2}\label{def:Pluckerrelations}
Fix \(s\in [k]\).
Let \(\alpha=(\alpha_1<\cdots<\alpha_{s-1})\in \binom{[n]}{s-1}\),
\(\beta=(\beta_1<\cdots<\beta_{k+1})\in \binom{[n]}{k+1}\),
and \(\gamma=(\gamma_1<\cdots<\gamma_{k-s})\in \binom{[n]}{k-s}\).
For a subset \(I=\{i_1<\cdots<i_s\}\subset [k+1]\), set
\(I^c=[k+1]\setminus I=\{i'_1<\cdots<i'_{k-s+1}\}\) and define
\[
\sgn(I):=\sum_{j=1}^s i_j-\binom{s+1}{2}.
\]
The \textbf{straightening relation} attached to \((\alpha,\beta,\gamma)\) is
\[
\sum_{I\in \binom{[k+1]}{s}}(-1)^{\sgn(I)}
[\alpha_1,\dots,\alpha_{s-1},\beta_{i'_1},\dots,\beta_{i'_{k-s+1}}]\cdot
[\beta_{i_1},\dots,\beta_{i_s},\gamma_1,\dots,\gamma_{k-s}].
\]
\end{definition}

The homogeneous coordinate ring of $\gr(k,n)$ is: 
\[\k[\gr(k,n)] = R(k,n)/\cJ.\]

\begin{example}
    The Grassmannian $\gr(2,4)$ is defined by one equation in $\mathbb{P}^5$:
    $$
\k[\gr(2,4)] \cong \frac{R(2,4)}{\cJ} = \frac{\k[[1,2],[1,3],[1,4],[2,3],[2,4],[3,4]]}{\langle [1,4][2,3] - [1,3][2,4] + [1,2][3,4]\rangle}.
$$
\end{example}

We will often\footnote{In some arguments we allow columns to be permuted; in that case, columns may be reordered into increasing order at the cost of a sign, as in \eqref{eqn:extendedplucker}.} represent a monomial
\[
\mm=\prod_{i=1}^d [\alpha^{(i)}], \qquad \alpha^{(i)}=(a^{(i)}_1,\dots,a^{(i)}_k)\in [n]^k,
\]
as a \(k\times d\) tableau whose \(i\)th column consists of the entries \(a^{(i)}_1,\dots,a^{(i)}_k\).
Thus every entry of \(\alpha^{(i)}\) appears exactly once in column \(i\).
By convention, we arrange each column in strictly increasing order; if \(\alpha^{(i)}\) is not initially increasing, then sorting the column changes the sign of \([\alpha^{(i)}]\) according to \eqref{eqn:extendedplucker}.

\begin{example}
The straightening relation generating the defining ideal for $\gr(2,4)$ could be represented  as
\begin{equation}\label{eq: plucker-gr-24}
\yshort{1{\rr{2}}, {\rr{4}}{\rr{3}}} \; - \; 
\yshort{1{\rr{2}}, {\rr{3}}{\rr{4}}} \; + \; 
\yshort{1{\rr{3}}, {\rr{2}}{\rr{4}}}
\end{equation}
where in this case, $\alpha = \{1\}$, \textcolor{red}{$\beta = \{2,3,4\}$}, and $\gamma = \emptyset$.
Another example of a straightening relation in $\gr(4,8)$ would arise from the product $[1,2,6,7]\cdot [3,4,5,8]$. Here $\alpha = \{1,2\}$, \textcolor{red}{$\beta = \{3,4,5,6,7\}$}, and $\gamma = \{8\}$.
$$
\yshort{1{\rr{3}}, 2{\rr{4}}, {\rr{6}}{\rr{5}},{\rr{7}}8} \; - \; \yshort{1{\rr{3}}, 2{\rr{4}}, {\rr{5}}{\rr{6}},{\rr{7}}8} \; + \; \yshort{1{\rr{3}}, 2{\rr{4}}, {\rr{5}}{\rr{7}},{\rr{6}}8} \; - \; \yshort{1{\rr{3}}, 2{\rr{5}}, {\rr{4}}{\rr{7}},{\rr{6}}8} \; \cdots 
$$
\end{example}

\begin{definition} A tableau is called \textbf{semistandard} if it is strictly increasing along columns and weakly increasing along rows.
\end{definition}

The coordinate ring $\k[\gr(k,n)]$ is the direct sum of its graded pieces:
\[
\k[\gr(k,n)] = \bigoplus_{d = 0}^\infty \k[\gr(k,n)]_d, 
\]
where each $\k[\gr(k,n)]_d$ is a finite dimensional $\k$-vector space spanned by the monomials $\{\prod_{i = 1}^d [\aa^{(i)}]: \aa^{(i)}\in {[n]\choose k}\text{ for all }i\}$.
The following classical theorem is originally due to Hodge \cite{Hodge} (c.f.\cite{bruns1998cohen}*{Lemma 7.2.3}).

\begin{theorem}\label{thm:hodge}
Fix \(d\ge 0\). Identify a degree-\(d\) monomial
\(\prod_{i=1}^d[\alpha^{(i)}]\in \k[\gr(k,n)]_d\) with the \(k\times d\) tableau
whose \(i\)th column is \(\alpha^{(i)}\) (written in increasing order, using
\eqref{eqn:extendedplucker} to account for signs).
Then the monomials corresponding to semistandard tableaux form a \(\k\)-basis of
\(\k[\gr(k,n)]_d\).
\end{theorem}

These monomials are often referred to as the \textbf{standard monomials} for $\gr(k,n)$. The precise definition of standard monomials for subvarieties of $\gr(k,n)$ will be given in Definition~\ref{def:std-mon}.

\subsection{Bruhat decomposition of \texorpdfstring{$\gr(k,n)$}{Gr(k,n)}}

Fix $\{e_1,\dots,e_n\}$ to be a basis for $\k^n$. Set $E_\bullet = 0\subsetneq E_1 \subsetneq \dots \subsetneq E_{n} = \k^n$ to be the standard flag where $E_i = \langle e_1,\dots,e_i\rangle$. For any subspace $E_i$, define $\mathrm{pr}_{E_i}:\k^n\twoheadrightarrow E_i$ to be the  projection onto the coordinates in $E_i$. 

Following \cite{KLS13juggling}*{Section 4.2}, for $I = \{i_1<\dots<i_k\}\in {[n]\choose k}$, define the Schubert cell to be
\begin{equation}\label{eqn:Schubertcell}
    X_I^\circ:= \left\{V\in \gr(k,n): \dim(\mathrm{pr}_{E_j}(V)) = \abs{I\cap [j]}\text{ for all }j\in [n]\right\}.
\end{equation}
The Bruhat decomposition of $\gr(k,n)$ is 
\begin{equation}\label{eqn:Bruhat}
    \gr(k,n) = \bigsqcup_{I\in {[n]\choose k}}X_I^\circ.
\end{equation}

Define the Schubert variety 
\begin{equation}\label{eqn:Schubertvariety}
    X_I:= \left\{V\in \gr(k,n): \dim(\mathrm{pr}_{E_j}(V)) \leq \abs{I\cap [j]}\text{ for all }j\in [n]\right\}.
\end{equation}
This is the closure of $X_I^\circ$ and has a decomposition into Schubert cells:
\[X_I = \bigsqcup_{I\leq J} X_J^\circ.\]
  Note that $\gr(k,n) = X_{\{1,2,\cdots, k\}}$. 

The defining ideal of $X_I$ is (see \cite{S07}*{Corollary 1.3.13})
\[\cJ_{X_I} = \cJ + \langle [\aa]: I\nleq \aa \rangle. \]
Let $\k[X_I] = R(k,n)/\cJ_{X_I}$. Similar to the case of $\gr(k,n)$, Hodge obtained a description of standard monomials for Schubert varieties:

\begin{theorem}[See \cite{lakshmibai2008standard} (Lemma~4.3.4.1)]
Write \(I=\{i_1<\cdots<i_k\}\in \binom{[n]}{k}\).
The monomials \(\prod_{t=1}^d[\alpha^{(t)}]\) that correspond to semistandard
\(k\times d\) tableaux whose \(j\)th row entries are all \(\ge i_j\) (for each
\(1\le j\le k\)) form a basis of \(\k[X_I]_d\).
\end{theorem}

Let $\prec$ be any linear extension of $\mathcal{P}$
(e.g., the lexicographic order on $\binom{[n]}{k}$) and let $<_\omega$ be the induced degree reverse lexicographic monomial order on $R(k,n)$. 

\begin{definition}\label{def:std-mon}
Fix a linear extension \(\prec\) of the Pl\"ucker poset \(\mathcal P\) and let \(<_\omega\)
be the induced degree reverse lexicographic term order on \(R(k,n)\).
For a subvariety \(X\subseteq \gr(k,n)\), let \(\cJ_X\subseteq R(k,n)\) be its defining ideal
under the Pl\"ucker embedding.
A monomial \(\mm=\prod_{i=1}^d[\aa^{(i)}]\) is a \textbf{standard monomial} for \(X\)
(with respect to \(<_\omega\)) if \(\mm\notin \inn_\omega(\cJ_X)\).
\end{definition}

The monomials in Theorem~\ref{thm:hodge} are standard monomials in the sense of \Cref{def:std-mon} by the following (see also {\cite{SturmfelsAlgo}*{Theorem~3.1.7}):

\begin{theorem}\cite{SturmfelsWhite}\label{thm:grgrobner}
Let \(\mathcal P=(\binom{[n]}{k},\le)\) be the Pl\"ucker poset (coordinatewise order),
and fix the term order \(<_\omega\) induced by a linear extension \(\prec\) of \(\mathcal P\).
Then the straightening relations in Definition~\ref{def:Pluckerrelations} form a
Gr\"obner basis of \(\cJ\) with respect to \(<_\omega\). Moreover,
\[
\inn_\omega(\cJ)=\langle [\aa][\bb] : \aa,\bb\in \binom{[n]}{k}
\text{ are incomparable in }\mathcal P \rangle.
\]
\end{theorem}

Therefore, with respect to the ordering $<_\omega$, standard monomials for $\gr(k,n)$ correspond to \textbf{multichains} $\aa_{i_1} \leq \aa_{i_2} \leq \dots \leq \aa_{i_d}$ in the poset $\cP$, and standard monomials for $X_I$ correspond to multichains in the upper order ideal $\{\aa\in \cP: \aa\geq I\}$. 
The Gr\"{o}bner degeneration from \Cref{thm:grgrobner} is often referred to as a \textbf{Hodge degeneration}.
\begin{remark}
    When $s = 1$ in \Cref{def:Pluckerrelations}, the straightening relations specialize to 
    a set of relations which also generate the ideal $\cJ$, but in general do not form a Gr\"obner basis \cite{SturmfelsAlgo}*{Chapter~3.1 Exercise (3)}.
\end{remark}

Using the straightening relations described in the previous subsection, one can show that the Pl\"{u}cker algebra is a \emph{Hodge algebra} or an \emph{algebra with straightening law (ASL)} associated to a poset $H$.  
We refer the reader to \cite{bruns2022determinants}*{Section 6.2} for a more in-depth discussion.

\section{Positroid Varieties}
\label{section: positroid-background}
A comprehensive introduction of positroid varieties can be found in \cite{speyer2023richardson}. We recommend that readers unfamiliar with positroid varieties refer to it as a supplement.

Let $\k$ be an algebraically closed field of characteristic 0. 
In this section, we give four equivalent definitions of  positroid varieties using
\begin{itemize}
    \item bounded affine permutations (Definition~\ref{def:positroid-baf}), 
    \item cyclic rank matrices (Definition~\ref{def:essential-rank}), 
    \item projection of Richardson varieties (Definition~\ref{def:projrichardson}), and 
    \item cyclic intersection of Schubert varieties (Proposition~\ref{prop:positroidcyclicschub}).
\end{itemize}
The equivalence of these definitions is Theorem~5.1 of \cite{KLS13juggling}.

\subsection{Bounded affine permutations} Let $\widetilde{\mathcal{S}}_n$ be the group of bijections $f:\ZZ \rightarrow \ZZ$ such that 
\[f(i+n) = f(i)+n,\]
where the group operation is composition. We refer to these bijections as \textbf{affine permutations}.
We will sometimes write an affine permutation as $f = [ f(1)\ f(2) \cdots f(n) ]$. Define
\[\widetilde{S}_n^m = \{f\in \widetilde{\mathcal{S}}_n: \sum_{i = 1}^n (f(i)-i) = mn\}.\]
In particular, $W_{\widehat{A}_{n-1}}:=\widetilde{S}_n^0$ is the affine Coxeter group of type $\widehat{A}_{n-1}$.

We  recall some facts about affine Bruhat order on $W_{\widehat{A}_{n-1}}$, following \cite{BB2005}. For $a<b\in\ZZ$ such that $a \not\equiv b \pmod n$, a \textbf{reflection}
$t_{a,b}$ is defined as 
\[t_{a,b}:=\prod_{r\in \ZZ}(a+rn, b+rn). \]

For $f,g\in W_{\widehat{A}_{n-1}}$, we say $f\gtrdot g$ if there exist $a,b\in \ZZ$ such that $a<b$,  $g(a)<g(b)$, $f=gt_{a,b}$, and for all $a<i<b$, we have $g(i)<g(a)$ or $g(i)>g(b)$. The \textbf{Bruhat order} is obtained by taking the transitive closure of the covering relation $\gtrdot$. The Bruhat order on $W_{\widehat{A}_{n-1}}=\widetilde{S}_n^0$ induces a partial order on $\widetilde{S}_n^m$ for each $m\in \ZZ$. For $k=0,\dots,n-1$, define the $k$-Bruhat order to be the transitive closure of the covering relation $\gtrdot_k $, where $f\gtrdot_k g$ is defined as $f\gtrdot g$ and $f=gt_{a,b}$ where $a\le k' < b$ for some $k'\equiv k \pmod n$. (We won't need this $k$-Bruhat order until Section~\ref{sec:cyclic_demazure}).

To compute the length of $f\in \widetilde{S}_n$, we have the formula
    \[
    \ell(f)=\mathrm{inv}(f(1),\ldots , f(n))+\sum_{1\leq m < l \leq n} \left\lfloor \frac{|f(m)-f(l)|}{n} \right\rfloor,
    \]
    where $\mathrm{inv}(f(1),\ldots , f(n))=|\{ (m,l) \;|\; 1\leq m < l \leq n, f(m)>f(l) \}|$ is the number of ordinary inversions.
    
Define the set of \textbf{bounded affine permutations} to be  
\begin{equation}\label{eqn:bounded}
        \bound(k,n) := \{f\in \widetilde{S}_n^k: i\leq f(i)\leq  i+n \text{ for all $i\in \ZZ$}
        \}.
\end{equation}
The partial order on $\widetilde{S}_n^k$ then restricts to a partial order on $\bound(k,n)$.

Given \(V\in \gr(k,n)\), fix a choice of \(k\times n\) matrix \(\widetilde{V}\) whose row span equals \(V\).
Let \(v_1,\dots,v_n\) denote the column vectors of \(\widetilde{V}\), and extend this sequence periodically by setting
\[
v_i := v_{i+n} \qquad \text{for all } i\in \mathbb{Z}.
\]
For integers \(i\le j\), let \(\widetilde{V}_{[i,j]}\) denote the matrix with column vectors \(v_i,\dots,v_j\).
For integers $i\le j$ with $j\le i+n$, define the cyclic interval
\[
[i,j]^\circ := \{i,i+1,\dots,j\}\subset \mathbb Z,
\]
and let $\widetilde V_{[i,j]^\circ}$ be the $k\times |[i,j]^\circ|$ submatrix of $\widetilde V$ whose columns are $v_i,v_{i+1},\dots,v_j$ (using the periodic extension above).
Set
\[
\rank(\widetilde V_{[i,j]^\circ}) := \dim \Span(v_i,v_{i+1},\dots,v_j).
\]

Associated to \(\widetilde{V}\) is the affine permutation
\begin{equation}\label{eqn:f_V}
f_{\widetilde{V}}:\mathbb{Z}\to\mathbb{Z}, \qquad
f_{\widetilde{V}}(i) := \min\{\, j\ge i \mid v_i \in \Span(v_{i+1},\dots,v_j) \,\}.
\end{equation}
This affine permutation lies in \(\bound(k,n)\), and every bounded affine permutation arises in this way.
Moreover, \(f_{\widetilde{V}}\) depends only on the row span \(V=\mathrm{rowspan}(\widetilde{V})\), and not on the chosen representative \(\widetilde{V}\).
We therefore denote this affine permutation by \(f_V\).

\begin{definition}\label{def:positroid-baf}
The \textbf{open positroid variety }associated to a bounded affine permutation $f$ is
\[\Pi_{f}^\circ \coloneqq \{V\in \gr(k,n):f_V = f\},\]
and the \textbf{positroid variety} is its Zariski closure
$\Pi_f = \overline{\Pi_f^{\circ}}$. 
\end{definition}
In fact, we have
\[\Pi_f = \bigsqcup_{f'\geq f}\Pi_{f'}^\circ.\]
Define $\cJ_f\subset R(k,n)$ to be the defining ideal of $\Pi_f$ under the Pl\"ucker embedding.

\subsection{Cyclic rank matrices}
A second way to define positroid varieties is through \textbf{cyclic rank matrices} \cite{KLS13juggling}*{Corollary 3.12}.

\begin{conventions}
Given $f\in \bound(k,n)$, write $f$ as the $\infty\times\infty$ permutation matrix with $1$'s in positions $(i,f(i))$ and $0$'s elsewhere.
For integers \(i\in\mathbb Z\) and \(j\) with \(i\le j\le i+n\), we write
\([i,j]^\circ=\{i,i+1,\dots,j\}\subset\mathbb Z\).
When drawing the \(\infty\times\infty\) permutation matrix of an affine permutation,
rows are indexed by \(i\in\mathbb Z\) (increasing downward) and columns by \(j\in\mathbb Z\)
(increasing to the right). “Weakly southwest of \((i,j)\)” means row index \(\ge i\)
and column index \(\le j\).
\end{conventions}

Define the \textbf{cyclic rank matrix} $r(f)$ by
\[
r(f)_{i,j}
=
|[i,j]^\circ|
-
\#\{\text{1s in the matrix of }f\text{ weakly southwest of }(i,j)\},
\qquad (i\in \mathbb Z,\ i\le j\le i+n).
\]
If $V\in \gr(k,n)$ satisfies $f_V=f$, then
\[
r(f)_{i,j}=\rank(\widetilde V_{[i,j]^\circ})=\dim\Span(v_i,\dots,v_j)
\qquad \text{for all } i\le j\le i+n.
\]

A \textbf{cyclic rank condition} on $U\in \gr(k,n)$ is an inequality of the form
\[
\rank(\widetilde U_{[i,j]^\circ}) \le r
\qquad (i\in\mathbb Z,\ i\le j\le i+n,\ r\in \mathbb Z).
\]

\noindent
The open positroid variety associated to $f$ can be described by equalities of cyclic ranks:
\[
\Pi_f^\circ
=
\{V\in \gr(k,n): \rank(\widetilde V_{[i,j]^\circ})= r(f)_{i,j}
\text{ for all }i\in\mathbb Z,\ j\in[i,i+n]\},
\]
and $\Pi_f$ is obtained by replacing ``$=$'' with ``$\le$''.

\noindent
The \emph{essential set} \(ess(f)\) is the finite set of “corner positions” in the affine permutation matrix
at which the cyclic rank inequalities are not implied by neighboring inequalities. Formally,
define the \textbf{essential set of \(f\)} by
\[
ess(f):=\Bigl\{(i,j)\ \Big|\ 
\begin{array}{l}
i\in\mathbb Z,\ i\le j\le i+n,\\[2pt]
f(i-1)>j,\quad f^{-1}(j+1)<i,\\
f(i)\le j,\quad f^{-1}(j)\ge i
\end{array}
\Bigr\}.
\]
A diagrammatic description and illustration for $ess(f)$ is given in \cite{K14interval}*{Section 2.1}. We omit this description, but  encourage the reader to refer to \cite{K14interval}*{Section 2.1}. This is analogous to Fulton's essential set condition.

\begin{definition}\label{def:essential-rank}
    The positroid variety is 
    \begin{equation}\label{eqn:essential-rank}
    \Pi_f=\{U\in \gr(k,n): \rank(\widetilde{U}_{[i,j]^\circ})\leq r(f)_{i,j}\text{ for all }(i,j)\in ess(f)\}.
\end{equation}
\end{definition}

The following statement, which follows from \cite{KLS2014projections}*{Theorem 5.1} is crucial for our main results.

\begin{prop}\label{prop:basic-intersection}
Let \(f\in\bound(k,n)\) and let \(ess(f)\) be as above.
Then the ideal defining \(\Pi_f\subset\gr(k,n)\) is generated scheme-theoretically by
the minors of the \(k\times |[i,j]^\circ|\) submatrices \(\widetilde U_{[i,j]^\circ}\)
which express the finitely many inequalities
\[
\rank(\widetilde U_{[i,j]^\circ})\le r(f)_{i,j}\qquad\text{for }(i,j)\in ess(f).
\]
Equivalently, \(\Pi_f\) is the scheme-theoretic intersection of the closed subschemes cut out by
these essential cyclic rank conditions.
\end{prop}

\begin{example}
\label{ex:positroid-decompose}
    Set $k=3$, $n=6$. Let $f$ be the bounded affine permutation \[[5,2,4,9,7,12].\] Then 
    \[\Pi_f=\{U\in \gr(3,6): \rank(\widetilde{U}_{[2,2]})\le 0,\rank(\widetilde{U}_{[2,4]})\le 1,\rank(\widetilde{U}_{[1,5]})\le 2,\rank(\widetilde{U}_{[5,2]^\circ})\le 2\}.\]
\end{example}
\subsection{\texorpdfstring{$k$}{k}-Bruhat orders and Grassmann intervals}

For permutations $u,v\in \symm_n$, we say $u$ \textbf{$k$-covers} $v$, denoted $u\gtrdot_k v$, if $u\gtrdot v$ in strong Bruhat order and $\{u(1),\cdots, u(k)\}\neq \{v(1),\cdots, v(k)\}$. The \textbf{$k$-Bruhat order} is the partial order on $\symm_n$ generated by taking the transitive closure of the $k$-covering relation $\gtrdot_k$. Let $[v,u]_k\subset \symm_n$ be the $k$-Bruhat interval. This is a graded poset of rank $\ell(u)-\ell(v)$, where $\ell(u)$ denotes the Coxeter length of the permutation $u$.

Consider the equivalence relation on the set of $k$-Bruhat intervals where
$[v,u]_k \sim [v',u']_k$
if there exists $z\in \symm_k\times \symm_{n-k}$ (where $\symm_k\times \symm_{n-k}$ is identified as a subgroup of $\symm_n$) such that $v' = vz, u' = uz$ where $\ell(v') = \ell(v)+\ell(z)$ and $\ell(u') = \ell(u)+\ell(z)$. Let $\mathcal{Q}(k,n)$ be the equivalence classes of $k$-Bruhat intervals.

Let $G = \GL_n(\k)$ and $B,B_-$ be the Borel and opposite Borel subgroup of $G$ consisting of invertible upper and lower triangular matrices, respectively. 
Upon choosing a basis of $\k^n$, a point $gB\in G/B$ can be identified with a complete flag $$F_{\bullet} = 0\subsetneq F_1 \subsetneq F_2 \subsetneq  \cdots \subsetneq F_{n-1}\subsetneq F_n  = \k^n,$$ where each $F_i$ is the span of the first $i$ columns of any representative of $gB$.

For $w\in\symm_n$, define the \textbf{Schubert variety} $X_w$ and the \textbf{opposite Schubert variety} $X^w$ to be the Zariski closure of the \textbf{Schubert cell} $X_w^\circ := B_-wB/B$ and \textbf{opposite Schubert cell} $X^w_\circ := BwB/B$ respectively. Define the \textbf{Richardson variety} $X_v^u$ to be the intersection $X^u\cap X_v$. In particular, $X_v^u$ is non-empty if and only if $v\leq u$.

Let $\pi: \mathrm{Fl}(n) \rightarrow \gr(k,n)$ be the natural projection where
$\pi(F_{\bullet}) = F_k$. It follows from \cite{KLS2014projections}*{Proposition~3.3} that  If $[v,u]_k\sim [v',u']_k$, then $\pi(X_v^u) = \pi(X_{v'}^{u'})$. This allows the following definition.

\begin{definition}\label{def:projrichardson}
Let \(u,v\in\symm_n\) with \(u\ge_k v\) in \(k\)-Bruhat order.
Let \(X_v^u=X^u\cap X_v\subset \mathrm{Fl}(n)\) be the Richardson variety and the projection
\(\pi:\mathrm{Fl}(n)\to \gr(k,n)\).
Define the associated positroid variety by
\[
\Pi_v^u:=\pi(X_v^u).
\]
This depends only on the equivalence class of the \(k\)-Bruhat interval \([v,u]_k\in\mathcal Q(k,n)\).
\end{definition}

For a permutation $w\in \symm_n$, define the descent set of $w$ to be 
\[\mathrm{Des}(w):= \{i\in [n-1]:w(i)>w(i+1)\}.\]
We say a permutation $w$ is $k$-Grassmannian for some $k<n$ if $\mathrm{Des}(w)\subseteq \{k\}$.

\begin{lem-def}[\cite{KLS13juggling}, Proposition~2.3]
     For every equivalence class in $\mathcal{Q}(k,n)$, there is a unique representative $[v,u]_k$ such that $u$ is a $k$-Grassmannian permutation. We will call such an interval a \textbf{Grassmann interval}.
\end{lem-def}

Since $\Pi_v^u$ is independent of the choice of representative in each equivalence class in $\mathcal{Q}(k,n)$, we will assume $[v,u]_k$ is the Grassmann interval unless specified otherwise.

For $I = \{i_1<i_2<\cdots<i_k\}\subset [n]$, write $I^{\vee} = [n]\setminus I = \{i^\vee_1<\cdots<i^\vee_{n-k}\}$. Define the Grassmannian permutation $w_I$ as follows: 
\[w_I(j) = \begin{cases}
    i_j & \text{ if }j\leq k\\
    i^{\vee}_{j-k} & \text{ otherwise}
\end{cases}.\]

\begin{lemma}[\cite{KLS13juggling}, Proposition~3.15]\label{lemma:wtouv}
Let \(f\in \bound(k,n)\), written in window notation \(f=[f(1)\ \cdots\ f(n)]\).
Let \(I=\{i\in[n]: f(i)>n\}\), and let \(\widehat f\in\symm_n\) be defined by
\(\widehat f(i)\equiv f(i)\pmod n\) for \(1\le i\le n\).
Let \(w_I\in\symm_n\) be the Grassmannian permutation associated to \(I\) defined above.
Then \(\Pi_f=\Pi_{v(f)}^{u(f)}\) where \(u(f)=w_I\) and \(v(f)=\widehat f\,u(f)\).
In particular, \(u(f)\) is \(k\)-Grassmannian.
\end{lemma}

\begin{example}
    Take $f=[5,2,4,9,7,12]$ as in Example~\ref{ex:positroid-decompose}. Then $\Pi_f=\Pi_{v(f)}^{u(f)}$ where $u(f)=456123$ and $v(f)=316524$.
\end{example}

The action of the longest element $w_0\in \symm_n$ on $\gr(k,n)$ is an involution on the positroid stratification. Specifically, for $w\in \symm_n$, $w_0X^w=w_0\cdot\overline{BwB/B}=w_0\cdot\overline{w_0B_-w_0w/B}=\overline{B_-w_0w/B}=X_{w_0w}$, and similarly
$w_0X_w=X^{w_0w}$. Therefore, for $v<u\in \symm_n$, we have $w_0X_v^u=X^{w_0v}_{w_0u}$. Projecting to $\gr(k,n)$, we have $w_0\Pi_v^u=\Pi_{w_0u}^{w_0v}$. 
For $f\in \bound(k,n)$, define $f^*\in \bound(k,n)$ such that $\Pi_{f^*}=\Pi^{w_0v(f)}_{w_0u(f)}$.
Equivalently, if \(f=[f(1)\ \cdots\ f(n)]\) is written in window notation, then
\(f^*\) is the unique bounded affine permutation whose juggling state is obtained
by reversing the cyclic order and replacing each entry \(i\) by \(n+1-i\).
We will connect this action of $w_0$ to evacuation in Section~\ref{subsec:evacuation}. 
\subsection{Cyclic rotations of Bruhat decomposition}

Define the cyclic rotation map
\begin{align}\label{eqn:defchi}
\begin{split}
    \chi:\k^n&\longrightarrow \k^n\\
    e_i &\mapsto e_{i+1}
\end{split}
\end{align}
where the indices are taken mod $n$. We abuse the notation and denote $\chi:\gr(k,n)\rightarrow \gr(k,n)$ the cyclic rotation on $\gr(k,n)$ induced by the map $\chi$ as in \eqref{eqn:defchi}. This is to say, for $V\in \gr(k,n)$ with a representative
   \[\widetilde{V} = 
\begin{bmatrix}
    \vert & \vert & \  & \vert \\
    v_1   & v_2 & \cdots & v_n   \\
    \vert & \vert & \  & \vert
\end{bmatrix},\] 
we have
\begin{equation}\label{eqn:chi}
    \chi(\widetilde{V}) := 
    \begin{bmatrix}
        \vert & \vert & \  & \vert \\
        v_n   & v_1 & \cdots & v_{n-1}   \\
        \vert & \vert & \  & \vert
    \end{bmatrix}\text{ and } \chi(V):=\mathrm{rowspan}(\chi(\widetilde{V})).
    \end{equation}

Similar to \eqref{eqn:Bruhat}, for every $0\le j\le n-1$, we have the rotated Bruhat decomposition:
\begin{equation}
    \gr(k,n) = \bigsqcup_{I\in {[n]\choose k}} \chi^j(X_I^\circ).
\end{equation}

The positroid stratification is the common refinement of $n$ cyclically rotated Bruhat decompositions:
\begin{equation}\label{eqn:positroidstrat}
    \gr(k,n) = \bigsqcup_{(I^{(0)},\dots,I^{(n-1)})\in {[n]\choose k}^n} \bigcap_{j = 0}^{n-1} \chi^j X^\circ_{I^{(j)}}. 
\end{equation}
The strata that are not empty are indexed by the $n$-tuples $\mathbf{I} = (I^{(0)},\dots,I^{(n-1)})$ such that for every $j=0,\cdots,n-1$,
\begin{equation}\label{eqn:jugglingstate}
    I^{(j+1)} \supseteq (I^{(j)}\setminus \{1\})-1,
\end{equation}
where the ``$-1$" means subtract $1$ from every element in the set $I^{(j)}\setminus \{1\}$. We use the convention that $I^{(n)} := I^{(0)}$. 
The tuples satisfying \eqref{eqn:jugglingstate} are also known as \textbf{juggling states} (see \cite{KLS13juggling}*{Section 3.3}).

\begin{definition}\label{def:jugglingstate}
    For any juggling state $\mathbf{I}$, the open positroid variety is
    \[\Pi_{\mathbf{I}}^\circ = \bigcap_{j = 0}^{n-1} \chi^j X^\circ_{I^{(j)}},\]
and the positroid variety is $\Pi_{\mathbf{I}} = \overline{\Pi_{\mathbf{I}}^\circ}$. 
\end{definition}
In fact, each $\Pi_{\mathbf{I}}$ is the intersection of $n$ cyclically rotated Schubert varieties:
\begin{prop}[\cite{KLS13juggling}*{Theorem~5.10}]\label{prop:positroidcyclicschub}
    $\Pi_{\mathbf{I}} = \bigcap_{j = 0}^{n-1} \chi^j X_{I^{(j)}}$.
\end{prop}

We  define the  cyclic shift on $\bound(k,n)$ as
\[\chi(f)(i) = f(i-1)+1.\]

For $f\in\bound(k,n)$, we can compute the corresponding juggling state $\mathbf{I}(f)=(I^{(0)},\cdots,I^{(n-1)})$ where $I^{(j)}:=v(\chi^{-j}(f))([k])$. Then $\Pi_f=\Pi_\mathbf{I}$.

\begin{example}
    Let $f=[5,2,4,9,7,12]$ as in Example~\ref{ex:positroid-decompose}. Then 
    \[\mathbf{I}(f)=(\{1,3,6\},\{2,4,5\},\{1,3,4\},\{1,2,3\},\{1,2,5\},\{1,2,4\}).\]
\end{example}

\subsection{Basic positroid varieties}
\label{subsec:basic-positroid}

For any $S\subset [n]$ and any $r\in \mathbb{N}$, define
\[\Pi_{S\leq r} = \{M\in \gr(k,n): \rank(M_{S}) \leq r\},\]
where $\rank(M_S)$ is the rank of the submatrix of $M$ with column index $S$.

\begin{lemma}\label{lemma:chi}
    For $f\in \bound(k,n)$, $\chi(\Pi_{f}) = \Pi_{\chi(f)}$.
\end{lemma}
\begin{proof}
    Pick any $V\in \Pi_f^{\circ}$, we have $f = f_V$ as in \eqref{eqn:f_V}. Therefore $\chi(f) = f_{\chi(V)}$ and thus $\chi(\Pi_{f}^\circ) = \Pi_{\chi(f)}^\circ$ for all $f\in \bound(k,n)$. Since $\chi$ preserves the partial ordering on $\bound(k,n)$, we get \[\chi(\Pi_{f}) = \bigsqcup_{f'\geq f} \chi(\Pi_{f'}^\circ) = \bigsqcup_{f'\geq f} \Pi_{\chi(f')}^\circ  = \Pi_{\chi(f)}.\]
\end{proof}

\begin{lem-def}
If $S$ is a cyclic interval,
we say $\Pi_{S\leq r}$ is a \textbf{basic positroid variety}. This is indeed an instance of a positroid variety. Set $\cJ_{S\leq r}\subset R(k,n)$ to be its defining ideal.
\end{lem-def}
\begin{proof}
    Let $S = [\alpha+1,\alpha+m]^\circ$. In the case where $\alpha = 0$, this is the Schubert variety $X_I$ where $I = [r]\cup [m+1,m+k-r]$.
    Since $\Pi_{S\leq r} = \chi^{\alpha}(\Pi_{[1,m]\leq r})$, $\Pi_{S\leq r}$ is a positroid variety.
\end{proof}
By Definition~\ref{def:essential-rank} and Proposition~\ref{prop:basic-intersection}, each positroid variety $\Pi_f$ is the (scheme-theoretic) intersection of basic positroid varieties:
\begin{equation}\label{eqn:basic-int}
    \Pi_f = \bigcap_{(i,j)\in ess(f)} \Pi_{[i,j]^\circ\leq r(f)_{i,j}}
\end{equation}

\begin{prop}\label{prop:uv}
    Let $0\le r<k$ and $S = [\alpha+1,\alpha+m]^\circ$ for some $\alpha\in [0,n-1]$, $r< m<n$, and $n-m+r\ge k$.  $\Pi_{S\leq r}$ is the positroid variety $\Pi_v^u$ where
    \begin{enumerate}
        \item if $\alpha< n-m-(k-r)$, then $u = w_{[n-k+1,n]}$ and $v = w_{[k+m-r+\alpha]\setminus [r+\alpha+1,m+\alpha]}$.
        \item if $n-m-(k-r)\leq \alpha \leq n-m$, then $u = w_{[n-m-(k-r)+1,\alpha]\cup[m-r+\alpha+1,n]}$ and $v = w_{[n]\setminus [r+\alpha+1,m+\alpha]}$.
        \item if $n-m<\alpha\leq n-r$, then $u = w_{[\alpha-(k-r)+1,\alpha]\cup [n-r+1,n]}$ and $v = w_{[\alpha-(n-m)+k-r]\setminus [\alpha-(n-m)]}$.
        \item if $n-r<\alpha\leq n-1$, then $u = w_{[n-k+1,n]}$ and $v = w_{[\alpha-(n-m)+k-r]\setminus [\alpha-n+r+1,\alpha-n+m]}$.
    \end{enumerate}
    (We note that when $m\le r$ the rank condition is trivial, and $n-m+r\ge k$ is a consequence of the underlying $k\times n$ matrix being full rank.)
\end{prop}

\begin{remark}
\label{rmk:uv-des}
  In Proposition~\ref{prop:uv}, the permutation $u$ has a unique descent at $k$ and $v$ has at most one descent. The descent of $v$ is at $k$ if and only if $\alpha=0$, and $v=\id$ if and only if $\alpha+m=n$. Otherwise $v$ has a descent at $l\neq k$, where $l>k$ if $a+m<n$ and $l<k$ if $\alpha+m>n$.
\end{remark}

\begin{example}
    \label{ex:uv}
    Fix $n=11$, $k=5$, $m=6$, $r=3$. Then $n-m-(k-r)=3$, $n-m=5$, $n-r=8$.
    \begin{enumerate}
        \item If $\alpha=2$, then $u=789\underline{10}\,\underline{11}123456$ and $v=123459\underline{10}678\underline{11}$.
        \item If $\alpha=5$, then $u=459\underline{10}\,\underline{11}123678$ and $v=\id$.
        \item If $\alpha=7$, then $u=679\underline{10}\,\underline{11}123458$ and $v=341256789\underline{10}\,\underline{11}$.
        \item If $\alpha=9$, then $u=789\underline{10}\,\underline{11}123456$ and $v=156234789\underline{10}\,\underline{11}$.
    \end{enumerate}
\end{example}

\begin{proof}[Proof of Proposition~\ref{prop:uv}] 
We start with the case where $\alpha = 0$. It follows from \cite{K14interval}*{Proposition~2.1} that 
$\Pi_{[1,m]\leq r} = \Pi_{f_0}$ where $f_0$ is the bounded affine permutation such that
\begin{equation}\label{eqn:f0}
    f_0(i) = \begin{cases}
    r+i & \text{ if }1\leq i\leq m-r\\
    k+i & \text{ if }m-r<i\leq n-k+r\\
    i+m+k-r & \text{ if }n-k+r<i\leq n
\end{cases}.
\end{equation}
More generally, by Lemma~\ref{lemma:chi}, for $S = [\alpha+1,\alpha+m]^{\circ}$, we have $\Pi_{S\leq r} = \Pi_{f_\alpha}$ where
\begin{equation}\label{eqn:May11aaa}
    f_\alpha(i) = f_0(i-\alpha)+\alpha.
\end{equation}
Set $I_\alpha = \{i\in [n]:f_\alpha(i)>n\}$, combining  \eqref{eqn:f0} and \eqref{eqn:May11aaa}, we get 
\begin{equation}\label{eqn:May11bbb}
    I_\alpha = \begin{cases}
        [n-k+1,n] & \text{ if }0\leq \alpha\leq n-m-(k-r)\\
        [n-m-(k-r),\alpha]\cup [m-r+\alpha+1,n] & \text{ if }n-m-(k-r)<\alpha\leq n-m\\
        [\alpha-(k-r)+1,\alpha]\cup[n-r+1,n] &\text{ if }n-m<\alpha\leq n-r\\
        [n-k+1,n]&\text{ if }n-r<\alpha\leq n-1
    \end{cases}.
\end{equation}
By Lemma~\ref{lemma:wtouv}, we get the corresponding $(u,v)$ as desired in Proposition~\ref{prop:uv}.
\end{proof}

\subsection{\texorpdfstring{$\gr(k,n)$}{Gr(k,n)} and \texorpdfstring{$\gr(n-k,n)$}{Gr(k,n)}}
For a positroid variety $\Pi_f\subset \gr(k,n)$, let 
$\displaystyle \k[\Pi_f] = \frac{R(k,n)}{\cJ_f}$ be the homogeneous coordinate ring of $\Pi_f$.

\begin{prop}\cite{KLS13juggling}*{Theorem~5.15}\label{prop:vanishplucker}
The ideal $\cJ_f$ 
is generated by the set of straightening relations as in Definition~\ref{def:Pluckerrelations} and the set of Pl\"ucker coordinates $[\aa]$ that vanish on $\Pi_f$.
\end{prop}

As an easy consequence, we have the following.
\begin{cor}\label{cor:vanishplucker}
    The defining ideal of a basic positroid variety $\Pi_{S\leq r}$ is
    \[\cJ_{S\leq r} = \langle [\aa]:\abs{\aa\cap S}\geq r+1\rangle.\]
\end{cor}

\begin{construction}\label{const: complement-interval}
Let $S = [\alpha+1,\alpha+m]^\circ$ be a cyclic interval and $r<m$. 
Define the following:
\begin{itemize}
    \item $S^\vee \coloneqq [\alpha+m+1,\alpha]^\circ$.
    \item $r^\vee \coloneqq n-k-m+r$
    \item For any monomial $\mm = \prod_{i = 1}^d [\mathbf{a}^{(i)}]\in R(k,n)$, define $\mm^\vee$ to be $\prod_{i = 1}^d [\mathbf{a}^{(i)\vee}]\in R(n-k,n)$, where $\mathbf{a}^{(i)\vee} \coloneqq [n]\setminus \mathbf{a}^{(i)}$ for all $i\in [d]$.
    \item For any polynomial $g \in R(k,n)$, define $g^\vee \in R(n-k,n)$ to be the polynomial obtained from $g$ by applying $\vee$ to each monomial summand.
\end{itemize}
There is a stratification-preserving isomorphism $\varphi$ of varieties between $\gr(k,n)$ and $\gr(n-k,n)$ which preserves the positroid stratification. In coordinates, the isomorphism is written as $\varphi:[\mathbf{a}^\vee]\mapsto [\mathbf{a}]$. It is easy to check that the straightening relations are preserved under this mapping of coordinates.
\end{construction}

\begin{prop}
    \label{prop:basic-dual}
    Adopt notation and hypotheses as in \Cref{const: complement-interval}. Then
    $$\varphi(\Pi_{S\le r})=\Pi_{S^\vee\le r^\vee}.$$
\end{prop}

\begin{proof}
    The ideal $\cJ_{S\le r}$ is generated by the straightening relations for $\gr(k,n)$ as well as vanishing of coordinates $[\aa]$ such that $|\aa\cap S|\ge r+1$. Therefore,
    \begin{align*}
        & n-|\aa\cap S| \le n-r-1 \\
        \implies & |(\aa\cap S)^\vee |\le n-r-1 \\
        \implies & |\aa^\vee \cup S^\vee |\le n-r-1 \\
        \implies &|\aa^\vee| + |S^\vee| - |\aa^\vee \cap S^\vee| \le n-r-1 \\
        \implies &n-k-m+r+1\le |\aa^\vee \cap S^\vee|.
    \end{align*}
    It follows that $\varphi(\Pi_{S\le r})=\Pi_{S^\vee\le r^\vee}$.
\end{proof}

%%%%%%%%%%%%%%%%%%%%%%%%%%%%%%%%%%%%%%
\section{Initial ideals and standard monomials for positroid varieties}\label{section: initial-ideal}
%%%%%%%%%%%%%%%%%%%%%%%%%%%%%%%%%%%%%%
The goal of this section is to describe the standard monomials for any positroid variety $\Pi_f$. Let $B(k,n,d)$ denote the set of rectangular semistandard tableaux of shape $k\times d$ with entries $\le n$. Recall from Section~\ref{section: hodge-algebra-background} that the set $B(k,n,d)$ can be identified with standard monomials for $\gr(k,n)$ under the Hodge degeneration. We write $\mm\in B(k,n,d)$ both for the monomial and its tableau. For a positroid variety $\Pi_f$, let\[B_f(k,n,d):=\{\mm\in B(k,n,d): \mm \not\in \inn_\omega(\cJ_f) \}.\]
Then $B_f(k,n,d)$ is the set of degree-$d$ standard monomials for $\Pi_f$, and it forms a basis of $\k[\Pi_f]_d$.

Our approach is to first consider the case where $\Pi_f = \Pi_{S\leq r}$ is a basic positroid variety. 
Theorem~\ref{thm:basic-std-mon} gives a characterization of the standard monomials in terms of their corresponding semistandard tableaux. This is done by combining the case where $S$ is an interval (Theorem~\ref{thm:initial-ideal}) and the case where $S^\vee$ is an interval (Proposition~\ref{prop:init-dual}). 
We then show that the set of standard monomials for an arbitrary positroid variety is the intersection of the set of standard monomials for some basic positroid varieties (Theorem~\ref{thm:std-mon-intersect}).

\subsection{The standard monomials for a basic positroid variety} 

Let $\mm = \prod_{i = 1}^d [\aa^{(i)}]$ be a standard monomial of $\gr(k,n)$ where the Pl\"ucker variables $\aa^{(1)} \leq \aa^{(2)} \leq \dots \leq \aa^{(d)}$ are sorted lexicographically. 

\begin{definition}\label{def:antidiag}
Fix an interval \(S=[\alpha+1,\alpha+m]\subseteq[n]\) and an integer \(r<m\).
Let \(\mm\in B(k,n,d)\) be a \(k\times d\) semistandard tableau with entries
\(a^{(i)}_j\) in row \(j\) and column \(i\).
A \textbf{generalized antidiagonal} of \(\mm\) for the rank condition \(S\le r\) is a choice of
indices
\[
1\le \rho_1<\cdots<\rho_\ell\le k,
\qquad
1\le \sigma_1\le \cdots \le \sigma_\ell\le d,
\]
such that the selected entries
\(
a^{(\sigma_\ell)}_{\rho_1}, a^{(\sigma_{\ell-1})}_{\rho_2},\dots,a^{(\sigma_1)}_{\rho_\ell}
\)
lie in \(S\) and are strictly increasing:
\[
\alpha+1\le a^{(\sigma_\ell)}_{\rho_1}<\cdots<a^{(\sigma_1)}_{\rho_\ell}\le \alpha+m.
\]
We say \(\mm\) \textbf{minimally contains} a generalized antidiagonal for \(S\le r\) if
for every column \(t\in[d]\), the tableau obtained by deleting column \(t\)
contains no generalized antidiagonal for \(S\le r\).
\end{definition}

Alternatively, take the reading word of the semistandard tableaux $\mm$ by reading the columns from right to left, and within each column top to bottom (i.e., traditional Chinese reading order.) 
Then $\mm$ contains a generalized antidiagonal for $S\leq r$ if and only if there is a strictly increasing subsequence of size $r+1$ with values in $S$ in the reading word\footnote{We thank Alex Yong for this description.}.

We start with the characterization of standard monomials for positroid varieties defined by a single interval rank condition.

\begin{theorem}\label{thm:initial-ideal}
Fix an interval \(S=[\alpha+1,\alpha+m]\subseteq[n]\) and \(r<m\), and consider the basic positroid variety
\(\Pi_{S\le r}\subset \gr(k,n)\) with defining ideal \(\cJ_{S\le r}\subset R(k,n)\).
A degree-\(d\) standard monomial \(\mm\in B(k,n,d)\) lies in \(B_{S\le r}(k,n,d)\)
(i.e.\ \(\mm\notin \inn_\omega(\cJ_{S\le r})\))
if and only if the tableau \(\mm\) contains no generalized antidiagonal of length \(r+1\)
with entries in \(S\) (Definition~\ref{def:antidiag}).
\end{theorem}

\begin{example}
Let $S = [4,7]$ and $r = 2$ in $\gr(5,10)$. The following tableaux do not correspond to standard monomials of $\Pi_{S\le r}$. Some minimal generators of $\inn_\omega(\cJ_{S\leq r})$ include
$$
\ytableausetup{boxsize = 1.2em}
\yshort{1,{\rr{4}},{\rr{5}},{\rr{7}},9} \qquad
\yshort{12,2{\rr{4}}, 38, {\rr{6}}9,{\rr{7}}{10}}\qquad
\yshort{12,2{\rr{4}},3{\rr{5}},{\rr{6}}8,{\rr{7}}9} \qquad 
\yshort{
122,23{\rr{5}},3{\rr{6}}7,{\rr{7}}79,8{10}{10}} \; .
$$
In the second tableau above, the unique sequence forming a generalized antidiagonal of length $r+1 = 3$ is
$$
(a^{(2)}_2, a^{(1)}_4, a^{(1)}_5) = (4,6,7).
$$
The following monomials do \textit{not} contain any generalized antidiagonal of size $3$ with entries in the interval $S$, and thus are standard monomials for $\Pi_{S\le r}$:
$$
\yshort{16,27,38,49,5{10}} \qquad 
\yshort{
1111,2222,4567,7778,8889
}
$$
\end{example}

Our main tool for proving \Cref{thm:initial-ideal} will be the following result from \cite{KLS2014projections}, which allows us to understand standard monomials for $\Pi_f$ in terms of chains of permutations in Bruhat order.

\begin{notation}\label{not: oneline-not-chain}
For a permutation $v\in \symm_n$ we will sometimes write $v$ in one-line notation $v = (v(1)\ v(2)\cdots v(n))$. For any positive integer $p\leq n$, set $v([p]):= \{v(1),\cdots, v(p)\}$. 
\end{notation}

\begin{prop}\cite{KLS2014projections}*{Theorem 7.1}\label{prop:chainlift}
Let \(\Pi_{v}^{u}\subset \gr(k,n)\) be the positroid variety associated to a $k$-Bruhat interval \([v,u]_k\),
and let \(\cJ_{[v,u]}\) denote its defining ideal under the Pl\"ucker embedding.
Fix \(\mathbf{a}^{(1)},\dots,\mathbf{a}^{(d)}\in\binom{[n]}{k}\) with
\(\mathbf{a}^{(1)}\le \cdots \le \mathbf{a}^{(d)}\).
Then \(\prod_{i=1}^d[\mathbf a^{(i)}]\notin \inn_\omega(\cJ_{[v,u]})\) if and only if
there exists a chain \(\{v^{(1)},\dots,v^{(d)}\}\subset \symm_n\) such that
\begin{equation}\label{eqn:chain}
v\le v^{(1)}\le \cdots \le v^{(d)}\le u
\quad\text{and}\quad
v^{(i)}([k])=\mathbf a^{(i)}\ \text{for all }i.
\end{equation}
\end{prop}

The following proposition justifies why we restrict our attention to basic positroid varieties with $S = [\alpha+1,\alpha+m]$ and $r<m$; that is, we do not need to consider separately the case where the interval $S$ ``wraps around'' when we prove \Cref{thm:initial-ideal}.

\begin{prop}\label{prop:init-dual}
Adopt the notation of Construction~\ref{const: complement-interval}, so that
\(\varphi:\gr(k,n)\to \gr(n-k,n)\) identifies Pl\"ucker coordinates via complements and sends
\(\Pi_{S\le r}\) to \(\Pi_{S^\vee\le r^\vee}\).
Then a monomial \(\mm\in R(k,n)\) lies in \(\inn_\omega(\cJ_{S\le r})\) if and only if its
complement monomial \(\mm^\vee\in R(n-k,n)\) lies in \(\inn_\omega(\cJ_{S^\vee\le r^\vee})\).
\end{prop}

\begin{proof}[Proof of \Cref{prop:init-dual}]
    Let $[v,u]$ be the Grassmann interval such that $\Pi_{S\leq r} = \Pi_v^u$ as in Proposition~\ref{prop:uv}. Note that since $\alpha+1\leq n < \alpha+m$, we are in case $(3)$ and $(4)$ of Proposition~\ref{prop:uv}. Let $z\in \symm_k\times \symm_{n-k}\subset \symm_n$ be such that $vz$ is the maximal representative in the coset $v\cdot \symm_k\times \symm_{n-k}$.
    Then $[vz,uz]_k\sim [v,u]_k$ and thus 
    \[\prod_{i = 1}^d [\aa^{(i)}]\in \inn_\omega(\cJ_{[v,u]}) \iff \prod_{i = 1}^d [\aa^{(i)}]\in \inn_\omega(\cJ_{[vz,uz]}).\]
    Let $w_0$ be the longest permutation in $\symm_n$. Since multiplying by $w_0$ is an order reversing automorphism on $\symm_n$ with respect to strong Bruhat order, 
    by \eqref{eqn:chain}, we also have 
    \[\prod_{i = 1}^d [\aa^{(i)}]\in \inn_\omega (\cJ_{[vz,uz]}) \iff \prod_{i = 1}^d [\aa^{(i)\vee}]\in \inn_\omega(\cJ_{[uzw_0,vzw_0]}).\]
    In particular, $[uzw_0,vzw_0]$ is a Grassmann interval.
    Therefore we only need to verify that $\Pi_{[uzw_0,vzw_0]} = \Pi_{S^\vee\leq r^\vee}$ through Proposition~\ref{prop:uv}. 
    Here we set $[\alpha'+1,\alpha'+m'] := S^\vee$ and $\Pi_{S^\vee\leq r^\vee}\subset \gr(k',n)$, then 
    \begin{equation}\label{eqn:S'}
        \alpha' = \alpha+m-n, r^\vee = n-k-m+r, m' = n-m\text{ and }k' = n-k.
    \end{equation}
    We now divide into four cases:

    \noindent\textbf{Case \RNum{1}} ($n-m<\alpha\leq n-r$ and $r\leq \alpha-n+m$): Here we are in case (3) of Proposition~\ref{prop:uv} and thus $u = w_{[\alpha-(k-r)+1,\alpha]\cup [n-r+1,n]}$ and $v = w_{[\alpha-(n-m)+k-r]\setminus [\alpha-n+m]}$. Since $r\leq \alpha-n+m$, we have 
    \[v(1)\cdots v(k) = (\alpha-n+m+1)\cdots (\alpha-n+m+k-r)\ 1\cdots r.\]
    and 
    \[z = (k-r)\cdots 1\ k\cdots (k-r+1)\ n\cdots (k+1).\]
    Since $r\leq \alpha-n+m<m$, we have $n-m'-k'+r^\vee\leq \alpha'<n-m'$ by \eqref{eqn:S'} and $\Pi_{S^\vee\leq r^\vee}$ falls into case (2) of Proposition~\ref{prop:uv}.
    We then have 
    \[vzw_0 = w_{[r+1,n]\setminus [\alpha-n+m+1,\alpha-n+m+k-r]} = w_{[n-m'-k'+r^\vee+1,\alpha']\cup[m'-r^\vee+\alpha'+1,n]}\]
    and
    \[uzw_0 = w_{[n]\setminus [\alpha-k+r+1,\alpha]} = w_{[n]\setminus [r^\vee+\alpha'+1,m'+\alpha']}\]
    as desired.

    \noindent\textbf{Case \RNum{2}.} ($n-m<\alpha\leq n-r$ and $r>\alpha-n+m$): Here we are still in case (3) of Proposition~\ref{prop:uv} as in Case \RNum{1}. Since $r>\alpha-n+m$, we have
    \[v(1)\cdots v(k) = (\alpha-n+m+1)\cdots (\alpha-n+m+k-r)\ 1\cdots (\alpha-n+m) \ (\alpha-n+m+k-r+1) \cdots k,\]
    and 
    \[z = k\cdots (\alpha-n+m+k-r+1)\ (k-r)\cdots 1\ (\alpha-n+m+k-r)\cdots (k-r+1)\ n\cdots (k+1).\]
    Since $r>\alpha-n+m$, we have $0\leq \alpha' < n-m'-k'+r^\vee$ and $\Pi_{S^\vee\leq r^\vee}$ falls into case (1) of Proposition~\ref{prop:uv}. We have
    \[vzw_0 = w_{[k+1,n]} = w_{[n-k'+1,n]}\]
    and
    \[uzw_0 = w_{[\alpha+m-r]\setminus [\alpha-k+r+1,\alpha]} = w_{[k'+m'-r^\vee+\alpha']\setminus [r^\vee+\alpha'+1,m'+\alpha']}\]
    as desired.
    
    \noindent\textbf{Case \RNum{3}.} ($n-r<\alpha\leq n-1$ and $r\leq \alpha-n+m$): Here we are in case (4) of Proposition~\ref{prop:uv} and thus $u = w_{[n-k+1,n]}$ and $v = w_{[\alpha-n+m+k-r]\setminus[\alpha-n+r+1,\alpha-n+m]}$. Since $r\leq \alpha-n+m$, $\Pi_{S^\vee\leq r^\vee}$ falls into case (2) of Proposition~\ref{prop:uv}. A similar computation as in the previous two cases gives us
    \[vzw_0 = w_{[r+1,\alpha-n+m]\cup [\alpha-n+m+k-r+1,n]} = w_{[n-m'-k'+r^\vee+1,\alpha']\cup[m'-r^\vee+\alpha'+1,n]}\]
    and 
    \[uzw_0 = w_{[n]\setminus [\alpha-k+r+1,\alpha]} = w_{[n]\setminus [r^\vee+\alpha'+1,m'+\alpha']}\]
    as desired.
    
    \noindent\textbf{Case \RNum{4}.} ($n-r<\alpha\leq n-1$ and $r> \alpha-n+m$): Here $\Pi_{S\leq r}$ is in case (4) of Proposition~\ref{prop:uv} and $\Pi_{S^\vee\leq r^\vee}$ is in case (1). A similar computation yields
    \[vzw_0 = w_{[k+1,n]} = w_{[n-k'+1,n]}\]
    and 
    \[uzw_0 = w_{[\alpha+m-r]\setminus [\alpha-k+r+1,\alpha]} = w_{[k'+m'-r^\vee+\alpha']\setminus [r^\vee+\alpha'+1,m'+\alpha']}\]
    as desired.
    
    Since $n-m<\alpha\leq n-1$, we exhaust all possibilities and thus complete the proof.
\end{proof}

 Combining Theorem~\ref{thm:initial-ideal} and Proposition~\ref{prop:init-dual}, we obtain a characterization of standard monomials for basic positroid varieties.
\begin{theorem}\label{thm:basic-std-mon}
Fix a cyclic interval $S = [\alpha+1,\alpha+m]^\circ$ and $r\in \mathbb{N}$.
Let \(S^\vee=[n]\setminus S\) and \(r^\vee:=n-k-m+r\) as in Construction~\ref{const: complement-interval}.
Then $\mm = \prod_{i = 1}^{d}[\aa^{(i)}]$ is a standard monomial for $\Pi_{S\leq r}$ if and only if $\mm\in B(k,n,d)$ and 
    \begin{enumerate}
        \item If $\alpha+m\leq n$, then $\mm$ does not contain a generalized antidiagonal of size $r+1$ with entries in $S$;
        \item If $\alpha+m>n$, then $\mm^\vee$ does not contain a generalized antidiagonal of size $r^\vee+1$ with entries in $S^\vee$.
    \end{enumerate}
\end{theorem}

\begin{proof}
If $\alpha+m\le n$, then $S=[\alpha+1,\alpha+m]\subseteq [n]$ is an ordinary interval, so the claim is exactly \Cref{thm:initial-ideal}.

If $\alpha+m>n$, then $S=[\alpha+1,\alpha+m]^\circ$ wraps around. In this case, apply \Cref{prop:init-dual}: a monomial $\mm$ lies in $\inn_\omega(\cJ_{S\le r})$ if and only if $\mm^\vee$ lies in $\inn_\omega(\cJ_{S^\vee\le r^\vee})$. By construction, $S^\vee$ is a (non-wrapping) interval in $[n]$, so \Cref{thm:initial-ideal} applies to $\Pi_{S^\vee\le r^\vee}$ and shows that $\mm^\vee\notin \inn_\omega(\cJ_{S^\vee\le r^\vee})$ if and only if $\mm^\vee$ contains no generalized antidiagonal of size $r^\vee+1$ with entries in $S^\vee$. Translating back via \Cref{prop:init-dual} gives the stated criterion for $\mm$.
\end{proof}

\begin{example}
Fix $\gr(2,6)$ and let $S = [6,1]^\circ$ and $r = 1$. 
By \Cref{const: complement-interval}, we have $S^\vee = [6]\setminus S = [2,5]$ and $r^\vee = 6-2-2+1 = 3$.
Given a monomial $\mm$ in $\inn_\omega(\cJ_{S\leq r})\subset R(2,6)$, we may obtain the corresponding monomial in $\inn_\omega(\cJ_{S^\vee\leq r^\vee}) \subset R(4,6)$ by taking complements of each of the Pl\"{u}cker coordinates dividing $\mm$.
For instance,
\begin{gather*}
\yshort{1,6} \xleftrightarrow[]{\text{complement}} \yshort{{\rr{2}},{\rr{3}},{\rr{4}}, {\rr{5}}}
\;, \quad \yshort{23,65} \xleftrightarrow[]{\text{complement}} \yshort{11,32,44,56} = \yshort{11,23,44,65} \; , \\
\end{gather*}
\begin{gather*}
\yshort{1234,3456}\xleftrightarrow[]{\text{complement}} \yshort{2111,4322,5543,6665} = \yshort{111{\rr{2}},22{\rr{3}}4,3{\rr{4}}55,{\rr{5}}666}.
\end{gather*}
\end{example}

\begin{notation}\label{not: ordered-sets-leq}
For ordered sets $I = \{i_1<i_2<\cdots<i_p\}$ and $J = \{j_1<j_2<\cdots<j_p\} \subset[n]$, we write $I\leq J$ if $i_a\leq j_a$ for all $a\in [p]$.
\end{notation}

The following classical result on strong Bruhat order is a key ingredient for proving Theorem~\ref{thm:initial-ideal}.

\begin{lemma}[\cite{BB2005}, Corollary 2.6.2]\label{lemma:bruhat}
    For $u,v\in \symm_n$, the permutation $v$ is less than or equal to $u$ in Bruhat order if and only if for all $\ell\in \mathrm{Des}(v)$,
    \[\{v(1),\cdots,v(\ell)\}\leq \{u(1),\cdots,u(\ell)\}.\]
    In particular, if $u = w_I, v = w_J$ are both $k-$Grassmannian permutations, then 
    \[u\leq v \iff I\leq J.\]
\end{lemma}

We are now ready to prove the main theorem of this section.

\begin{proof}[Proof of Theorem~\ref{thm:initial-ideal}]
    
    \noindent $(\implies):$ We first show that if the monomial $\mm = \prod_{i = 1}^d [\mathbf{a}^{(i)}]\in B(k,n,d)$ contains a generalized antidiagonal for $S\leq r$, then $\mm$ does not lift to a chain in Bruhat order as in \Cref{prop:chainlift} and therefore is not a standard monomial for $\Pi_{S\leq r}$.
Let 
$$\cc = (a^{(\sigma_{r+1})}_{\rho_1}, a^{(\sigma_{r})}_{\rho_2}, \dots , a^{(\sigma_1)}_{\rho_{r+1}}) = (c_1,\dots, c_{r+1})$$
be a generalized antidiagonal of length $r+1$ in $\mm$.
    
Write $R = \sqcup_{i = 1}^d R_i$ where 
$$R_i = \{\rho_{j} : a_j^{(i)} = c_h \text{ for some } h\in [r+1]\}.$$
In words, $R_i$ is the set of row indices of entries in the generalized antidiagonal whose column indices are $i$.
    
    Suppose, seeking contradiction, that there exist permutations $v^{(1)},\cdots,v^{(d)}$ satisfying \eqref{eqn:chain}. Set $v^{(0)} = v$ and $v^{(d+1)} = u$. 
    We first consider the case where $\alpha \le n-m-(k-r)$. By Proposition~\ref{prop:uv}, 
    \[u=w_{[n-k+1,n]}\text{ and }v=w_{[k+m-r+\alpha]\setminus [r+\alpha+1,m+\alpha]}.\]
    Here $v$ has a unique descent at $k+\alpha$.
    
    Set $\bb^{(i)}=v^{(i)}([k+\alpha]) = \{v^{(i)}(1),\cdots, v^{(i)}(k+\alpha)\}$ for $0\le i\le d+1$. Then $u([k+\alpha]) = \bb^{(d+1)}=[\alpha]\cup[n-k+1,n]$ and $v([k+\alpha]) = \bb^{(0)}=[r+\alpha]\cup [m+\alpha+1,k+m-r+\alpha]$. Notice that $b^{(0)}_{r+\alpha+1}=m+\alpha+1$ and $b^{(d+1)}_\alpha=\alpha$. Furthermore, $\aa^{(i)}\subset \bb^{(i)} $ for all $1\le i\le d$ and $\bb^{(i)}\leq \bb^{(i+1)}$ (where here $<$ is the order from \Cref{not: ordered-sets-leq}).   

    Set $s_i = \sum_{j = i}^d |R_j|$. Since $\bb^{(d)}\leq \bb^{(d+1)}$, we have $\alpha\leq b^{(d)}_\alpha \leq b^{(d+1)}_\alpha = \alpha$. Thus $b^{(d)}_\alpha=\alpha$. Since $b^{(d)} \supset \{a^{(d)}_{\rho}:\rho\in R_d\}$ and $a^{(d)}_{\rho}> \alpha = b^{(d)}_\alpha$ for all $\rho\in R_d$, we have
    \[b^{(d)}_{\alpha+s_d} \leq a^{(d)}_{\rho_{s_d}}.\]
    Similarly, since $\bb^{(d-1)}\le \bb^{(d)}$, we have 
    $b^{(d-1)}_{\alpha+s_d}\le b^{(d)}_{\alpha+s_d}\le a^{(d)}_{\rho_{s_d}}<a_{\rho_{s_{d}+1}}^{(d-1)}$, and thus $b^{(d-1)}_{\alpha+s_d} < a^{(d-1)}_\rho$ for all $\rho\in R_{d-1}$.
    Since $\bb^{(d-1)}\supset \{a^{(d-1)}_{j}:j\in R_{d-1}\}$, we get
    \[b^{(d)}_{\alpha+s_{d-1}} \leq a^{(d)}_{\rho_{s_{d-1}}}.\]
    Continuing this pattern, we eventually have \[ b^{(1)}_{\alpha+r+1} = b^{(1)}_{\alpha+s_1} \leq a^{(1)}_{\rho_{s_1}} = a^{(1)}_{\rho_{r+1}} \leq \alpha+m.\] 
    This contradicts $b^{(0)}_{\alpha+r+1}=\alpha+m+1$ as $b^{(0)}_{\alpha+r+1} \leq b^{(1)}_{\alpha+r+1}$.

    Now consider the case where $n-m-(k-r)<\alpha\le n-m$. By Proposition~\ref{prop:uv}, $$u = w_{[n-m-(k-r)+1,\alpha]\cup[m-r+\alpha+1,n]}\text{ and }v = w_{[n]\setminus [r+\alpha+1,m+\alpha]}.$$ When $\alpha+m<n$, $v$ has a unique descent at $r+n-m$. 

    Set $\bb^{(i)}=\{v^{(i)}(1),\cdots, v^{(i)}(r+n-m)\}$ for $0\le i\le d+1$. Then $\bb^{(d+1)}=[\alpha]\cup[m-r+a+1,n]$ and $\bb^{(0)}=[r+\alpha]\cup[m+\alpha+1,n]$. The rest of the argument is identical to the previous case.

    \noindent $(\impliedby):$ We now show that if $\mm = \prod_{i = 1}^d [\mathbf{a}^{(i)}]\in B(k,n,d)$ does not contain a generalized antidiagonal for $S\leq r$, then  $\prod_{i = 1}^d [\mathbf{a}^{(i)}]\notin \inn_\omega (\cJ_{S\le r})$. 
    Notice that each Pl\"{u}cker variable dividing $\mm$ does not contain a generalized antidiagonal; that is, $(\aa^{(i)}\cap S) \leq r$  for all $i\in [d]$.
    In particular, $[\aa^{(i)}]\notin \inn_\omega(\cJ_{S\le r})$ and, by Proposition~\ref{prop:chainlift}, 
    $v([k])\leq \aa^{(i)}\leq u([k])$ for all $i\in [d]$.
     
     Suppose first that $\alpha\le n-m-(k-r)$; that is, we are in case $(1)$ of Proposition~\ref{prop:uv}. Here we have $\mathrm{Des}(u) = \{k\}$ and $\mathrm{Des}(v) = \{k+\alpha\}$.
    By Proposition~\ref{prop:chainlift}, we will show that there exists permutations $v^{(1)}\dots v^{(d)}$ satisfying conditions in \eqref{eqn:chain}. 
    By Lemma~\ref{lemma:bruhat}, it suffices to construct $\{\mathbf{b}^{(i)}:i\in [d]\}$ such that
    \begin{align}\label{eqn:defb}
    \begin{split}
        \mathbf{a}^{(i)}\subset \mathbf{b}^{(i)} \subset [n]&,\  |\mathbf{b}^{(i)}| = k+\alpha\text{ and }\\
        [k+m-r+\alpha]\setminus[r+\alpha+1,m+a] &\leq \mathbf{b}^{(1)} \leq \dots \leq \mathbf{b}^{(d)} \leq [\alpha]\cup[n-k+1,n].
     \end{split}
    \end{align}
    Indeed, for each $1\le i\le d$, we can construct $v^{(i)}$ from $\aa^{(i)}$ and $\bb^{(i)}$ by setting 
    \[\{v^{(i)}(1)<\cdots < v^{(i)}(k)\}=\aa^{(i)},\]
     \[\{v^{(i)}(k+1)<\cdots<v^{(i)}(k+\alpha)\}=\bb^{(i)}\setminus\aa^{(i)},\text{ and}\]

    Set $\bb^{(d+1)}:= u([\alpha+k]) = [\alpha]\cup[n-k+1,n]$. For each $1\le i\le d$, we iteratively construct $\bb^{(i)}$ from $\aa^{(i)}$ and $\bb^{(i+1)}$ as follows:
    \begin{enumerate}
        \item Initialize $p=1$;
        \item For each $1\le j \le k+\alpha$:
        \begin{itemize}
            \item[-] If $b_j^{(i+1)} < a_p^{(i)}$, set $b_j^{(i)}:=b_{j}^{(i+1)}$
            \item[-] Otherwise, set $b_{j}^{(i)}:=a^{(i)}_p$, and increment $p$ by 1.
        \end{itemize}
    \end{enumerate}
    Finally, set $\bb^{(0)}:=[r+\alpha]\cup [\alpha+m+1,k+m-r+\alpha]$.

    It is easy to see by construction that for all $1\le i\le d$ we have $\bb^{(i)}\leq \bb^{(i+1)}$ and 
    $\aa^{(i)}\subset \bb^{(i)}$.
    By \eqref{eqn:defb}, we are left to show that $\bb^{(1)}\geq \bb^{(0)} \coloneqq v([k+\alpha])$.
    Notice that since $[\alpha]\subset \bb^{(d+1)}$, by construction, we have $[\alpha]\subset \bb^{(0)},\bb^{(1)}$. Since $\bb^{(0)} = [r+\alpha]\cup[\alpha+m+1,k+m-r+\alpha]$, it is enough to show that 
    \[|\bb^{(1)}\cap S |\leq r.\]

     For $1\le i\le d, 1\le j\le k$, define
\[\antidiag(i,j):=\Biggl\{ ((\rho_1,\sigma_h),\cdots, (\rho_h,\sigma_1)):\begin{array}{c}
    1\le \rho_1<\cdots < \rho_h\le j, \\
    i\le \sigma_1\leq \cdots\leq \sigma_h\le d,\\
    \alpha+1\le a^{(\rho_1)}_{\sigma_h}<\cdots<a^{(\rho_h)}_{\sigma_1}\le \alpha+m
\end{array}
         \Biggl\}
\]
    to be the set of generalized antidiagonals in the first $j$ rows and last $d-i+1$ columns.

\begin{notation}\label{not:lad}
For \(1\le i\le d\) and \(1\le j\le k\), let \(\antidiag(i,j)\) denote the set of generalized antidiagonals
contained in the first \(j\) rows and the last \(d-i+1\) columns of the tableau \(\mm\):
\[
\antidiag(i,j):=\Bigl\{((\rho_1,\sigma_h),\dots,(\rho_h,\sigma_1)):\ 
\begin{array}{c}
1\le \rho_1<\cdots<\rho_h\le j,\quad i\le \sigma_1\le\cdots\le \sigma_h\le d,\\
\alpha+1\le a^{(\sigma_h)}_{\rho_1}<\cdots<a^{(\sigma_1)}_{\rho_h}\le \alpha+m
\end{array}
\Bigr\}.
\]
Define \(\lad(i):=\max\{\length(D): D\in \antidiag(i,k)\}\).
\end{notation}

    \begin{claim}
    \label{claim:b-size}
    The following holds:
    \begin{equation}\label{eqn:lad}
        |\bb^{(1)}\cap S |\le \max \left \{|\bb^{(d+1)}\cap S| \; , \; \lad(1) \right\}.
    \end{equation}
    \end{claim}

    \begin{proof}
        Fix $h$ to be the largest number such that $b^{(1)}_{\alpha+h}\in S$. By definition,
        \begin{equation}\label{eqn:Apr25ccc}
            |\bb^{(1)}\cap S|=h.
        \end{equation}

        Observe that $b_{\alpha+h}^{(d+1)}\ge \alpha+m$, since otherwise we must have $b_{\alpha+h+1}^{(1)}\le b_{\alpha+h+1}^{(d+1)}=b_{\alpha+h}^{(d+1)}+1\le \alpha+m$, 
        contradicting the choice of $h$. We must check two cases: 

        \noindent\textbf{Case \RNum{1}.} Suppose  $b^{(d+1)}_{\alpha+h}=\alpha+m$. Then we argue that 
        \begin{equation}\label{eqn:Apr25aaa}
    |\bb^{(1)}\cap S |=|\bb^{(d+1)}\cap S|,
        \end{equation}
        and thus \eqref{eqn:lad} holds.
        Indeed, since $[\alpha]\subset \bb^{(d+1)}$ and $b^{(d+1)}_{\alpha+h} = \alpha+m$, we have $$\bb^{(d+1)}\cap S = \{b^{(d+1)}_{\alpha+1},\cdots,b^{(d+1)}_{\alpha+h}\}.$$ Since $[\alpha]\subset \bb^{(1)}$, by the  definition of $h$, $\bb^{(1)}\cap S =\{b^{(1)}_{\alpha+1},\cdots,b^{(1)}_{\alpha+h}\}$. Therefore both sides of \eqref{eqn:Apr25aaa} equal $h$.

        \noindent\textbf{Case \RNum{2}.} Suppose $b_{\alpha+h}^{(d+1)}>\alpha+m$. 
        Then for each $0\le j< h$, $b_{\alpha+h-j}^{(d+1)}>\alpha+m-j$. 
        There exists a largest $i_h\ge 1$ such that $b^{(i_h)}_{\alpha+h}=b^{(1)}_{\alpha+h}$. Since $b_{\alpha+h}^{(d+1)}>\alpha+m\geq b^{(i_h)}_{\alpha+h}$, it must be the case that $i_h\le d$.
        By construction, $b^{(i_h)}_{\alpha+h}=a^{(i_h)}_{j_h}$ for some $j_h$. We now inductively construct $(i_y,j_y)$ for $y = h-1,\cdots, 1$.    
        
        Suppose $(i_y,j_y)$ has been constructed with $y\geq 2$. Let $i_{y-1}$ be the largest index such that
        $b^{(i_{y-1})}_{\alpha+y-1}=b^{(i_y)}_{\alpha+y-1}$. In particular, we have 
        \[i_{y}\leq i_{y-1} \text{ and } b^{(i_{y-1})}_{\alpha+y-1}< b^{(i_y)}_{\alpha+y}.\]
        Since $b^{(i_h)}_{\alpha+h} = \bb^{(1)}_{\alpha+h}\leq \alpha+m$, we have $b^{(i_{y-1})}_{\alpha+y-1} \leq \alpha+m-h+y-1$. Since $b^{(d+1)}_{\alpha+y-1}>\alpha+m-h+y-1$, it must be the case that $i_{y-1}\le d$.       
        
        Since $b^{(i_{y-1}+1)}_{\alpha+y-1}>b^{(i_{y-1})}_{\alpha+y-1}$,
        by construction, $b^{(i_{y-1})}_{\alpha+y-1}=a^{(i_{y-1})}_{j_{y-1}}$ for some $j_{y-1}$. Thus we have
        $$a^{(i_{y-1})}_{j_{y-1}} = b^{(i_{y-1})}_{\alpha+y-1}=b^{(i_{y})}_{\alpha+y-1}<b^{(i_{y})}_{\alpha+y} = a^{(i_y)}_{j_y}.$$
        Since $i_{y-1}\ge i_y$, by semistandardness, we must have $j_{y-1}<j_y$. 
        Therefore the pair $(i_{y-1},j_{y-1})$ we constructed satisfy
        \begin{equation}\label{eqn:May12aaa}
             i_{y}\leq i_{y-1}\leq d,\ j_{y-1}<j_y,\text{ and }  b^{(i_{y-1})}_{\alpha+y-1}=a^{(i_{y-1})}_{j_{y-1}}<a^{(i_y)}_{j_y} = b^{(i_y)}_{\alpha+y}.
        \end{equation}

        Let $(i_1,j_1),\cdots, (i_h,j_h)$ be pairs as inductively constructed above. By \eqref{eqn:May12aaa}, $a_{j_1}^{(i_1)}\cdots a_{j_h}^{(i_h)}$ is a generalized antidiagonal of length $h$. By \eqref{eqn:Apr25ccc} and the definition of $\lad$,
        \begin{equation}\label{eqn:Apr25bbb}
            |\bb^{(1)}\cap S |=h\le \lad(1),
        \end{equation}
        which implies \eqref{eqn:lad}.

        Since \eqref{eqn:lad} holds in both cases, our proof is complete.
    \end{proof}

    Returning to the proof of \Cref{thm:initial-ideal}, by Claim~\ref{claim:b-size} it suffices to show that 
\(|\bb^{(d+1)}\cap S|\le r\) and \(\lad(1)\le r\).
    Since $\mathbf{b}^{(d+1)}\cap S=[n-k+1,n]\cap[\alpha+1,\alpha+m]$, we have
         \[|\bb^{(d+1)}\cap S|=\begin{cases}0 &\text{ if } n-k+1>\alpha+m \\
         \alpha+m-n+k & \text{ otherwise}\end{cases},\]
    and therefore $\abs{\bb^{(d+1)}\cap S} \le r$ by the assumption that $\alpha \le n-m-(k-r)$.
Moreover, our standing assumption is that 
\(\mm\) does not contain any generalized antidiagonal of length \(r+1\) with entries in \(S\).
Equivalently (by the definition of \(\lad(1)\)), we have \(\lad(1)\le r\).
Therefore \(\max\{|\bb^{(d+1)}\cap S|,\lad(1)\}\le r\), and Claim~\ref{claim:b-size} gives
\(|\bb^{(1)}\cap S|\le r\).

    We are left with case (2) of Proposition~\ref{prop:uv}, where $n-m-(k-r)\leq \alpha \leq n-m$, \[u = w_{[n-m-(k-r)+1,\alpha]\cup[m-r+\alpha+1,n]}\text{ and }v = w_{[n]\setminus [r+\alpha+1,m+\alpha]}.\] The only difference is that we now set \[\bb^{(d+1)}:=[\alpha]\cup [m-r+\alpha+1,n]\text{ and }\bb^{(0)}=[r+\alpha]\cup [m+\alpha+1, n].\]
    We construct $\bb^{(i)}$ for $1\le i\le d$ using the same algorithm as before; now $\bb^{(i)}$ has length $r+n-m\ge k$. The statement and proof of Claim~\ref{claim:b-size} go through the same as before. The rest of the argument also follows since now $|\bb^{(d+1)}\cap S|=r$.
\end{proof}
    
\begin{example}
    Let $n=11$, $k=5$, $\alpha=2$, $m=6$, $r=3$ as in Example~\ref{ex:uv}(1), so the basic interval positroid variety in consideration is $\Pi_{[3,8]\le 3}$. Recall that $u=789\underline{10}\,\underline{11}123456$ and $v=123459\underline{10}678\underline{11}$. The  monomial given by the following tableau

\[
\ytableausetup{centertableaux}
\yshort{
112{\rr{\textbf{4}}},
22{\rr{\textbf{5}}}7,
3488,
5{\rr{\textbf{6}}}99,
{\rr{\textbf{8}}}9{10}{10}}
\]
contains a generalized antidiagonal with entries 4,5,6,8 labeled and boldfaced in red, and hence is an element of $\inn_\omega(\cJ_{[3,8]\le 3})$. 

If we remove the second column the tableau becomes 
\[
\ytableausetup{centertableaux}
\yshort{
124,
257,
388,
599,
8{10}{10}}
\]
which no longer contains a generalized antidiagonal of length $4$. Let $\prod_{i=1}^{3}\aa^{(i)}$ be the monomial given by this tableau where $\aa^{(i)}$ corresponds to the $i$-th column.
We also show $\bb^{(0)},\cdots,\bb^{(4)}$ in the following tableau by following the algorithm given in the $(\implies)$ direction of the proof of Theorem~\ref{thm:initial-ideal}, where $\bb^{(i)}$ corresponds to the $i+1$-th column
\[
\ytableausetup{centertableaux}
\yshort{
11111,
22222,
33447,
45578,
58889,
9999{10},
{10}{10}{10}{10}{11}}
\]
Then $v^{(3)}=4789\underline{10}12356\underline{11}$, $v^{(2)}=2589\underline{10}14367\underline{11}$, and $v^{(1)}=123589\underline{10}467\underline{11}$.
\end{example}

\subsection{Standard monomials of arbitrary positroid varieties}\label{subsec: B(k,n,d)}

\begin{prop}
\label{prop:unique-min}
    Let $\mm:=\prod_{i=1}^d[\aa^{(i)}]\in B(k,n,d)$, where $\aa^{(1)}\le \cdots \le \aa^{(d)}$. Then there exists a unique minimal positroid variety $\Pi_f$ such that $\mm\in B_f(k,n,d)$. In other words, if $\Pi_{f'}$ is a positroid variety such that $\mm \in B_{f'}(k,n,d)$, then $\Pi_f\subseteq \Pi_{f'}$.
\end{prop}

\begin{proof}
Let $v^{(1)}$ be the $k$-\emph{anti}-Grassmannian permutation (meaning $v^{(1)}$ has a unique ascent at position $k$) such that $v^{(1)}([k])=\{\aa^{(1)}_1,\cdots,\aa^{(1)}_k\}$.
By \cite{KLS2014projections}*{Proposition~2.1}, among all permutations \(w\in\symm_n\) with \(w([k])=\{\aa^{(i)}_1,\dots,\aa^{(i)}_k\}\) and \(w\ge v^{(i-1)}\) (Bruhat order), there is a unique minimal one; we define \(v^{(i)}\) to be this permutation.
Let $\Pi_f$ be the positroid variety $\Pi_{v^{(1)}}^{v^{(d)}}$.

Now suppose $\mm\in B_{f'}(k,n,d)$ and $\Pi_{f'}=\Pi_{v'}^{u'}$. By Proposition~\ref{prop:chainlift}, there exist permutations $w^{(1)},\cdots, w^{(d)}$ such that $v'\le w^{(1)}\le \cdots \le w^{(k)} \le u'$ and $w^{(i)}([k])=\{\aa^{(i)}_1,\cdots,\aa^{(i)}_k\}$ for all $1\le i\le d$. Naturally, $\mm$ is a standard monomial for $\Pi_{w^{(1)}}^{w^{(k)}}\subseteq \Pi_{f'}$. By \cite{KLS2014projections}*{Proposition 3.3}, we may without loss of generality assume $w^{(1)}$ is anti-Grassmannian; namely $w^{(1)}=v^{(1)}$.
We now argue by induction that $w^{(i)}\ge v^{(i)}$ for each $1\le i\le d$. Since $w^{(i)}\ge w^{(i-1)}\ge v^{(i-1)}$ by the induction hypothesis, and  $v^{(i)}([k])=w^{(i)}([k])=\{\aa^{(i)}_1,\cdots,\aa^{(i)}_k\}$, we must have $v^{(i)}\le w^{(i)}$ by the construction of $v^{(i)}$. Therefore, $\Pi_{f'}\supseteq \Pi_{w^{(1)}}^{w^{(i)}}\supseteq \Pi_f$.
\end{proof}

\begin{theorem}
\label{thm:std-mon-intersect}
Let \(\Pi_f\) be a positroid variety, and write it as an intersection of basic positroid varieties
\(\Pi_f=\bigcap_i \Pi_{S_i\le r_i}\) as in \eqref{eqn:basic-int}.
For each \(i\), let \(f_i\in \mathrm{Bound}(k,n)\) satisfy \(\Pi_{f_i}=\Pi_{S_i\le r_i}\).
Then $B_f(k,n,d)=\bigcap_i B_{f_i}(k,n,d)$.
\end{theorem}
\begin{proof}
    First notice that whenever $\Pi_f\subseteq \Pi_{f'}$, we have $B_f(k,n,d)\subseteq B_{f'}(k,n,d)$. Indeed, Since $\cJ_f\supseteq \cJ_{f'}$, we have $\inn_\omega(\cJ_f)\supseteq \inn_\omega(\cJ_{f'})$, so $B_f(k,n,d)\subseteq B_{f'}(k,n,d)$. It follows that 
$B_f(k,n,d)\subseteq\bigcap_i B_{f_i}(k,n,d)$.
    
    We now show that if $\mm \in \bigcap_i B_{f_i}(k,n,d)$, then $\mm\in B_f(k,n,d)$. By Proposition~\ref{prop:unique-min}, there is a unique minimum positroid variety $\Pi_g$ such that $\mm\in B_g(k,n,d)$ and $\Pi_g\subseteq \Pi_{f_i}$ for each $i$. Therefore $\Pi_g\subseteq \Pi_f$.  It follows that $\mm\in B_f(k,n,d)$.
\end{proof}

\begin{example}[Standard monomials of a positroid variety] Let $f=[5,2,4,7,9,12]$. Then $\Pi_f=\Pi_{[2]\le 0}\cap \Pi_{[2,4]\le 1} \cap \Pi_{[1,5]\le 2}$. Equipped with Theorem~\ref{thm:initial-ideal}, we are ready to describe the Stanley-Reisner complex, denoted $\Delta(\Pi_f)$, of the Hodge degeneration of this positroid variety directly, instead of considering it as a projection of an order complex in the Bruhat order.  The vertices of $\Delta(\Pi_f)$ are labeled by the Pl\"{u}cker coordinates $[136],[146],[156],[356],[456]$. Notice that these form a chain in the Pl\"ucker poset, but they do not form a face of $\Delta(\Pi_f)$ since $\scalebox{0.7}{\ytableausetup{centertableaux}
\yshort{1{\rr{3}}, {\rr{4}}5,66}}$ contains a generalized antidiagonal for $\Pi_{[2,4]\le 1}$. The facets of $\Delta(\Pi_f)$ correspond to the monomials
\[\ytableausetup{centertableaux}
\yshort{1114,3455,6666}
\text{ \ \ and \ \ }
\yshort{1134,3555,6666}
,\]
each of dimension 3.
\end{example}

%%%%%%%%%%%%%%%%%%%%%%%%
\section{Gr\"{o}bner bases for  positroid varieties}\label{section: gb-equations}
%%%%%%%%%%%%%%%%%%%%%%%%

The goal of this section is to give explicit equations for the Gr\"{o}bner basis of a positroid variety inside $\Gr(k,n)$.
For this section, assume that $\k$ is an algebraically closed field such that $\textrm{char}(\k) = 0$ or $\textrm{char}(\k) > k+1$.

Fix a cyclic interval $S\subseteq [n]$ and an integer $r\geq 0$.
Recall that \(\Pi_{S\le r}\subset \gr(k,n)\) denotes the basic positroid variety defined by the rank condition on \(S\),
and \(\cJ_{S\le r}\subset R(k,n)\) is its defining ideal under the Pl\"ucker embedding.

Let $T=(T_i^j)$ be a $k\times d$ tableau with entries in $[n]$, with no assumptions on row/column order or repetitions.
Write $T^j=(T_1^j,\dots,T_k^j)\in[n]^k$ for the $j$th column, and set
\[
\mm_T:=\prod_{j=1}^d [T^j]\in R(k,n).
\]
Assume we have chosen a subset of positions
\[
D\subseteq \{(i,j)\in [k]\times[d] : T_i^j\in S\}
\qquad\text{with}\qquad |D|=r+1.
\]
Let $\symm_D\cong \symm_{r+1}$ be the symmetric group on the set of positions $D$.
It acts on tableaux by permuting the entries in the positions of $D$ and leaving all other entries fixed; equivalently, each $\pi\in\symm_D$ sends $T$ to the tableau $\pi\cdot T$ defined by
\[
(\pi\cdot T)_{i}^{j}=
\begin{cases}
T_{\pi^{-1}(i,j)} & \text{if }(i,j)\in D,\\
T_i^j & \text{if }(i,j)\notin D.
\end{cases}
\]
In other words, \(\pi\in\symm_D\) permutes the entries of \(T\) among the chosen positions \(D\),
and fixes all entries outside \(D\).
This gives an embedding $\symm_D\hookrightarrow \symm_T\cong \symm_{kd}$, where $\symm_T$ denotes the permutation group of the $kd$ positions of $T$.

For each column $1\le j\le d$, define
\[
\lambda_j := |\{\, i\in[k] : (i,j)\in D \,\}|,
\]
so that $\lambda_j$ records how many of the selected positions lie in column $j$.

\begin{prop}\label{prop: positroid-equations}
The polynomial $\mathbf{p} = \sum_{w\in \symm_D} \sgn(w) (\mm_{w\cdot T})$ is in the ideal $\cJ_{S\leq r}$. 
\end{prop}

We demonstrate what equations of these types look like in the following example.

\begin{example}\label{ex: gb-relations} Let $k=4$, $n=11$, $S = [3,7]$, and $r = 2$. The following degree 2 polynomial is in $\cJ_{S\leq r}$, where $D=\{(1,2),(2,1),(3,1)\}$:
\begin{equation}\label{eq: example-partition-1-2}\scalebox{0.9}{
\ytableausetup{centertableaux}
\yshort{1{\rr{3}},{\rr{5}}5,{\rr{7}}9,{10}{11}} \; - \; \yshort{1{\rr{5}},{\rr{3}}5,{\rr{7}}9,{10}{11}} \; -\; \yshort{1{\rr{3}},{\rr{7}}5,{\rr{5}}9,{10}{11}} \; + \; \yshort{1{\rr{7}},{\rr{3}}5,{\rr{5}}9,{10}{11}}\; +\; 
\yshort{1{\rr{5}},{\rr{7}}5,{\rr{3}}9,{10}{11}}\;-\;
\yshort{1{\rr{7}},{\rr{5}}5,{\rr{3}}9,{10}{11}}.}
\end{equation}
which simplifies to 
\begin{equation*}
    2\,[1,5,7,10]\cdot [3,5,9,11]-2\,[1,3,5,10]\cdot [5,7,9,11].
\end{equation*}

We give another example where the tableau $T$ (and any $w\cdot T$) is not semistandard. Here $D = \{(1,2),(2,1),(3,3)\}$. The following degree 3 polynomial is in $\cJ_{S\leq r}$:
\begin{gather*}\scalebox{0.9}{
\ytableausetup{centertableaux}
\yshort{2{\rr{5}}9,{\rr{7}}1{4},9{10}{\rr{3}},{11}8{2}}
 \; - \; 
\yshort{2{\rr{3}}9,{\rr{7}}1{4},9{10}{\rr{5}},{11}8{2}} \; - \; 
\yshort{2{\rr{7}}9,{\rr{5}}1{4},9{10}{\rr{3}},{11}8{2}} \;  + \;
\yshort{2{\rr{3}}9,{\rr{5}}1{4},9{10}{\rr{7}},{11}8{2}} \; + \;
\yshort{2{\rr{7}}9,{\rr{3}}1{4},9{10}{\rr{5}},{11}8{2}}\; - \;
\yshort{2{\rr{5}}9,{\rr{3}}1{4},9{10}{\rr{7}},{11}8{2}}.
}
\end{gather*}
\end{example}

Before proving \Cref{prop: positroid-equations}, we need the following Lemma.

\begin{lemma}[\cite{FultonYT}*{Lemma~1 of Chapter~9.1}]\label{lem: unsigned-plucker-relation}
Fix \(1\le q\le k\). For an index set \(I=\{i_1<\cdots<i_{k-q}\}\subset [k]\), let
\([d(I)]\) denote the \(k\)-tuple obtained from \([d_1,\dots,d_k]\) by replacing
\(d_{i_1},\dots,d_{i_{k-q}}\) (in these positions) with \(c_{q+1},\dots,c_k\) in this order.
Then in \(\k[\gr(k,n)]\) one has
\[
[c_1,\dots,c_k]\,[d_1,\dots,d_k]
=
\sum_{I\in \binom{[k]}{k-q}}
[c_1,\dots,c_q,d_{i_1},\dots,d_{i_{k-q}}]\,[d(I)].
\]
\end{lemma}

\begin{proof}[Proof of Proposition~\ref{prop: positroid-equations}] We
proceed by double induction on $d$, the degree of the relation, and $\lambda_d$.
\vskip 0.5em

\noindent \underline{\textbf{Base case: $d = 1$}.} This is immediate from Corollary~\ref{cor:vanishplucker}.
\vskip 0.5em

\noindent \underline{\textbf{Inductive step on $d$}.} Assume that the proposition holds in $d-1$. We will proceed by induction on $\lambda_d$.
\vskip 0.5em

\underline{\textbf{Base case: $\lambda_d = 0$.}} This means that the entries in $D$ occupies at most $d-1$ columns and that every monomial in $\mathbf{p}$ is divisible by some Pl\"ucker variable $[T^d]$. By inductive hypothesis, $\mathbf{p}/[T^d]\in \cJ_{S\leq r}$ and thus $\mathbf{p}\in \cJ_{S\leq r}$. 
\vskip 0.5em

\underline{\textbf{Inductive step on $\lambda_d$.}} Assume the desired relation holds for $\lambda_d < \ell$.   We choose a fixed position $(i_0,d)\in D$. Let $I:=\{i_1,\cdots,i_{\lambda_{d-1}}\}$ be indices such that $(i_y,d-1)\in D$ for $1\le y \le \lambda_{d-1}$. 

Apply the relation in Lemma \ref{lem: unsigned-plucker-relation} in the case where $q=k-1$ with variables given by the last two columns of each $w\cdot T$, so that the $i_0$-th element in column $d$ of each monomial is swapped with each element in column $d-1$. Let $\tau_i\in \symm_T$ denote the transposition that swaps the fixed element with the $i$th element in column $d-1$.
Then we have
\begin{align*}
    \mathbf{p} &=  \sum_{w\in \symm_D} \sgn(w) (\mm_{w\cdot T}) \\
    &= \sum_{i=1}^k \sum_{ w\in \symm_D} \sgn(w) (\mm_{\tau_i \cdot ( w\cdot T)})  \\
    &=\sum_{i\not\in I} \sum_{ w\in \symm_D} \sgn(w) (\mm_{\tau_i \cdot ( w\cdot T)}) +
    \sum_{i\in I} \sum_{ w\in \symm_D} \sgn(w) (\mm_{\tau_i \cdot ( w\cdot T)}) 
\end{align*}
Consider the first sum. For $i\not\in I$, let $D_i:=D\cup\{(i,d-1)\}\setminus\{(i_0,d)\}$. 
Then
\begin{align*}
    \sum_{i\not\in I} \sum_{ w\in \symm_D} \sgn(w) (\mm_{\tau_i \cdot ( w\cdot T)}) 
    &= \sum_{i\not\in I}\sum_{w'\in\symm_{D_i}} \sgn(\tau_i w' \tau_i)(\mm_{\tau_i\cdot (\tau_i w'\tau_i) \cdot T)}) \\
    &=\sum_{i\not\in I}\sum_{w'\in\symm_{D_i}} \sgn(w')(\mm_{w'\cdot (\tau_i \cdot T)})
\end{align*}
By the induction hypothesis on $\lambda_d$, this sum is in $\cJ_{S\le r}$. 

For the second sum, notice that for $i\in I$, $\tau_iw\in \symm_D$. It follows then
\begin{align*}
    \sum_{i\in I} \sum_{ w\in \symm_D} \sgn(w) (\mm_{\tau_i \cdot ( w\cdot T)}) 
    =&\sum_{i\in I} \sum_{ \tau_i w\in \symm_D} \sgn(w) (\mm_{\tau_i \cdot ( w\cdot T)}) \\
    =&\sum_{i\in I} \sum_{ \tau_i w\in \symm_D} -\sgn(\tau_i w) (\mm_{\tau_i \cdot ( w\cdot T)}) \\
    =&-\lambda_{d-1} \mathbf{p}.
\end{align*}
Hence \((1+\lambda_{d-1})\mathbf p\in \cJ_{S\le r}\), and by the characteristic assumption \(\mathbf p\in \cJ_{S\le r}\).
\end{proof}

\begin{remark}
Our assumption that the characteristic of the underlying field is either $0$ or at least $k+1$ stems from the fact that in the proof above, we require that $(1+\lambda_{d-1})$ be invertible. Since $\lambda_j$ is at most $k$ (the number of rows), it suffices to assume that the characteristic of the field is at least $k+1$.
\end{remark}

\begin{example}
We demonstrate the idea of the proof above using Example \ref{ex: gb-relations}. Let

\begin{equation*}\mathbf{p} = \scalebox{0.9}{
\ytableausetup{centertableaux}
\yshort{1{\rr{3}},{\rr{5}}5,{\rr{7}}9,{10}{11}} \; - \; \yshort{1{\rr{5}},{\rr{3}}5,{\rr{7}}9,{10}{11}} \; -\; \yshort{1{\rr{3}},{\rr{7}}5,{\rr{5}}9,{10}{11}} \; + \; \yshort{1{\rr{7}},{\rr{3}}5,{\rr{5}}9,{10}{11}}\; +\; 
\yshort{1{\rr{5}},{\rr{7}}5,{\rr{3}}9,{10}{11}}\;-\;
\yshort{1{\rr{7}},{\rr{5}}5,{\rr{3}}9,{10}{11}}.}
\end{equation*}
Applying Lemma \ref{lem: unsigned-plucker-relation} to the red element in the rightmost column of each tableau, we see that $\mathbf{p}$ equals

\begin{gather*}
     \;
\yshort{{\rr{3}}{1},{\rr{5}}5,{\rr{7}}9,{10}{11}} \; - \; \yshort{{\rr{5}}1,{\rr{3}}5,{\rr{7}}9,{10}{11}} \; -\; \yshort{{\rr{3}}1,{\rr{7}}5,{\rr{5}}9,{10}{11}} \; + \; \yshort{{\rr{7}}1,{\rr{3}}5,{\rr{5}}9,{10}{11}}\; +\; 
\yshort{{\rr{5}}1,{\rr{7}}5,{\rr{3}}9,{10}{11}}\;-\;
\yshort{{\rr{7}}1,{\rr{5}}5,{\rr{3}}9,{10}{11}}\\
+\; \left( \;
\yshort{1{\rr{5}},{\rr{3}}5,{\rr{7}}9,{10}{11}} \; - \; \yshort{1{\rr{3}},{\rr{5}}5,{\rr{7}}9,{10}{11}} \; -\; \yshort{1{\rr{7}},{\rr{3}}5,{\rr{5}}9,{10}{11}} \; + \; \yshort{1{\rr{3}},{\rr{7}}5,{\rr{5}}9,{10}{11}}\; +\; 
\yshort{1{\rr{7}},{\rr{5}}5,{\rr{3}}9,{10}{11}}\;-\;
\yshort{1{\rr{5}},{\rr{7}}5,{\rr{3}}9,{10}{11}}\;\right)\\
+\;\left(\; 
\yshort{1{\rr{7}},{\rr{5}}5,{\rr{3}}9,{10}{11}} \; - \; \yshort{1{\rr{7}},{\rr{3}}5,{\rr{5}}9,{10}{11}} \; -\; \yshort{1{\rr{5}},{\rr{7}}5,{\rr{3}}9,{10}{11}} \; + \; \yshort{1{\rr{5}},{\rr{3}}5,{\rr{7}}9,{10}{11}}\; +\; 
\yshort{1{\rr{3}},{\rr{7}}5,{\rr{5}}9,{10}{11}}\;-\;
\yshort{1{\rr{3}},{\rr{5}}5,{\rr{7}}9,{10}{11}}\; \right)\\
+\; \left( \;
\yshort{1{10},{\rr{5}}5,{\rr{7}}9,{\rr{3}}{11}} \; - \; \yshort{1{10},{\rr{3}}5,{\rr{7}}9,{\rr{5}}{11}} \; -\; \yshort{1{10},{\rr{7}}5,{\rr{5}}9,{\rr{3}}{11}} \; + \; \yshort{1{10},{\rr{3}}5,{\rr{5}}9,{\rr{7}}{11}}\; +\; 
\yshort{1{10},{\rr{7}}5,{\rr{3}}9,{\rr{5}}{11}}\;-\;
\yshort{1{10},{\rr{5}}5,{\rr{3}}9,{\rr{7}}{11}}\;\right),
\end{gather*}
where the sum of the $i$-th terms in each grouping equals the $i$-th term in $\mathbf{p}$.
Notice that the monomials in each group either 
\begin{itemize}
    \item all have a column containing $3$ elements in [3,7], or
    \item all sum to $-\mathbf{p}$.
\end{itemize}
We can then conclude that $3\mathbf{p}\in\cJ_{[3,7]\leq 2}$, and thus $\mathbf{p}\in \cJ_{[3,7]\leq 2}$.

\end{example}

\begin{lemma}\label{lem:mingens-basic}
Let \(S=[\alpha+1,\alpha+m]\subset[n]\) and \(r<m\).
The minimal monomial generators of \(\inn_\omega(\cJ_{S\le r})\) consist of:
\begin{enumerate}[(a)]
\item the quadratic monomials \([\aa][\bb]\) with \(\aa,\bb\) incomparable in the Pl\"ucker poset \(\mathcal P\); and
\item the monomials \(\mm_T\) with \(T\in B(k,n,d)\) that minimally contain a generalized antidiagonal for \(S\le r\)
(in the sense of Definition~\ref{def:antidiag}).
\end{enumerate}
Moreover, no monomial in (a) divides a monomial in (b), and the monomials in (b) are pairwise non-dividing.
\end{lemma}

\begin{proof}
By Theorem~\ref{thm:grgrobner}, \(\inn_\omega(\cJ)\) is generated by the monomials in (a),
and these are minimal since they all have degree \(2\).
Now let \(\mm\) be a standard monomial for \(\gr(k,n)\), so \(\mm\notin \inn_\omega(\cJ)\).
By Theorem~\ref{thm:initial-ideal}, \(\mm\in \inn_\omega(\cJ_{S\le r})\) if and only if the tableau of \(\mm\)
contains a generalized antidiagonal for \(S\le r\).
If \(\mm\) is such a monomial, delete columns one at a time until no further deletion is possible while remaining in
\(\inn_\omega(\cJ_{S\le r})\); the resulting divisor \(\mm_T\) is minimal and corresponds exactly to a tableau \(T\)
minimally containing a generalized antidiagonal, giving (b).
Finally, a monomial in (b) is a product of Pl\"ucker variables forming a chain in \(\mathcal P\), hence it is not divisible by any
incomparable product \([\aa][\bb]\) from (a).
\end{proof}

We are now ready to give explicit equations for the Gr\"{o}bner basis of $\cJ_{S\leq r}$ where $S$ is an interval. Recall that \(<_\omega\) is the term order induced by a linear extension of the Pl\"ucker poset, as in Section~\ref{section: hodge-algebra-background}. 

\begin{theorem}\label{thm: gb-basic-positroid} 
The following polynomials form a minimal Gr\"{o}bner basis of $\cJ_{S\leq r}$ under the term order $<_\omega$.

\begin{enumerate}[(i)]
    \item straightening relations as in Definition~\ref{def:Pluckerrelations};
\item polynomials as in Proposition~\ref{prop: positroid-equations} where
\(T\in B(k,n,d)\) (Section~\ref{section: initial-ideal}) minimally contains a generalized antidiagonal
\(D\) for \(S\le r\) (Definition~\ref{def:antidiag}).
\end{enumerate}
\end{theorem}

\begin{proof}
By Theorem~\ref{thm:grgrobner} and Proposition~\ref{prop: positroid-equations}, both classes of polynomials are in the ideal $\cJ_{S\leq r}$.
By Theorem~\ref{thm:initial-ideal}, it remains to check that $\mm_T$ is the leading term in $\sum_{w\in \symm_D} \sgn(w)(\mm_{w\cdot T})$ for any $T\in B(k,n,d)$ containing a generalized antidiagonal $D$. 

Observe that if $w$ only permutes entries within each column of $T$, then $\sgn(w)\mm_{w\cdot T} = \mm_{T}$. Otherwise, let $j$ be the smallest column index such that $(w\cdot T)^j$ is different, as a multiset, from $T^j$. Since $D$ is a generalized antidiagonal, $T_{i}^j>T_{i'}^{j'}$ for all $(i,j),(i',j')\in D$ such that $j<j'$. Therefore 
either $[(w\cdot T)^j]<[T^j]$ in $\cP$ or $[(w\cdot T)^j] = 0$. Since $<_\omega$ is revlex order, $\mm_T>_\omega \mm_{w\cdot T}$. Therefore $\mm_T$ is the leading term of $\sum_{w\in \symm_D}\sgn(w)(\mm_{w\cdot T})$. 

Minimality now follows from Lemma~\ref{lem:mingens-basic}:
the leading monomials of the polynomials in (i) and (ii) are exactly the minimal monomial generators of
\(\inn_\omega(\cJ_{S\le r})\), so no element of the Gr\"obner basis can be removed.
\end{proof}

Similar to the standard monomials, we can obtain Gr\"obner basis of $\cJ_{S^\vee\leq r^\vee}$ from Gr\"obner basis of $\cJ_{S\leq r}$. 

\begin{prop}
    \label{prop:dual-gb}
    $\mathrm{gb}(\cJ_{S\le r})$ is a Gr\"obner basis of $\cJ_{S\le r}$ if and only if $\{g^\vee: g\in \mathrm{gb}(\cJ_{S\le r})\}$ is a Gr\"obner basis for $\cJ_{S^\vee\le r^\vee}$.
\end{prop}

\begin{proof}
We assume, without loss of generality, that $S$ is an interval.
Let $g = \sum_{w\in \symm_D}\sgn(w)(\mm_{w\cdot T})\in \mathrm{gb}(\cJ_{S\le r})$, then $g^\vee \in \cJ_{S^\vee\le r^\vee}$ by Proposition~\ref{prop:basic-dual}. We first show that the leading term of $g^\vee$ is $\mm_{T}^\vee$. For any $w\in \symm_{D}$, by the proof of Theorem~\ref{thm: gb-basic-positroid}, we have
$\mm_T>_\omega \mm_{w\cdot T}$. 

Let $j$ be the largest column index such that ${T}^j$ is different from $(w\cdot T)^j$ as a multi-set. Since $w\in \symm_D$ and no entries in column $j+1$ through $d$ is permuted into column $j$, we have either $[(w\cdot T)^j]  = 0$ or $[(w\cdot T)^j] \geq [T^j]$ in $\cP$. The latter is equivalent to $[(T^j)^\vee] \geq [((w\cdot T)^j)^\vee]$. Since $<_\omega$ is revlex order, we conclude that 
\[\mm_T^\vee >_\omega \mm_{w\cdot T}^\vee,\]
and thus $\inn_\omega(g^\vee) = \mm_T^\vee$. By Proposition~\ref{prop:init-dual}, the set of $\mm_T^\vee$ generate the initial ideal $\inn_\omega(\cJ_{S^\vee\leq r^\vee})$. Therefore $\{g^\vee: g\in \mathrm{gb}(\cJ_{S\le r})\}$ forms a Gr\"obner basis.
\end{proof}

\begin{example}
    Suppose $n=5$, $k=2$, and consider $\Pi_{[5,1]^\circ \le 1}$. Then $\varphi(\Pi_{[5,1]^\circ \le 1})=\Pi_{[2,4]\le 2}\subset \gr(3,5)$. There is a degree 3 element of $\mathrm{gb}(\cJ_{[2,4]\le 2})$:
    \[\ytableausetup{centertableaux}
\yshort{11{\rr{2}},2{\rr{3}}4,{\rr{4}}55} - \yshort{11{\rr 2},2{\rr 4}4,{\rr 3}55} - \yshort{11{\rr 3},2{\rr 2}4,{\rr 4}55}
 .\] 
 Therefore, there is a degree 3 generator of $\mathrm{gb}(\cJ_{[5,1]^\circ\le 1})$:
 \[\ytableausetup{centertableaux}
\yshort{123,345} - \yshort{124,335} - \yshort{133,245}
\; .\]
\end{example}

Finally, we have the following statement for the Gr\"obner basis of an arbitrary positroid variety. For a basic positroid ideal \(\cJ_{S\le r}\), let \(\mathrm{gb}(\cJ_{S\le r})\) denote the Gr\"obner basis described in Theorem~\ref{thm: gb-basic-positroid}.

\begin{theorem}\label{thm:concat-gb}
Assume \(\Pi_f=\bigcap_{(i,j)\in ess(f)} \Pi_{[i,j]^\circ\le r(f)_{i,j}}\) as in \eqref{eqn:basic-int},
and for each \((i,j)\in ess(f)\) let \(f_{i,j}\in \bound(k,n)\) satisfy
\(\Pi_{f_{i,j}}=\Pi_{[i,j]^\circ\le r(f)_{i,j}}\).
Then $G:=\bigcup_i \mathrm{gb}(\cJ_{f_i})$ is a Gr\"obner basis for $\cJ_f$.
\end{theorem}

\begin{proof}
By Theorem~\ref{thm:std-mon-intersect} we have
$\inn_\omega(\cJ_f)=\sum_i \inn_\omega(\cJ_{f_i})$.
Since each $\mathrm{gb}(\cJ_{f_i})$ is a Gr\"obner basis for $\cJ_{f_i}$, we have
$\inn_\omega(\cJ_{f_i})=\langle \inn_\omega(g): g\in \mathrm{gb}(\cJ_{f_i})\rangle$.
Therefore
\[
\inn_\omega(\cJ_f)
=\sum_i \langle \inn_\omega(g): g\in \mathrm{gb}(\cJ_{f_i})\rangle
=\langle \inn_\omega(g): g\in G\rangle,
\]
so $G$ is a Gr\"obner basis for $\cJ_f$.

\end{proof}

\begin{remark}
    In \cite{knutson2009frobenius} there is a general result proven about concatenating Gr\"obner bases, based on degeneration of splittings, where both splittings are on the same affine space. Based on discussions with Knutson, we expect that there should be a generalization in which the ambient space degenerates too (as in the Hodge degeneration of the Grassmannian) that would produce our Theorem \ref{thm:concat-gb} about concatenation.
\end{remark}

%%%%%%%%%%%%%%%%%%%%%%%%%%%%%
%%%%%%%%%%%%%%%%%%%%%%%%%%%%%
\section{Promotion and evacuation on standard monomials}
\label{sec:promotion}
%%%%%%%%%%%%%%%%%%%%%%%%%%%%%
%%%%%%%%%%%%%%%%%%%%%%%%%%%%%

\subsection{Promotion}
\label{subsec:promotion}
Let \textbf{promotion} be the map $\mathrm{prom}:B(k,n,d)\longrightarrow B(k,n,d)$ defined as follows:
\begin{itemize}
    \item[-] If $\mm\in B(k,n,d)$ does not contain $n$, then increase each entry of $\mm$ by $1$.
    \item[-] If $\mm$ contains $n$, replace each $n$ with $\bullet$ and perform the following \textbf{jeu de taquin (jdt)} slide:
    \begin{align*}
        \ytableausetup{boxsize=1em}
        \begin{ytableau}
            a & c\\
            b & \bullet
        \end{ytableau}\longrightarrow 
        \begin{cases}
            \ytableausetup{boxsize=1em}
            \begin{ytableau}
            a & \bullet\\
            b & c
        \end{ytableau}& \text{ if }b\leq c \text{ or }a,b \text{ do not exist (i.e. $\bullet$ is in the leftmost column)} ,\vspace{0.5cm}\\
        \ytableausetup{boxsize=1em}
            \begin{ytableau}
            a & c\\
            \bullet & b
        \end{ytableau}& \text{ if }b> c\text{ or }a,c \text{ do not exist (i.e. $\bullet$ is in the top row)}
        \end{cases}
    \end{align*}
    until all \(\bullet\) entries form a Young diagram in the northwest corner (i.e.\ they occupy the leftmost boxes of the first row, then the leftmost boxes of the second row, etc.). Replace each $\bullet$ with $0$ and increase all entries by $1$. 
\end{itemize}
When performing promotion on a rectangular tableau with multiple \(\bullet\) entries,
we fix a deterministic order of jeu de taquin slides as in \cite{bloom2016proofs}:
we always slide into the leftmost available \(\bullet\) (breaking ties top-to-bottom).
See \cite{bloom2016proofs}*{Section 2} for an in-depth description of promotion on semistandard tableaux. 

\begin{theorem}\label{thm:promotion}
Let \(f\in \mathrm{Bound}(k,n)\) and \(d\ge 0\).
Write \(B_f(k,n,d)\subseteq B(k,n,d)\) for the degree-\(d\) standard monomials of \(\Pi_f\),
i.e.\ those \(\mm\in B(k,n,d)\) with \(\mm\notin \inn_\omega(\cJ_f)\) (Section~\ref{section: initial-ideal}).
Let \(\chi:\mathrm{Bound}(k,n)\to \mathrm{Bound}(k,n)\) be the cyclic shift
\(\chi(f)(i)=f(i-1)+1\) (Section~\ref{section: positroid-background}).
Then for any \(\mm\in B(k,n,d)\),
\[
\mm\in B_f(k,n,d)\quad \Longleftrightarrow \quad \mathrm{prom}(\mm)\in B_{\chi(f)}(k,n,d).
\]
In particular, \(\mathrm{prom}\) restricts to a bijection \(B_f(k,n,d)\to B_{\chi(f)}(k,n,d)\).
\end{theorem}

% \begin{theorem}\label{thm:promotion}
%     Suppose $f\in \bound(k,n)$ and $\mm = \prod_{i = 1}^d [\aa^{(i)}]\in B(k,n,d)$. Then $\mm\in B_f(k,n,d)$ if and only if $\mathrm{prom}(\mm) \in B_{\chi(f)}(k,n,d)$. In particular,     promotion gives a bijection between $B_f(k,n,d)$ and $B_{\chi(f)}(k,n,d)$ for all $k,n,d$ and all $f\in \bound(k,n)$.  
% \end{theorem}

In order to prove this statement, we reduce the statement to basic positroid varieties using the intersection description
of initial ideals from Theorem~\ref{thm:std-mon-intersect}.
For a basic cyclic interval rank condition \(S\le r\), Proposition~\ref{prop:basicprom}
shows that promotion carries \(\inn_\omega(\cJ_{S\le r})\) to \(\inn_\omega(\cJ_{\chi(S)\le r})\).
Combining these implications across the essential rank conditions of \(f\) yields the theorem.
The full proof appears after Proposition~\ref{prop:basicprom}.

\begin{example}\label{ex:promotion}
    Consider the basic positroid variety $\Pi_f = \Pi_{[2,4]\leq 2}\subset \gr(3,5)$. We have $\mm = [1,2,4]\cdot [2,3,5]\in \inn_\omega(\cJ_f)$ and $\mathrm{prom}(\mm) = [1,2,3]\cdot [3,4,5]\in \inn_\omega(\cJ_{\chi(f)})$. 
    Note that not only does promotion take a standard monomial of $\Pi_f$ to a standard monomial of $\Pi_{\chi(f)}$, but it also takes semistandard tableaux in the initial ideal $\inn_\omega(\cJ_f)$ to monomials in the ideal $\inn_\omega(\cJ_{\chi(f)})$. We demonstrate this process below:  \[\ytableausetup{centertableaux,boxsize = 1em}
\yshort{1{\rr{2}},2{\rr{3}},{\rr{4}}5}\rightarrow 
\yshort{1{2},2{3},4{\bullet}}\rightarrow \yshort{1{2},2{3},\bullet{4}}\rightarrow
\yshort{1{2},\bullet{3},2{4}}\rightarrow
\yshort{\bullet{2},1{3},2{4}}\rightarrow
\yshort{0{2},1{3},2{4}}\rightarrow
\yshort{1{\rr{3}},2{\rr{4}},3{\rr{5}}}.\]
\end{example}

\begin{lemma}\label{lemma:promcommute}
Let \(\varphi:\gr(k,n)\to \gr(n-k,n)\) be the complement identification from
Construction~\ref{const: complement-interval} (equivalently, \([\aa^\vee]\mapsto[\aa]\) in Pl\"ucker coordinates).
Then for any \(\mm\in B(k,n,d)\),
\[
\varphi(\mathrm{prom}(\mm))=\mathrm{prom}(\varphi(\mm)).
\]
\end{lemma}

\begin{proof}
    Let $T$ and $T^\vee$ be the semistandard tableaux corresponding to $\mm$ and $\mm^\vee$ with entries $n$ replaced with $\bullet$. Let $\mathrm{infusion}(T)$ and $\mathrm{infusion}(T^\vee)$ be the resulting tableaux after moving all $\bullet$ entries to the top left corner using jdt slides. 

     \noindent\textbf{Case \RNum{1}.} (No $\bullet$ in $T$): In this case, all entries in the bottom row of $T^\vee$ are $\bullet$. Therefore all jdt slides occurring in $T^\vee$ move $\bullet$ up but not to the left.
     Promotion simply adds one to all entries in $\mm$. For $\mm^\vee$, the effect of promotion is equivalent to adding one cyclically to each entry (so $n$ becomes 1) and then sorting within each column. As a result, $\varphi\circ \mathrm{prom}(\mm) = \mathrm{prom}(\mm^\vee)$ in this case.

     \noindent\textbf{Case \RNum{2}.} (No $\bullet$ in $T^\vee$):
     This follows formally from Case \RNum{1} by taking a dual.

     \noindent\textbf{Case \RNum{3}.} ($\bullet$ in both $T$ and $T^\vee$): 
    Let $\delta$ be the number of $\bullet$ in $T$. By assumption, $1\leq \delta \leq d-1$ and there are $(d-\delta)$ $\bullet$ in $T^\vee$. By definition, the set of (non-$\bullet$) entries in column $i$ of $T$ is disjoint from the set of entries in column $d+1-i$ of $T^\vee$. We will show that this property still holds for $\mathrm{infusion}(T)$ and $\mathrm{infusion}(T^\vee)$.
    
    We use the following two different orders of jdt slides for $T$ and $T^\vee$:
    \begin{itemize}
        \item[-] For $T$, we always slide into the $\bullet$ that is in the leftmost column. Equivalently, we first move the $\bullet$ in column $d-\delta+1$ to position $(1,1)$, then the $\bullet$ in column $d-\delta+2$ to position $(1,2)$ and so on. 
        \item[-] For $T^\vee$, we move each $\bullet$, from left to right, until it changes column. When all $\bullet$ have changed columns, we start with the leftmost $\bullet$ again and repeat the process until all $\bullet$ end up in the first row and column $1$ through $\delta$.
    \end{itemize}
    Notice that there are exactly $(\delta d-\delta^2)$ jdt slides that change the column index of $\bullet$ in both $T$ and $T^\vee$. Moreover, for any $i\in [\delta d-\delta^2]$, the $\bullet$ in $T$ moved from column $j+1$ to $j$ in the $i$-th slide if and only if the $\bullet$ in $T^\vee$ moved from column $d+1-j$ to $d-j$.
    
    Let $T_i$ and $T^{\vee}_i$ be the tableaux after the $i$-th such jdt slides. 
    \begin{claim}\label{claim:jdt}
        For all $j\in [d]$, the set of entries in column $j$ of $T_i$ is disjoint from the set of entries in column $d+1-j$ of $T_i^\vee$.
    \end{claim}
     \begin{proof}
    We proceed by induction on $i$. For $i = 0$, $T_0 = T$ and $T^\vee_0 = T^\vee$. The sets of entries are disjoint by definition. Suppose the statement holds for $i-1$. Let $j$ be the column index such that $\bullet$ moved from column $j+1$ to $j$ in $T$. 
    Let $\alpha$ and $\beta$ be the entry that are switched with $\bullet$ in $T_{i-1}$ and $T^\vee_{i-1}$ respectively. Let $C_1,C_2$ be the sets of entries in column $j$ and $j+1$ in $T_{i-1}$ respectively. Then 
       \begin{equation}\label{eqn:alpha}
           \alpha = \max\{a\in [n]: |C_1\cap[a,n]| > |C_2\cap[a,n]|\}.
       \end{equation}
       Similarly, since $C_1^\vee:=[n]\setminus C_1$ and $C_2^\vee:=[n]\setminus C_2$ are the sets of entries in column $d-j$ and $d+1-j$ in $T^\vee_{i-1}$, we have
       \begin{equation}\label{eqn:beta}
           \beta = \max\{b\in [n]: |C_2^\vee\cap[b,n]| > |C_1^\vee\cap[b,n]|\}.
       \end{equation}
       Since
       \[|C_1\cap [a,n]| + |C_1^\vee \cap [a,n]| =|C_2\cap [a,n]| + |C_2^\vee \cap [a,n]|,\]
       it is easy to see that the right-hand side of \eqref{eqn:alpha} and \eqref{eqn:beta} agrees. Therefore $\alpha = \beta$ and thus Claim holds for $i$. 
    \end{proof}
    By Claim~\ref{claim:jdt}, the set of entries in column $j$ of $T_{\delta d-\delta^2}$ is disjoint from entries in column $d+1-j$ of $T^\vee_{\delta d-\delta^2}$. Since $T_{\delta d-\delta^2}$ and $\mathrm{infusion}(T)$ have the same set of entries for each column and so do $T^\vee_{\delta d-\delta^2}$ and $\mathrm{infusion}(T^\vee)$, we are done.
\end{proof}

\begin{example}
We demonstrate Case \RNum{1} and \RNum{3} of Lemma~\ref{lemma:promcommute} in $\gr(3,5)$. The following two diagrams commute:
\[\text{Case \RNum{1}:}
\begin{tikzcd}[row sep=0.6cm,column sep=1cm]
\yshort{1{2},2{3},4{4}}  \arrow[r,"\text{prom}"] \arrow[d,"\varphi"] & \yshort{2{3},3{4},5{5}}\arrow[d,"\varphi"]\\
\yshort{1{3},5{5}} \arrow[r,"\text{prom}"] & \yshort{1{1},2{4}}
\end{tikzcd}
\qquad
\text{Case \RNum{3}:}
\begin{tikzcd}[row sep=0.6cm,column sep=1cm]
\yshort{1{2},2{3},{4}5}  \arrow[r,"\text{prom}"] \arrow[d,"\varphi"] & \yshort{1{3},2{4},3{5}}\arrow[d,"\varphi"]\\
\yshort{1{3},4{5}} \arrow[r,"\text{prom}"] & \yshort{1{4},2{5}}
\end{tikzcd}.
\]
\end{example}

For any cyclic interval $S = [\alpha+1,\alpha+m]^\circ$, set $\chi(S):=[\alpha+2,\alpha+m+1]^\circ$. 

\begin{lemma}\label{lemma:jdtantidiag}
    For any interval $S\subset [n-1]$, a monomial $\mm\in \inn_\omega(\cJ_{S\leq r})$ if and only if $\mathrm{prom}(\mm) \in \inn_\omega(\cJ_{\chi(S)\leq r})$.
\end{lemma}
\begin{proof}
    ($\implies$):
    By Theorem~\ref{thm:initial-ideal}, $\mm$ contains a generalized antidiagonal of size $r+1$ with entries in $S$. 
    Since $S\subset [n-1]$, no entry in any generalized antidiagonal is affected when $n$ is turned into $\bullet$. Let $j$ be the number of jdt slides applied to $\mm$ and let $T_j$ be the resulting tableau after $j$ slides. 

    \begin{claim}\label{claim:genanti}
       $T_j$ contains a generalized antidiagonal with entries in $S$ for all $j$. 
    \end{claim}
    \begin{proof}
        We proceed by induction on $j$, the number of jdt slides applied to $T$. If $j = 0$, by definition, there is a generalized antidiagonal in $T_0 = \mm$. Suppose the claim holds for $j-1$ and let \ytableausetup{boxsize=1em}
    \begin{ytableau}
        \none & a\\
        b & \bullet
    \end{ytableau} be the local picture of the $j$-th jdt slide. Let $A :=\{(\rho_{i},\sigma_{r+2-i}):i\in [r+1]\}$ be the set of positions of a generalized antidiagonal in $T_{j-1}$. We now divide into four cases based on if $a$ or $b$ lies in $A$:

    \noindent\textbf{Case \RNum{1}.}($a,b\notin A$): In this case, the jdt slide does not change the position of any entry in the generalized antidiagonal, and thus $A$ is still a generalized antidiagonal after the jdt slide. 

    \noindent\textbf{Case \RNum{2}.}($a\in A,b\notin A$): If $b>a$, then the jdt slide does not change any entry in $A$ and thus $A$ remains generalized antidiagonal. If $a\geq b$, then for any entry $c$ on the left of $b$ in the same row, we have $a\geq c$. Thus $c\notin A$ and sliding $a$ down does not change the fact that entries in $A$ form a generalized antidiagonal.

    \noindent\textbf{Case \RNum{3}.}($a\notin A,b\in A$): Here we focus on the case where $b>a$ since the jdt slide does not change entries in $A$ if $a\geq b$. Let $\sigma$ be the column index of $b$ and let $\{(\rho_i,\sigma):i\in I\}\subset A$ be the set of entries of $A$ in column $\sigma$ that are above $b$.
    
    With the order of jdt slides for $T$ as in Lemma~\ref{lemma:promcommute}, no two $\bullet$ appear in the same column. Thus the entries $\{(\rho_i,\sigma+1):i\in I\}$ contain no $\bullet$ in $T_{j-1}$ and are disjoint from $A$. Let $A' = A\setminus \{(\rho_i,\sigma):i\in I\} \cup \{(\rho_i,\sigma+1):i\in I\}$. 
    Since 
    \[T_{j-1}(\rho_i,\sigma)\leq T_{j-1}(\rho_i,\sigma+1) \leq a < b,\]
    $A'$ is a generalized antidiagonal in $T_{j-1}$ with $b\in A'$ being the top entry in column $\sigma$. Therefore sliding $b$ to the right does not violate any condition for $A'$ to be a generalized antidiagonal. 
    
    \noindent\textbf{Case \RNum{4}.}($a,b\in A$): Since $a,b\in A$, we have $b>a$ and $A$ is still a generalized antidiagonal after sliding $b$ to the right.
    \end{proof}
    
    By Claim~\ref{claim:genanti}, there is a generalized antidiagonal with entries in $S$ after all jdt slides.
    Since all entries will increase by $1$ after jeu de taquin, $\mathrm{prom}(\mm)$ contains a generalized antidiagonal of size $r+1$ with entries in $\chi(S)$. By Theorem~\ref{thm:initial-ideal}, $\mathrm{prom}(\mm) \in \inn_\omega(\cJ_{\chi(S)\leq r})$.

    ($\impliedby$): This direction follows from the same reasoning as above, with the slight change of replacing $\mathrm{prom}$ with $\mathrm{prom}^{-1}$.
\end{proof}

\begin{prop}\label{prop:basicprom}
    For any cyclic interval $S\subset [n]$ and any $\mm\in B(k,n,d)$, we have $\mm\in \inn_\omega(\cJ_{S\leq r})$ if and only if $\mathrm{prom}(\mm) \in \inn_\omega(\cJ_{\chi(S)\leq r})$.
\end{prop}
\begin{proof}
    Since $S\subset [n]$ is a cyclic interval, either $S\subset [n-1]$ or $S^\vee \subset [n-1]$. If $S\subset[n-1]$, this is the content of Lemma~\ref{lemma:jdtantidiag}. If $S^\vee \subset [n-1]$, by Proposition~\ref{prop:init-dual}, 
    \[\varphi(\mm) = \mm^\vee\in \inn_\omega(\cJ_{S^\vee\leq r^\vee}).\]
    By Lemma~\ref{lemma:promcommute} and Lemma~\ref{lemma:jdtantidiag}, 
    \[\varphi\circ \mathrm{prom}(\mm) = \mathrm{prom}\circ \varphi(\mm) \in \inn_\omega(\cJ_{\chi(S^\vee)\leq r^\vee}).\]
    Since $\chi(S^\vee) = \chi(S)^\vee$, by Proposition~\ref{prop:init-dual}, 
    \[\mathrm{prom}(\mm)\in  \inn_\omega(\cJ_{\chi(S)\leq r}).\qedhere\]
\end{proof}

\begin{proof}[Proof of Theorem~\ref{thm:promotion}] Let $\{S_i\leq r_i\}$ be the set of interval rank conditions associated to $f$. By Theorem~\ref{thm:std-mon-intersect},
\[\sum_{i}\inn_\omega(\cJ_{S_i\leq r_i}) = \inn_\omega(\cJ_f),\]
where both sides are square-free monomial ideals.
Therefore for $\mm\in B(k,n,d)$,
\[\mm\in \inn_\omega(\cJ_f)\iff \mm\in \inn_\omega(\cJ_{S_i\leq r_i})\]
for some $i$. Since $\sum_{i}\inn_\omega(\cJ_{\chi(S_i)\leq r_i}) = \inn_\omega(\cJ_{\chi(f)})$, by Proposition~\ref{prop:basicprom}, 
\[
\mm\in \inn_\omega(\cJ_f) \iff \mathrm{prom}(\mm)\in \inn_\omega(\cJ_{\chi(f)})\]
which is equivalent to
\[
\mm\in B_f(k,n,d) \iff \mathrm{prom}(\mm)\in B_{\chi(f)}(k,n,d).\qedhere\]
\end{proof}

\subsection{Evacuation}
\label{subsec:evacuation}
\begin{definition}\label{def:evac}
    Let \textbf{evacuation} be the map $\evac:B(k,n,d)\to B(k,n,d)$ defined by:
\begin{enumerate}
    \item replace every entry $x$ with $n+1-x$,
    \item rotate the tableau by $180^\circ$. 
\end{enumerate}
\end{definition}

For rectangular semistandard tableaux, \(\evac\) again produces a semistandard tableau of the same shape; see \cite{bloom2016proofs}*{Theorem 2.9(d)}.

\begin{remark}
    Similar to promotion, evacuation is also defined for semistandard tableaux of arbitrary shape. Our definition above is \cite{bloom2016proofs}*{Theorem 2.9(d)} and is specific to semistandard tableaux of rectangular shapes. 
\end{remark}

\begin{theorem}\label{thm:evac}
Let \(f\in \mathrm{Bound}(k,n)\) and \(d\ge 0\).
Let \(f^*\in \mathrm{Bound}(k,n)\) be defined by the condition
\(\Pi_{f^*}=w_0\cdot \Pi_f\) (see Section~\ref{section: positroid-background}).
Then for any \(\mm\in B(k,n,d)\),
\[
\mm\in B_f(k,n,d)\quad \Longleftrightarrow\quad \evac(\mm)\in B_{f^*}(k,n,d).
\]
\end{theorem}

\begin{example}
    Let $k=4$, $n=11$, and consider $\Pi_f=\Pi_{[3,7]\le 2}$. 
    Then $\Pi_{f^*}=\Pi_{[5,9]\le 2}$. We show an example of evacuation mapping a standard monomial in $B_f(k,n,d)$ to a standard monomial in $B_{f^*}(k,n,d).$
\end{example}
\ytableausetup{boxsize=1.2em}
\[
\begin{tikzcd}[row sep=0.6cm,column sep=2cm]
\yshort{1{\rr{3}},{\rr{5}}5,{\rr{7}}9,{10}{11}}\arrow[r,"x\mapsto n+1-x"] & \yshort{{11}{\rr{9}},{\rr{7}}7,{\rr{5}}3,{2}{1}}\arrow[r,"\text{rotate }180^\circ"] & \yshort{12,3{\rr{5}},7{\rr{7}},{\rr{9}}{11}}
\end{tikzcd}
\]
We need the following lemma:
\begin{lemma}\label{lemma:evac-dual}
    For any $\mm\in B(k,n,d)$, $\evac(\mm^{\vee}) = \evac(\mm)^\vee$.
\end{lemma}
\begin{proof}
    For $\aa\in {[n]\choose k}$, set 
$\aa^* := \{n+1-a:a\in \aa\}.$ 
The lemma follows from $\aa^{*\vee} = \aa^{\vee *}$.
\end{proof}

\begin{proof}[Proof of \Cref{thm:evac}]

By Theorem~\ref{thm:std-mon-intersect}, it suffices to show the statement holds when $\Pi_f=\Pi_{S\le r}$. Let $S^*:= \{n+1-s:s\in S\}$, we will show that \[\mm \in \inn_\omega(\cJ_{S\leq r})\iff \evac(\mm)\in \inn_{\omega}(\cJ_{S^*\leq r}).\] 
In fact, since both $\evac$ and $S\mapsto S^*$ are involutions, it is enough to show the forward direction.

We use the generalized antidiagonal characterization of \(\inn_\omega(\cJ_{S\le r})\) from Theorem~\ref{thm:initial-ideal}.
Suppose $S$ is an interval. Let $\mm = \prod_{i = 1}^d [\aa^{(i)}]\in \inn_\omega(\cJ_{S\leq r})$. 
By \Cref{def:evac}, $\evac(\mm) = \prod_{i = 1}^d[\bb^{(i)}]$ where $\bb^{(i)} = (\aa^{(d+1-i)})^*$. 
Since $\mm$ contains a generalized antidiagonal $\{\aa^{(\sigma_{r+1})}_{\rho_1}< \dots <\aa^{(\sigma_{1})}_{\rho_{r+1}}\}$ for $S\le r$ and $\bb^{(d+1-\sigma_{r+2-i})}_{k+1-\rho_i} = n+1-\aa^{(\sigma_{r+2-i})}_{\rho_i}$, we have
\begin{equation}\label{eqn:*genanti}
    \bb^{(d+1-\sigma_{1})}_{k+1-\rho_{r+1}}< \dots< \bb^{(d+1-\sigma_{r+1})}_{k+1-\rho_1}.
\end{equation}
Since $\sigma_1 \leq \dots \leq \sigma_{r+1}$ and $\rho_1 < \dots < \rho_{r+1}$, the $r+1$ element set in \eqref{eqn:*genanti}
forms a generalized antidiagonal of $\evac(\mm)$ for $S^*\leq r$ and thus $\evac(\mm)\in \inn_\omega(\cJ_{S^*\leq r})$. 

Now suppose $S$ is a wrapped-around interval. By \Cref{lemma:evac-dual}, $\evac(\mm^\vee) = \evac(\mm)^\vee$. Since $S^\vee$ is an interval, the following diagram commutes:
\begin{center}
\begin{tikzcd}[row sep=0.6cm,column sep=1cm]
\mm\in \inn_{\omega}(\cJ_{S\leq r})  \arrow[r,mapsto] \arrow[d,mapsto] & \mm^\vee \in \inn_{\omega}(\cJ_{S^\vee\leq r^\vee})\arrow[d,mapsto]\\
\evac(\mm) \arrow[r,mapsto] & \evac(\mm)^{\vee} = \evac(\mm^{\vee})\in \inn_\omega(\cJ_{S^{\vee *}\leq r^\vee})
\end{tikzcd}
\end{center}
Since $S^{\vee *} = S^{*\vee}$, by \Cref{prop:init-dual}, $\evac(\mm) \in \inn_\omega(\cJ_{S^{*}\leq r})$. 
\end{proof}

%%%%%%%%%%%%%%%%%%%
\section{A character formula for cyclic Demazure modules}
%%%%%%%%%%%%%%%%%%%
\label{sec:cyclic_demazure}

In this section we give a character formula for cyclic Demazure modules $V_{f}(d\omega_k)$. Section~\ref{subsec:cyclic_demazure_def} provides necessary background and definitions. In Section~\ref{subsec:hilbert-series-positroid},
 we relate a multigraded Hilbert series of the positroid variety $\Pi_f$ with the character of $V_f(d\omega_k)$, and state our main theorem. 

We summarize our strategy as follows. 
By slicing the positroid variety with a certain hyperplane (\Cref{prop:cut-top-point}), we obtain a recurrence \eqref{eqn:positroidrecur} on the Hilbert series of $\Pi_f$. 
To give an explicit inductive formula, we establish a connection between the $K$-theory of matrix Schubert varieties and that of positroid varieties (\Cref{prop:chain}). Specifically, in Section~\ref{subsec:msv}, we interpret Monk's rule for Grothendieck polynomials as an identity on (twisted) $K$-polynomials for matrix Schubert varieties. In Section~\ref{subsec:relating-hilb}, we relate the desired recurrence on the Hilbert series of positroid varieties to a similar recurrence on matrix Schubert varieties. This allows us to deduce our desired inductive formula (\Cref{thm:hilbrecur}) using Lenart's $K$-theoretic Monk's rule \cite{lenart2003k}.

\subsection{Cyclic Demazure modules}
\label{subsec:cyclic_demazure_def}
In this section we provide a detailed introduction to cyclic Demazure modules, which were introduced by Lam \cite{Lam19}. Readers familiar with these constructions should feel free to skip to the next section.

Let $G/B$ be a flag variety and $\lambda$ a dominant weight.
Under the Borel-Weil isomorphism, the dual of the Demazure module $V_{w}(\lambda)$ can be identified with the space of global sections $H^0(G/B, \mathcal{L}_\lambda)$ restricted to the Schubert variety $X_w$, where $\mathcal{L}_\lambda$ is the line bundle associated to $\lambda$. The character of a Demazure module can be computed by the well-known Demazure character formula via applications of the Demazure operators. When $\lambda=d\omega_k$ is a fundamental weight, there is an isomorphism $H^0(G/B,\mathcal{L}_{d\omega_k})\cong H^0(\gr(k,n),\mathcal{L}_{d\omega_k}) =: (V(d\omega_k))^\ast$ for $d\ge 0$. 

For a positroid variety $\Pi_f\subseteq \gr(k,n)$, define the restriction map:
\begin{equation}
    \res_f: H^0(\gr(k,n),\mathcal{L}_{d\omega_k})\longrightarrow H^0(\Pi_f,\mathcal{L}_{d\omega_k}).
\end{equation}

We first note that positroid varieties are projectively normal in the Pl\"ucker embedding. 
\begin{prop}
\label{prop:section-coord-d}
     For $d\ge 0$ and $f\in\bound(k,n)$, we have
     \begin{equation}
     H^0(\Pi_f, \mathcal{L}_{d\omega_k})\cong \k[\Pi_f]_d.
     \end{equation}
\end{prop}
\begin{proof}
    Positroid varieties are normal by \cite{KLS2014projections}*{Corollary 4.7}. By the proof of  \cite{KLS2014projections}*{Theorem 4.5}, the map $\res_f$ is surjective for any $f\in\bound(k,n)$. Hence positroid varieties are projectively normal, and the claim follows by \cite{hartshorne1977algebraic}*{Ex.II.5.14}.
\end{proof}
For $I\in {[n]\choose k}$, set $f_I\in \bound(k,n)$ to be the bounded affine permutation such that $\Pi_{f_I} = X_I$, where $X_I$ is the Schubert variety as in \eqref{eqn:Schubertvariety}.
By Proposition~\ref{prop:section-coord-d}, 
$H^0(X_I,\mathcal{L}_{d\omega_k})$ can be naturally identified with a quotient space of $H^0(\gr(k,n),\mathcal{L}_{d\omega_k})$. 
The \textbf{Demazure module} for $d\omega_k$ is
\begin{equation}
    V_{I}(d\omega_k) := (H^0(X_I,\mathcal{L}_{d\omega_k}))^*.
\end{equation}
Then $V_{I}(d\omega_k)$ is a subspace of $V(d\omega_k)$ by 
\begin{equation}\label{eqn:demazuremod}
    V_{I}(d\omega_k) = \{v\in V(d\omega_k): s(v) = 0 \text{ for all }s\in \ker(\res_{f_I})\}.
\end{equation}

Let $\chi:\gr(k,n)\rightarrow \gr(k,n)$ be as defined in \eqref{eqn:chi}. Define 
\[\chi^*: H^0(\gr(k,n),\mathcal{L}_{d\omega_k}) \rightarrow H^0(\gr(k,n),\mathcal{L}_{d\omega_k})\]
to be the pullback map induced by $\chi^{-1}$. 
Since $V(d\omega_k)\cong (H^0(\gr(k,n),\mathcal{L}_{d\omega_k}))^*$, the map $\chi^*$ induces a map on $V(d\omega_k)$. We abuse the notation and denote the map on $V(d\omega_k)$ by $\chi$. We are now ready to define the cyclic Demazure module as introduced in \cite{Lam19}.
\begin{definition}\label{def:cyclicdemazure}
    For $f\in \bound(k,n)$ and $(I_0,\dots, I_{n-1})$ the corresponding juggling state as in \Cref{def:jugglingstate}, the \textbf{cyclic Demazure module} is:
    \begin{equation}
        V_f(d\omega_k) := \bigcap_{i = 0}^{n-1} \chi^i(V_{I_i}(d\omega_k)).
    \end{equation}
\end{definition}
\begin{remark}
    We note the slight difference in notation between Definition~\ref{def:cyclicdemazure} and the definition in \cite{Lam19} using Grassmann necklaces.
    The $j$-th Grassmann necklace is precisely $I_j+j$ where addition is done entrywise mod $n$ as usual. 
\end{remark}

\begin{lemma}
\label{lem:kerres}
    Let $f\in \bound(k,n)$ corresponding to the juggling state $(I_0,\dots, I_{n-1})$. Define $f_i:=\chi^i(f_{I_i})$ to be the bounded affine permutation associated to the $i$th rotated Schubert variety, and $U_i:=\ker(\res_{f_i}).$ Then $\sum_{i=0}^{n-1} U_i=\ker(\res_{f})$. 
\end{lemma}

\begin{proof}
    Under the identification given in Proposition~\ref{prop:section-coord-d}, $U_i\cong \{s\in \k[\gr(k,n)]_d : s\in (\cJ_{f_i})_d\}$ and $\ker(\res_f)\cong \{s\in\k[\gr(k,n)]_d: s\in (\cJ_{f})_d\}$. The claim then follows from the fact that $\cJ_f=\sum_{i=0}^{n-1}{\cJ}_{f_i}$.
\end{proof}

\begin{lemma}\label{lemma:rotatedSchub}
    $(\chi^i(V_{I}(d\omega_k)))^* \cong H^0(\Pi_{\chi^i f_I},\mathcal{L}_{d\omega_k})$.\footnote{We thank an anonymous referee for this short proof.}
\end{lemma}

\begin{proof}
Since $\chi$ is an automorphism of $\Gr(k,n)$, we have
$\Pi_{f_i}=\chi^i(X_{I_i})$ and hence $\cJ_{f_i}=\chi^i(\cJ_{f_{I_i}})$.
Moreover, $\Pi_f=\bigcap_{i=0}^{n-1}\Pi_{f_i}$ scheme-theoretically, so
$\cJ_f=\sum_{i=0}^{n-1}\cJ_{f_i}$.

By Proposition~\ref{prop:section-coord-d}, under the identification
$H^0(\Gr(k,n),\mathcal{L}_{d\omega_k})\cong R(k,n)_d$ we have
$\ker(\res_g)=(\cJ_g)_d$ for any $g$. Therefore
\[
\ker(\res_f)=(\cJ_f)_d=\Bigl(\sum_{i=0}^{n-1}\cJ_{f_i}\Bigr)_d
=\sum_{i=0}^{n-1}(\cJ_{f_i})_d=\sum_{i=0}^{n-1}U_i,
\]
as claimed.
\end{proof}

The following proposition and its corollary allow us to study characters of cyclic Demazure modules through the homogeneous coordinate ring $\k[\Pi_f]$.
\begin{prop}
\label{prop:coord-ring-id}
    As a $T$-representation, 
    \[V_f(d\omega_k)^*\cong H^0(\Pi_f,\mathcal{L}_{d\omega_k})\cong \k[\Pi_f]_d.\]
\end{prop}

\begin{proof}
    By Lemma~\ref{lem:kerres} and Proposition~\ref{prop:section-coord-d}, we have \[H^0(\Pi_f,\mathcal{L}_{d\omega_k})\cong H^0(\gr(k,n), \mathcal{L}_{d\omega_k})/\sum_{i=0}^{n-1}\ker(\res_{f_i}).\] 
    The vector space isomorphism then follows from Definition~\ref{def:cyclicdemazure} and Lemma~\ref{lemma:rotatedSchub}. For $T$-equivariance, it follows from the fact that the action $\chi$ on $\gr(k,n)$ commutes with the standard $T$-action. 
\end{proof} 
Recall that $B_f(k,n,d)$ is the set of semistandard tableaux of shape $k\times d$ with entries $\le n$ that correspond to the standard monomials of $\Pi_f$. For $\alpha\in B_f(k,n,d)$, define $\content(\alpha):=(c_1,\dots,c_n)$ where $c_i$ counts the number of appearances of $i$ in $\alpha$. Finally, define $\mathbf{t}^\alpha:=\prod_{i=1}^n t_i^{\content(\alpha)_i}$. Our characterization of $B_f(k,n,d)$ (Theorem A) gives a tableaux formula for $\ch(V_f(d\omega_k))$.

\begin{cor}
    The $T$-character for the cyclic Demazure module $V_f(d\omega_k)$
    is 
    \[\ch(V_f(d\omega_k))=\sum_{\alpha\in B_f(k,n,d)} \mathbf{t}^\alpha.\]
\end{cor}

\begin{proof}
By \Cref{thm:basic-std-mon}, the degree-$d$ piece $\Bbbk[\Pi_f]_d$ has a basis of standard monomials
$\{m_\alpha : \alpha\in B_f(k,n,d)\}$ indexed by semistandard tableaux of shape
$k\times d$.  Writing $\alpha$ by its columns, let $I_1,\dots,I_d\in\binom{[n]}{k}$
be the column sets, so that
\[
m_\alpha = [I_1]\cdots[I_d]\in \Bbbk[\Pi_f]_d.
\]
For the standard diagonal torus $T=(\Bbbk^\times)^n$, the induced action on Pl\"ucker
coordinates is
\[
(t_1,\dots,t_n)\cdot [I] \;=\;\Bigl(\prod_{i\in I} t_i^{-1}\Bigr)[I].
\]
Hence each $m_\alpha$ is a $T$-weight vector, and its weight is the product of the
weights of its factors:
\[
(t_1,\dots,t_n)\cdot m_\alpha
= \Bigl(\prod_{j=1}^d\prod_{i\in I_j} t_i^{-1}\Bigr)m_\alpha
= \Bigl(\prod_{i=1}^n t_i^{-\content(\alpha)_i}\Bigr)m_\alpha
= \mathbf t^{-\alpha}\, m_\alpha.
\]
Therefore
\[
\ch(\Bbbk[\Pi_f]_d)=\sum_{\alpha\in B_f(k,n,d)} \mathbf t^{-\alpha}.
\]
By Proposition~\ref{prop:coord-ring-id}, $\Bbbk[\Pi_f]_d\cong V_f(d\omega_k)^*$ as
$T$-modules, so $\ch(V_f(d\omega_k)^*)=\sum_{\alpha\in B_f(k,n,d)} \mathbf t^{-\alpha}$.
Finally, since $\ch(W^*)(\mathbf t)=\ch(W)(\mathbf t^{-1})$, we obtain
\[
\ch(V_f(d\omega_k))=\sum_{\alpha\in B_f(k,n,d)} \mathbf t^{\alpha},
\]
as claimed.
\end{proof}

\subsection{Multigraded Hilbert series for positroid varieties}
\label{subsec:hilbert-series-positroid}
Fix a grading on $R:=R(k,n)$ by assigning each Pl\"ucker variable $\aa\in\binom{[n]}{k}$ the degree $\deg ([\aa])=(d_1,\dots, d_n)$ where $d_i=1$ if $i\in\aa$ and $d_i=0$ otherwise. The ring $R$ is then a $\ZZ^n$-graded polynomial ring, where $\deg([\aa])$ for all $\aa\in\binom{[n]}{k}$ lie in an open half space. This ensures that if $M$ is a $\ZZ^n$-graded finitely generated $R$-module, the $\mathbf{d}$th graded piece $M_\mathbf{d}$ is a finite dimensional $\mathbb{k}$-vector space for all $\mathbf{d}\in \ZZ^n$. For more general background on multigraded Hilbert series and $K$-polynomials, we refer the reader to \cite{miller2005combinatorial}*{Chapter 8}.

The following proposition is straightforward from the definitions of cyclic Demazure characters.
\begin{prop}
The degree-$d$ component of the multigraded Hilbert series 
\begin{equation}
    \hilb(R/\cJ_f, \mathbf{t})=\sum_{d=0}^\infty \sum_{\alpha \in B_f(k,n,d)}\mathbf{t}^\alpha,
\end{equation}
is the character of $V_f(d\omega_k)$. 
\end{prop}
 
To find a formula for the character of $V_f(d\omega_k)$, we establish a recurrence relation on the Hilbert series of positroid varieties. We rely on the following geometric identity. 

\begin{definition}\label{def:topp}
Let \(f\in \mathrm{Bound}(k,n)\) and write \(\Pi_f=\Pi_{v(f)}^{u(f)}\) as in Lemma~\ref{lemma:wtouv}.
Define
\[
\topp(f):=v(f)([k])=\{v(f)(1),\dots,v(f)(k)\}\in \binom{[n]}{k}.
\]
Equivalently, \([\topp(f)]\) is the lexicographically smallest Pl\"ucker coordinate that does not vanish on \(\Pi_f\).
\end{definition}

\begin{prop}
\label{prop:cut-top-point}
Let $C_0(f) := \{f'\in \bound(k,n): f'\gtrdot_0 f\}$.
Then the following formula is true scheme-theoretically:
\begin{equation}
\label{eqn:cotrans}
    \{[\topp(f)]=0\} \cap \Pi_f = \bigcup_{f'\in C_0(f)}\Pi_{f'}.
\end{equation}
\end{prop}

\begin{proof}
    This is exactly \cite{KLS13juggling}*{Corollary 7.3}, translated to the language of bounded affine permutations.
\end{proof}

We now compute the Hilbert series of both sides of \eqref{eqn:cotrans}. Since $[\topp(f)]$ is not a zero divisor of $R/\cJ_f$, we have 
\begin{equation}\label{eqn:LHS}
    \hilb(R/(\cJ_f+\langle [\topp(f)]\rangle); \mathbf{t}) = (1-\mathbf{t}^{\topp(f)}) \hilb(R/\cJ_f;\mathbf{t})
\end{equation}
For the right-hand side, we need the following definition.
\begin{definition}\label{def:L_0}
Recall that \(\lessdot_0\) denotes the covering relation in \(0\)-Bruhat order on \(\widetilde S_n^k\)
(Section~\ref{section: positroid-background}).
Let $\Gamma_0(f)$ denote the set of saturated chains in $0$-Bruhat order,
\[f\lessdot_0 ft_{a_1,b_1}\lessdot_0 ft_{a_1,b_1}t_{a_2,b_2} \lessdot_0\cdots \lessdot_0 ft_{a_1,b_1}t_{a_2,b_2}\cdots t_{a_m,b_m},\]
where we require each $b_i\in [1,n]$, and
and we impose the lexicographic condition that the sequence \((b_i,-a_i)\) is strictly decreasing,
i.e.\ \(b_i>b_{i+1}\) or \(b_i=b_{i+1}\) and \(a_i<a_{i+1}\).
Let $\Lenart_0(f)$ be the set of permutations that are endpoints of chains in $\Gamma_0(f)$.
\end{definition}

We are now ready to state the key Hilbert-series recursion, from which the character recursion
(Theorem~\ref{thm:character}) follows immediately.

\begin{theorem}\label{thm:hilbrecur}
Fix \(f\in \mathrm{Bound}(k,n)\), and let \(R=R(k,n)\) with the \(\mathbb Z^n\)-grading
\(\deg([\aa])=\sum_{i\in\aa}e_i\) described in Section~\ref{subsec:hilbert-series-positroid}.
Let \(C_0(f)=\{f'\in \mathrm{Bound}(k,n): f'\gtrdot_0 f\}\) and let \(\Lenart_0(f)\) be the set of endpoints
of chains in \(\Gamma_0(f)\) (Definition~\ref{def:L_0}).
If \(\Pi_f\) is not a point stratum (equivalently, \(C_0(f)\neq\emptyset\)), then
\begin{equation}\label{eq:hilbert-recursion}
\hilb\!\left(R/\bigcap_{f'\in C_0(f)}\cJ_{f'};\mathbf t\right)
=\sum_{g\in \Lenart_0(f)\cap \mathrm{Bound}(k,n)}(-1)^{\ell(g)-\ell(f)+1}\,
\hilb(R/\cJ_{g};\mathbf t).
\end{equation}
Moreover, for $g\in \Lenart_0(f)$,
\begin{equation}\label{eqn:Jun1aaa}
    g\notin \bound(k,n)\iff g\geq f'\gtrdot_0 f
\end{equation}
for some $f'\notin \bound(k,n)$.
\end{theorem}

We can then deduce the following inductive character formula by combining \eqref{eqn:LHS}  and Theorem~\ref{thm:hilbrecur}.

\begin{theorem}\label{thm:character}
Let \(f\in \mathrm{Bound}(k,n)\) and \(d\ge 0\).
If \(d=0\), then \(\ch(V_f(0\cdot \omega_k))=1\).
If \(d\ge 1\) and \(f\) is not a point stratum, then in the character ring of \(T\) we have
\begin{equation}\label{eq:character-formula}
\ch(V_f(d\omega_k))
=
\mathbf t^{\topp(f)}\left(
\ch(V_f((d-1)\omega_k))
+
\sum_{g\in \Lenart_0(f)\cap \mathrm{Bound}(k,n)}
(-1)^{\ell(g)-\ell(f)+1}\,
\ch(V_g(d\omega_k))
\right).
\end{equation}
\end{theorem}

\begin{remark}
One can rephrase \eqref{eq:character-formula} by replacing the indexing bounded affine permutation with Grassmann intervals.
Let $u = u(f), v = v(f)$ and set $V_v^u:= V_f$. Then
    \begin{equation}
\label{eq:character-formula-uv}
        \ch (V_v^u(d\omega_k))=\mathbf{t}^{v([k])}\ch (V_v^u((d-1)\omega_k)+\sum_{\substack{v'\in \Lenart_k(v)\\ v'\leq u}}(-1)^{\ell(v')-\ell(v)+1} V_{v'}^u(d\omega_k).
\end{equation}
Indeed, all $f'\in \Lenart_0(f)$ lie in the same coset $W_{A_{n-1}}f$ by \Cref{thm:lub} and \Cref{prop:poset-iso}. Equivalently, $u(f) = u(f')$ for all $f'\in \Lenart_0(f)$. In this case, using the bijection between $\bound(k,n)$ and $\mathcal{Q}(k,n)$ (see Section~3.4 of \cite{KLS13juggling}), it is straightforward to see that taking least upper bounds in $\bound(k,n)$ is the same as taking least upper bounds of $v$'s in $\symm_n$ and that $f'\in \Lenart_0(f)\cap \bound(k,n)$ if and only if $v'\in \Lenart_k(v)$ and $v'\leq u$. 
\end{remark}

\begin{remark}
Theorem~\ref{thm:character} should be viewed as a combinatorial character formula:
the coefficient of a monomial \(\mathbf t^\alpha\) is determined by an alternating sum over
certain \(0\)-Bruhat chains starting at \(f\), and can be computed purely combinatorially.
\end{remark}

\begin{example}
    \label{ex:char-formula}
    We illustrate how Theorem~\ref{thm:character} yields explicit combinatorial coefficients
in the character by working out the example \(f=[3465]\).
Then $\Pi_f=\Pi_v^u\subset \gr(2,4)$ where $u=3412$ and $v=2134$, so $\topp(f)=\{1,2\}$. The set $\Gamma_0(f)$ consists of 
    $f\lessdot_0 ft_{-1,1}$, $f\lessdot_0 ft_{0,1}$, and $f\lessdot_0 ft_{-1,1}\lessdot_0 ft_{-1,1}t_{0,1}$. Here $ft_{-1,1}=[2475]$, $ft_{0,1}=[1467]$, and $ft_{-1,1}t_{0,1}=[1476]$. 
    In terms of rank conditions, $\Pi_{[3465]}=\Pi_{[4,1]^\circ\le 1}$, $\Pi_{[2475]}=\Pi_{[4,2]^\circ\le 1}$, $\Pi_{[1467]}=\Pi_{[1]\le 0}$, and $\Pi_{[1476]}=\Pi_{[4,2]^\circ\le 1}\cap \Pi_{[1]\le 0}$.
    We precompute
    \[\ch(V_{[3465]}(\omega_2))=t_1t_2+t_1t_3+t_2t_3+t_2t_4+t_3t_4,\]
    \[\ch(V_{[2475]}(2\omega_2))=t_3^2t_4^2+t_2t_3^2t_4+t_1t_3^2t_4+t_2^2t_3^2+t_1t_2t_3^2+t_1^2t_3^2,\]
    \[\ch(V_{[1467]}(2\omega_2))=t_2^2t_3^2 +t_2^2t_3t_4+t_2t_3^2t_4+t_2^2t_4^2+t_2t_3t_4^2+t_3^2t_4^2.\]
    \[\ch(V_{[1476]}(2\omega_2))=t_2^2t_3^3+t_2t_3^2t_4+t_3^2t_4^2.\]

    Then by (\ref{eq:character-formula}) we have 
    \begin{align*}
    \ch(V_{[3465]}(2\omega_2)) = &t_1^2t_2^2+t_1^2t_2t_3+t_1t_2^2t_3+t_1t_2^2t_4 +t_1t_2t_3t_4+t_3^2t_4^2+t_2t_3^2t_4\\
    &+t_1t_3^2t_4+t_2^2t_3^2+t_1t_2t_3^2+t_1^2t_3^2+t_2^2t_3t_4+t_2^2t_4^2+t_2t_3t_4^2.      
    \end{align*}
  One can easily verify this calculation against our characterization of standard monomials on positroid varieties.
\end{example}

 \subsection{\texorpdfstring{$K$}{K}-theoretic Monk's rule and the Hilbert series of matrix Schubert varieties}
 \label{subsec:msv}
 
We quickly recall the definition of matrix Schubert varieties \cite{Fulton92}, following the conventions of \cite{knutsonMiller}.

\begin{definition}
    Let $w\in \symm_n$. For $1\le q,p \le n-1$, let \[r_{q\times p}:=\abs{\{(i,j)\le (q,p): w(i)=j\}}.\]
    Let $Z$ denote the $n\times n$ matrix whose entry at $(i,j)$ is the indeterminate $z_{i,j}$, and $Z_{q\times p}$ the submatrix of $Z$ with rows in $1,\ldots,q$ and columns in $1,\ldots, p$. The \textbf{matrix Schubert variety} $\overline{X}_w$ is the vanishing locus of the ideal generated by all minors of size $1+r_{q\times p}$ in $Z_{q\times p}$ for all $1\le q,p \le n-1$. The defining ideal $\cI_w$ of the matrix Schubert variety $\overline{X}_w$ is called a \textbf{Schubert determinantal ideal}.
\end{definition}

We also recall the definition of Grothendieck polynomials. 
\begin{definition}
\label{def:grothendieck}
    Let \[\partial_i:=\frac{1-s_i}{x_i-x_{i+1}}\text{ and }\pi_i:= \partial_i(1-x_{i+1}),\]
be operators on polynomials in $\ZZ[x_1,x_2,\cdots]$, where $s_i$ acts by switching $x_i$ with $x_{i+1}$, $x_i$ acts as multiplication by $x_i$, and $1$ denotes the trivial action. 
The \textbf{Grothendieck polynomials} $\mathfrak{G}_w$ for $w\in \symm_\infty$ are the unique polynomials that satisfy
$\mathfrak{G}_{w_0^n}=x_1^{n-1}x_2^{n-2}\cdots x_{n-1}$ where
\[w_0^n:=n\ n-1 \ \cdots \ 1\]
is the longest permutation in $\symm_n$,
and $\pi_i\mathfrak{G}_w=\mathfrak{G}_{ws_i}$ if $\ell(ws_i)=\ell(w)-1$.
\end{definition}

Knutson--Miller \cite{knutsonMiller} relates the Grothendieck polynomials with $K$-polynomials of matrix Schubert varieties. Let $S:=\k[z_{i,j}]_{i,j\in[n]}$ where $\deg(z_{i,j})=x_i$.

\begin{theorem}{\cite{knutsonMiller}*{Theorem A}} 
         The twisted $K$-polynomial for $\overline{X}_w$,
        $\mathcal{K}(S/\mathcal{I}(\overline{X}_w);1-\mathbf{x})$, is equal to $\mathfrak{G}_w.$
\end{theorem}

\begin{remark}
    We note that the definition of Grothendieck polynomials in \cite{knutsonMiller}*{Definition 1.1.3} is different than Definition~\ref{def:grothendieck}. The difference is exactly the twist $x_i\mapsto 1-x_i$.
\end{remark}

We recall Lenart's $K$-theoretic Monk's rule in type $A$ \cite{lenart2003k}. Let $w$ be a permutation and $s_k$ the simple reflection $(k,k+1)$ for $k\ge 1$. 

\begin{definition}\label{def:L_k}
    Let $\Gamma_k(w)$ denote the set of saturated chains in $k$-Bruhat order,
\[
    w\lessdot_k wt_{a_1,b_1}\lessdot_k wt_{a_1,b_1}t_{a_2,b_2} \lessdot_k\cdots \lessdot_k wt_{a_1,b_1}t_{a_2,b_2}\cdots t_{a_m,b_m},
\]
where the pairs $(a_i,b_i)$ satisfies ($b_i>b_{i+1}$) or $(b_i=b_{i+1} \text{ and } a_i<a_{i+1})$, and $\Lenart_k(w)$ is defined as the set of permutations that are endpoints of chains in $\Gamma_k(w)$.
\end{definition}

The following identity holds for Grothendieck polynomials
$\G_w$.
\begin{theorem}[Monk's rule for Grothendieck polynomials \cite{lenart2003k}]
\label{thm:lenart-K-monk}
\[\G_{s_k}(x)\G_w(x)=\sum_{u\in \Lenart_k(w)} (-1)^{\ell(u)-\ell(w)+1}\G_u(x).\] 
\end{theorem}

 We interpret this $K$-theoretic Monk's rule as a geometric identity on matrix Schubert varieties\footnote{We thank Allen Knutson for the proof idea.}.
\begin{prop}
    \label{prop:geometric-monk}
    Let $w$ be a permutation and $s_k$ the simple reflection $(k,k+1)$ for $k\ge 1$. Let $m_{w[k]}$ denote the matrix minor obtained by taking rows in $\{1,\cdots,k\}$ and columns in $w[k]:=\{w(1),\cdots w(k)\}$ of the generic $n\times n$ matrix (where we take $n$ large enough so that all $k$-Bruhat covers of $w$ lies in $\symm_n$). Then scheme-theoretically,
    \[\{m_{w[k]} = 0\}\cap \overline{X}_w=\bigcup_{w\lessdot_k wt_{ab}} \overline{X}_{wt_{ab}}\]
\end{prop}
\begin{proof}
    The structure of this proof is the same as the proof of the geometric cotransition formula in \cite{knutson2022schubert}, as it is a direct generalization.
    For any matrix $M$, let $M_{\alpha,\beta}$ be the submatrix with row indices in $\alpha$ and column indices in $\beta$.

     We claim that $\{m_{w[k]}=0\}\cap \overline{X}_w$ is $B_-\times B$ stable. Since $B_-$ acts on matrices by performing downward row operations, it is clear that $\{m_{w[k]}=0\}\cap \overline{X}_w$ is stable under this action. $B$ acts on matrices by rightward column operations; it suffices to show stability under adjacent column operations. Let $M\in \{m_{w[k]}=0\}\cap \overline{X}_w$ and suppose we add a multiple of column $j$ to column $j+1$ and obtain $M'$. It is clear if $j+1\not\in w[k]$ or $j,j+1\in w[k]$ then $M'\in m_{w[k]}\cap \overline{X}_w$. We assume now that $j+1\in w[k]$ and $j\not\in w[k]$.  Let $r:=\abs{\{i:i\le k, w(i)<j\}}$. Then conditions from $\overline{X}_w$ dictate that the rank of the northwest submatrix with southeast corner $(k,j)$ is at most $r$. If $M_{[k],\{j\}}$ is not in the column span of $M_{[k],[j-1]}$, then the matrix $M_{[k],[j-1]}$ has rank at most $r-1$. Since $w[r]\subseteq [j-1]$, the rank of the matrix $M_{[k], w[r]}$ is at most $r-1$. Since $M_{[k], w[r]}=M'_{[k], w[r]}$, the minor $M'_{[k],w[k]}$ must vanish. If $M_{[k],\{j\}}$ is in the column span of $M_{[k],[j-1]}$, the claim then follows from the fact that $\rk M'_{[k],w([k])}\le\rk M'_{[k],\{j\}\cup w[k]}=\rk M_{[k],\{j\}\cup w[k]}=\rk M_{[k], w[k]}<k$. 

    Therefore, set-theoretically, $\{m_{w[k]}=0\}\cap \overline{X}_w$ is a union of matrix Schubert varieties, each of dimension $\dim \overline{X}_w-1$. Scheme theoretically, the intersection does not have embedded components since $\{m_{w[k]}=0\}$ is a Cartier divisor and $\overline{X}_w$ is normal. The permutations $u$ such that $u\gtrdot w$ and $\overline{X}_u\subset \{m_{w[k]}=0\}$ are exactly of the form $u=wt_{ab}$, where $w\lessdot_k wt_{ab}$.  It remains to show that the multiplicity of each such $\overline{X}_u$ in the scheme-theoretic intersection $\{m_{w[k]}=0\}\cap \overline{X}_w$ is 1. For this, we show that the tangent space to $\{m_{w[k]}=0\}\cap \overline{X}_w$ at the point $u=wt_{ab}$ where $w\lessdot_k u$ is $T_u\overline{X}_u$. The tangent space of $\{m_{w[k]}=0\}$ at $u=wt_{ab}\gtrdot_k w$ is exactly the hyperplane defined by the equation $z_{a,w(a)}=0$. 
    The claim then follows from \cite{knutson2022schubert}*{Lemma 4.4}. \qedhere

\end{proof}

\begin{cor}
\label{cor:monk-hilbert}
    We have the following identity on the Hilbert series:
    \begin{equation}\label{eqn:wmonkhilb}
    (1-x_1\cdots x_k)\hilb(S/\mathcal{I}_w;\mathbf{x})=\hilb(S/\bigcap_{u\gtrdot_k w}\mathcal{I}_u;\mathbf{x})
    \end{equation}
    and consequently on the twisted $K$-polynomials
    \begin{equation}(1-(1-x_1)\cdots (1-x_k))\mathcal{K}(S/\mathcal{I}_w;\mathbf{1-x})=\mathcal{K}(S/\bigcap_{u\gtrdot_k w}\mathcal{I}_u;\mathbf{1-x}),\end{equation}
    namely 
    \begin{equation}
        \mathfrak{G}_{s_k}(\mathbf{x})\mathfrak{G}_{w}(\mathbf{x})=\mathcal{K}(S/\bigcap_{u\gtrdot_k w}\mathcal{I}_u;\mathbf{1-x}).
    \end{equation}
\end{cor}

Combining Theorem~\ref{thm:lenart-K-monk} and Corollary~\ref{cor:monk-hilbert}, we get
\begin{cor}
    \label{cor:grothendieck-agree}
    \begin{equation}
        \mathcal{K}(S/\bigcap_{u\gtrdot_k w}\mathcal{I}_u;\mathbf{1-x})=
        \sum_{u\in \Lenart_k(w)} (-1)^{\ell(u)-\ell(w)+1}\G_u(x).
    \end{equation}
\end{cor}

\subsection{Relating the two Hilbert series}
\label{subsec:relating-hilb}

\begin{definition}\label{def:tilde}
Let \(f\in \mathrm{Bound}(k,n)\) (viewed as an affine permutation \(f:\mathbb Z\to\mathbb Z\)).
Set
\[
A(f):=\{\,i\in [2n]: f(i-n)\in [n]\,\}\subset [2n],
\qquad
A(f)^c=\{i_1<\cdots<i_n\}.
\]
Define \(\widetilde f\in \symm_{2n}\) by
\[
\widetilde f(i)=f(i-n)+k \quad\text{for } i\in A(f),
\]
and on \(A(f)^c\) define \(\widetilde f\) to be strictly decreasing by setting
\[
\widetilde f(i_j)=
\begin{cases}
2n-j+1, & 1\le j\le n-k,\\
n-j+1, & n-k< j\le n.
\end{cases}
\]
For \(U\subseteq \mathrm{Bound}(k,n)\), set \(\widetilde U:=\{\widetilde f: f\in U\}\).
\end{definition}

\begin{example}
    We show an example for computing $\widetilde{f}$ from $f$. Suppose \[f=[6,2,5,10,11,7,8],\] then $n=7$, $k=3$. The dots in the bounded region of Figure~\ref{fig:f-tilde} correspond to $f^{-1}(\{8,9,\cdots,14\})$. 
    Then \[\widetilde{f}=14\ 13\ 12\ 6\ 7\ 11\ 4\ 9\ 5\ 8\ 3\ 2\ 10\ 1 \in \symm_{2n}.\]
\end{example}

\begin{figure}

\begin{tikzpicture}[scale=1,transform shape]
\pgfmathsetmacro{\wid}{0.5}
\draw[SkyBlue, step=\wid] (0,0) grid (7,7);
    \draw[blue, thick] (3*\wid, 6*\wid) -- (3*\wid, 14*\wid);
    \draw[blue, thick] (10*\wid, 0) -- (10*\wid, 8*\wid);
    \foreach \i/\j in {0/0,1/2,2/3,3/7,4/5,5/10,6/9,7/4,8/6,9/1,10/8,11/11,12/12,13/13} {
        \pgfmathsetmacro{\x}{(\i)*\wid+0.5*\wid}%
        \pgfmathsetmacro{\y}{(\j)*\wid+0.5*\wid}%
        \node[] at (\x,\y) {$\bullet$};
    }
    \foreach \i/\j in 
    { 3/6, 4/5, 5/4, 6/3, 7/2, 8/1, 9/0} {
        \pgfmathsetmacro{\x}{(\i)*\wid}%
        \pgfmathsetmacro{\y}{(\j)*\wid}%
        \pgfmathsetmacro{\z}{(\j+8)*\wid}%

        \draw[blue, thick]  (\x,\y)--(\x+\wid,\y);
        \draw[blue, thick]  (\x,\z)--(\x+\wid,\z);
        \draw[blue, thick]  (\x,\y)--(\x, \y+\wid);
        \draw[blue, thick]  (\x+\wid,\z)--(\x+\wid, \z-\wid);
    }
\end{tikzpicture}
\caption{Construction of $\widetilde{f}$}
\label{fig:f-tilde}

    \end{figure}
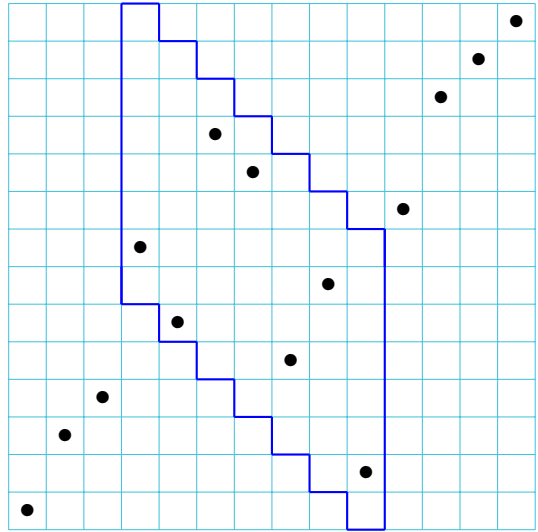
We now relate $\Lenart_0(f)$ in \Cref{def:L_0} with $\Lenart_n(\widetilde{f})$ as in \Cref{def:L_k}.

\begin{lemma}\label{lemma:relateL}
    For any $f,g\in \bound(k,n)$, $g\in \Lenart_0(f)$ if and only if $\widetilde{g}\in \Lenart_{n}(\widetilde{f})$.
\end{lemma}

We defer the proof to Section~\ref{subsec:proof}. Our strategy is as follows.
The forward direction follows by translating \(0\)-Bruhat covers of bounded affine permutations into
\(n\)-Bruhat covers of the associated permutations \(\widetilde f\in \symm_{2n}\).
The reverse direction uses the boundedness filtration in Lemma~\ref{lemma:bounded} to stay inside \(\widetilde{\bound(k,n)}\).   

We will use the following lemma, which follows from the fact that families of Schubert determinantal ideals or positroid ideals are distributive families; see \Cref{appendix: distributivity} for more details.

\begin{lemma}\label{lemma:hilbrecur}
Let $A$ be a finitely generated, $\mathbb{Z}^n$-graded $\k$-algebra. Let $I_1,\dots I_r$ be homogeneous ideals of $A$. In the following two cases,
\begin{enumerate}
    \item $A = \k[z_{i,j}]_{i,j\in [n]}, \deg(z_{i,j}) = t_i$, and each $I_i$ is a Schubert determinantal ideal;
    \item $A = R(k,n)$, $\deg(P_J) = t^J := \prod_{j\in J}t_j$, and each $I_i$ is a positroid ideal.
\end{enumerate}
    Then
    \begin{equation}\label{eqn:hilbrecur}
        \hilb(A/\bigcap_{i\in [r]}I_i) = \sum_{\emptyset\neq U\subseteq [r]}(-1)^{\#U-1}\hilb(A/\sum_{u\in U}I_u).
    \end{equation}
\end{lemma}

For a poset $P$ and a subset $U\subset P$, let $\lub_P(U)$ (or $\lub(U)$ if the underlying poset is clear from context) denote the set of least upper bounds for the elements in $U$.

\begin{lemma}
    \label{lem:inttosum}
    Let $P$ denote either the poset of $\bound(k,n)\sqcup\{\star\}$ or $\symm_n$. Let $U\subset P$. Then 
    \begin{equation}
    \label{eqn:sumtoint}
    \sum_{u\in U}I_u=\bigcap_{u'\in \lub_P(U)} I_{u'}.
    \end{equation}
    Here we declare $\star$ to be above all elements of $\bound(k,n)$ and $I_\star$ the ideal generated by all Pl\"ucker variables.
\end{lemma}
\begin{proof}
    Since Frobenius-split subvarieties are reduced, it is enough to show 
    \[\bigcap_{u\in U}X_u = \bigcup_{u'\in \lub_P(U)} X_{u'}\]
    as sets, 
    where $X_u = V(I_u)$ is the corresponding positroid/matrix Schubert variety. 
    
    We first prove ``$\supseteq$''. For any $(u,u')\in U\times \lub_P(U)$, since $u' \geq u$, we have $X_{u'}\subseteq X_u$. Therefore $\bigcup_{u'\in \lub_P(U)} X_{u'} \subseteq \bigcap_{u\in U}X_u$.
    
    Now we prove ``$\subseteq$''. Suppose there exists $X_v$, an irreducible component of $\bigcap_{u\in U}X_u$ such that $X_v\not\subseteq \bigcup_{u'\in \lub_{P}(U)}X_{u'}$. Then
    $v\ngeq u'$ for any $u'\in \lub_P(U)$. By the definition of $\lub_P(U)$, there exists $u\in U$ such that $u\nleq v$. Therefore $X_v\not\subseteq X_u$, a contradiction. We can then conclude that $\bigcap_{u\in U}X_u = \bigcup_{u'\in \lub_P(U)} X_{u'}$. 
\end{proof}

Therefore, by repeatedly applying \eqref{eqn:hilbrecur} and \eqref{eqn:sumtoint}, we can write 
\begin{equation}\label{eqn:hilbaltsgn}
    \hilb(A/\bigcap_{i\in [r]}I_i) = \sum_{u'\in U'}c_{u'}\hilb(A/I_{u'}),
\end{equation}
for some $U'\subseteq \bound(k,n)$ (or $\symm_n$) and $c_{u'}\in \mathbb{Z}$.

To understand the recursive construction of \eqref{eqn:hilbaltsgn}, we define the poset $\genP_{P}(T)$ to record the terms appearing on the right-hand side of \eqref{eqn:hilbaltsgn}. 
\begin{definition}
    \label{def:lub-poset}
    Let $P$ be a finite poset with a unique maximum, and $T\subset P$ a subset of elements in $P$. Define 
    \[\genP_P(T)=\begin{cases}
        T & \text{ if } |T| = 1 \\
        \bigcup_{T'\subseteq T}\genP_P(\lub_P(T')) & \text{ if } |T|>1
    \end{cases}
    \]
and we endow $\genP_{P}(T)$ with a poset structure by inheriting the partial order from $P$.
\end{definition}

\begin{remark}
     In the case where $P = \mathcal{S}_n$ or $\bound(k,n)\cup \{\star\}$, the least upper bounds of a subset $T\subseteq P$ correspond to the irreducible components of the intersection of matrix Schubert/positroid varieties indexed by $T$. The poset $\genP_{P}(T)$ is closely related to posets with the ``intersection-decompose'' property as defined in \cite{K09mobius}.
\end{remark}

\begin{example} Let $P$ be the poset on $\mathcal{S}_4$ under Bruhat order. Let $a:=2413$, $b:=4123$, $c:=2341$, $d:=3142$, and $T:=\{a,b,c,d\}$. Since generating the poset $\genP_{P}(T)$ only requires taking upper bounds, it is enough to look within the interval $[2143,4321]$ as in Figure~\ref{fig:2143-4321} instead of the entire $P$. We have
\begin{align*}
    \lub(a,b) = \{4213\}, \lub(a,c) &= \{2431\},\lub(a,d) = \{3412\}\\
    \lub(b,c) = \{4231\},\lub(b,d) &= \{4132\},\lub(c,d) = \{3241\}\\
    \lub(a,b,c) = \{4231\}, \lub(a,b,d) &= \{4312,4231\}, \lub(a,c,d) = \{3421,4231\}\\
    \lub(b,c,d) = \{4231\}, &\lub(a,b,c,d) = \{4231\}.
\end{align*}
By Definition~\ref{def:lub-poset}, if $\lub(T')$ is a singleton, then $\genP_{P}(\lub(T')) = \lub(T')$. Therefore
\[\genP_{P}(T) = \bigcup_{T'\subseteq T} \lub(T')\cup \genP_{P}(\lub(a,b,d))\cup\genP_{P}(\lub(a,c,d)).\]
Since 
\[\genP_{P}(\lub(a,b,d)) = \{4312,4231,4321\}\text{ and }\genP_{P}(\lub(a,c,d)) = \{3421,4231,4321\},\]
we conclude that $\genP_{P}(T) = \{w : 2143<w\leq 4321\}$.

\begin{figure}[ht]
    \centering
    \begin{tikzcd}[ampersand replacement=\&,cramped,sep=small]
	\&\&\&\& 2143 \\
	\& 2413 \&\& 4123 \&\& 2341 \&\& 3142 \\
	4213 \&\& 2431 \&\& 3412 \&\& 4132 \&\& 3241 \\
	\&\& 4312 \&\& 4231 \&\& 3421 \\
	\&\&\&\& 4321
	\arrow[no head, from=1-5, to=2-2]
	\arrow[no head, from=1-5, to=2-4]
	\arrow[no head, from=1-5, to=2-6]
	\arrow[no head, from=1-5, to=2-8]
	\arrow[no head, from=2-2, to=3-1]
	\arrow[no head, from=2-2, to=3-3]
	\arrow[no head, from=2-2, to=3-5]
	\arrow[no head, from=2-4, to=3-1]
	\arrow[no head, from=2-4, to=3-7]
	\arrow[no head, from=2-6, to=3-9]
	\arrow[no head, from=2-8, to=3-9]
	\arrow[no head, from=3-1, to=4-5]
	\arrow[no head, from=3-3, to=2-6]
	\arrow[no head, from=3-3, to=4-5]
	\arrow[no head, from=3-3, to=4-7]
	\arrow[no head, from=3-5, to=2-8]
	\arrow[no head, from=3-5, to=4-7]
	\arrow[no head, from=3-7, to=2-8]
	\arrow[no head, from=3-7, to=4-5]
	\arrow[no head, from=3-9, to=4-5]
	\arrow[no head, from=3-9, to=4-7]
	\arrow[no head, from=4-3, to=3-1]
	\arrow[no head, from=4-3, to=3-5]
	\arrow[no head, from=4-3, to=3-7]
	\arrow[no head, from=4-3, to=5-5]
	\arrow[no head, from=4-5, to=5-5]
	\arrow[no head, from=4-7, to=5-5]
\end{tikzcd}
    \caption{The interval $[2143,4321]$ in Bruhat order.}
    \label{fig:2143-4321}
\end{figure}
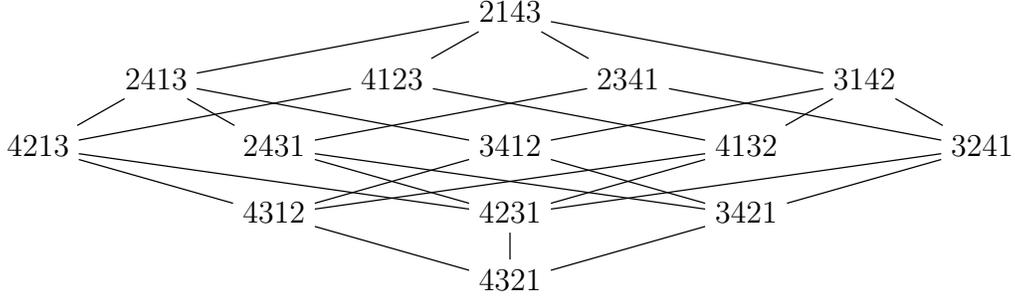

\end{example}

For $w\in \symm_n$ and $f\in \bound(k,n)$, set
\begin{equation}\label{eqn:C_k}
    C_k(w):= \{u\in \symm_n:u\gtrdot_k w\},
\end{equation}
and
\begin{equation}\label{eqn:C_0}
    C_0(f) := \{f'\in \bound(k,n): f' \gtrdot_0 f\}.
\end{equation}
Applying \eqref{eqn:hilbaltsgn} to the RHS of \eqref{eqn:wmonkhilb}, we have
\begin{equation}\label{eqn:schubrecur}
    \hilb(S/\bigcap_{u\in C_k(w)}\mathcal{I}_u) = \sum_{u'\in \genP_{\symm_n}(C_k(w))} c_{u'}\hilb(S/\mathcal{I}_{u'}),
\end{equation}
where each $c_{u'}\in \{0,\pm 1\}$ by Corollary~\ref{cor:grothendieck-agree}. Similarly, we have
\begin{equation}\label{eqn:positroidrecur}
    \hilb(R/\bigcap_{u\in C_0(f)}\cJ_f) = \sum_{f'\in \genP_{\bound(k,n)\sqcup \{\star\}}(C_0(f))} c_{f'}\hilb(R/\cJ_{f'}).
\end{equation}

\begin{lemma}\label{lemma:Ln}
    For $f\in\bound(k,n)$, let $\widetilde{f}$ be as defined in Definition~\ref{def:tilde} and $C_n(\widetilde{f})$ as in \eqref{eqn:C_k}. 
    Then $\Lenart_n(\widetilde{f})\subseteq \genP_{\symm_{2n}}(C_n(\widetilde{f}))$ consists of permutations $u'$ such that $c_{u'}\neq 0$ in \eqref{eqn:schubrecur}.
\end{lemma}

\begin{proof}
    Combining \eqref{eqn:schubrecur} and Corollary~\ref{cor:grothendieck-agree}, we have
    \begin{align*}
        \mathcal{K}(S/\bigcap_{u\in C_n(\tilde{f})}\mathcal{I}_u, \mathbf{1-x}) &= \sum_{w\in \mathcal{L}_n(\tilde{f})} (-1)^{\ell(w)-\ell(\tilde{f})-1}\mathcal{K}(S/\cI_{w},\mathbf{1-x})\\
        &= \sum_{u'\in \genP_{\symm_{2n}}(C_n(\tilde{f}))} c_{u'}\mathcal{K}(S/\cI_{u'},\mathbf{1-x}). 
    \end{align*}
Since $\mathcal{K}(S/I_w,\mathbf{1-x}) = \mathfrak{G}_{w}(\mathbf{x})$, we obtain the following equation on Grothendieck polynomials:
\begin{equation}\label{eqn:Grothendieckequal}
    \sum_{w\in \mathcal{L}_n(\tilde{f})} (-1)^{\ell(w)-\ell(\tilde{f})-1}\mathfrak{G}_w(\mathbf{x})
        = \sum_{u'\in \genP_{\symm_{2n}}(C_n(\tilde{f}))} c_{u'}\mathfrak{G}_{u'}(\mathbf{x}).
\end{equation}
Since the set of Grothendieck polynomials $\mathfrak{G}_w(\mathbf{x})$ for $w\in \symm_{2n}$ are linearly independent, \eqref{eqn:Grothendieckequal} holds if and only if 
\[\mathcal{L}_n(\tilde{f}) = \{u'\in \genP_{\mathcal{S}_{2n}(C_n(\tilde{f}))}: c_{u'}\neq 0\}\text{ and }c_{u'} = (-1)^{\ell(u')-\ell(\tilde{f})-1}.\]
This concludes the proof.
\end{proof}
To prove Theorem~\ref{thm:hilbrecur}, we
relate the coefficients $c_{f'}$ appearing on the right-hand side of \eqref{eqn:positroidrecur} with the coefficients $c_{u'}$ appearing on the right-hand side of \eqref{eqn:schubrecur} in the case where $w = \widetilde{f}\in \symm_{2n}$. 
We need the following proposition.

\begin{prop}\label{prop:chain}
    For $f\in \bound(k,n)$, if
    $f'\in \genP_{\bound(k,n)\sqcup \{\star\}}(C_0(f))\setminus \{\star\}$, then   $\widetilde{f'}\in \genP_{\symm_{2n}}(C_{n}(\widetilde{f}))$. Moreover, $c_{f'} = c_{\widetilde{f'}}$.
\end{prop}

We defer the proof of \Cref{prop:chain} to the next subsection. The key input is the poset isomorphism \(W_{A_{n-1}}f \cong W_P\widetilde f\) from Proposition~\ref{prop:poset-iso},
which is compatible with taking least upper bounds.
We then compare the recursive inclusion--exclusion constructions of \(\genP\) on both sides.

We now give a proof of our main theorem assuming \Cref{prop:chain}.

\begin{proof}[Proof of Theorem~\ref{thm:hilbrecur}]
The proof of this theorem builds on the following results:
\begin{enumerate}[(i)]
    \item the expansion \eqref{eqn:positroidrecur} for positroid Hilbert series with coefficients \(c_{f'}\);
    \item the analogous expansion \eqref{eqn:schubrecur} for matrix Schubert varieties with coefficients \(c_{u'}\);
    \item the coefficient comparison \(c_{f'}=c_{\widetilde{f'}}\) from Proposition~\ref{prop:chain};
    \item the identification of nonzero coefficients with Lenart endpoints from Lemma~\ref{lemma:Ln};
    \item  the equivalence \(\widetilde{f'}\in \Lenart_n(\widetilde f)\iff f'\in \Lenart_0(f)\) from Lemma~\ref{lemma:relateL}.
\end{enumerate}

Recall from \eqref{eqn:schubrecur} and \eqref{eqn:positroidrecur} that 
\[
    \hilb(S/\bigcap_{u\in C_n(\widetilde{f})}\mathcal{I}_u) = \sum_{u'\in \genP_{\symm_{2n}}(C_n(\widetilde{f}))} c_{u'}\hilb(S/\mathcal{I}_{u'}),
\]
\begin{equation}\label{eqn:Jun12aaa}
    \hilb(R/\bigcap_{u\in C_0(f)}\cJ_u) = \sum_{f'\in \genP_{\bound(k,n)\sqcup \{\star\}}(C_0(f))} c_{f'}\hilb(R/\cJ_{f'}).
\end{equation}
By \Cref{prop:chain}, for each $f'\neq \star$ appearing in \eqref{eqn:Jun12aaa}, $\widetilde{f'}\in \genP_{\symm_{2n}}(C_n(\widetilde{f}))$ and $c_{f'} = c_{\widetilde{f'}}$. Combining with \Cref{lemma:Ln}, we have $c_{f'}\neq 0$ if and only if $\widetilde{f'}\in \Lenart_n(\widetilde{f})$, and $c_{f'} = (-1)^{\ell(\widetilde{f'}) - \ell(\widetilde{f}) -1} = (-1)^{\ell(f') - \ell(f) -1}$ if it is non-zero. By \Cref{lemma:relateL}, $\widetilde{f'}\in \Lenart_n(\widetilde{f})$ if and only if $f'\in \Lenart_0(f)$, and 
we are left to show that $c_\star = 0$. 

Suppose to the contrary that $c_\star\neq 0$, then $\star \in \genP_{\bound(k,n)\sqcup \{\star\}}(C_0(f))$. This implies that there are $f_1,\dots,f_m\in \genP_{\bound(k,n)\sqcup \{\star\}}(C_0(f))\setminus \{\star\}$ such that $\star \in \lub(f_1,\dots,f_m)$. Geometrically, this is 
$\bigcap_{i = 1}^m \Pi_{f_i} = \emptyset$.
Let $[v,u]$ and $[v_i,u_i]$ be the Grassmann intervals such that $\Pi_{f} = \Pi_v^u$ and $\Pi_{f_i} = \Pi_{v_i}^{u_i}$. Since $f_i\in W_{A_{n-1}}f$ for all $f_i$, by Lemma~\ref{lemma:wtouv}, $u_i = u$ for all $i$. Therefore the $T$-fixed point indexed by $u([k])$ is a point on $\Pi_{f_i}$ for all $i$ and the intersection is not empty, a contradiction. 

The statement in \eqref{eqn:Jun1aaa} is precisely part (2) of \Cref{lemma:bounded}.
\end{proof}

\begin{remark}
As an anonymous referee pointed out, Equation \eqref{eqn:cotrans} in  \Cref{prop:cut-top-point} identifies the hyperplane section
$\{[\topp(f)]=0\}\cap \Pi_f$ scheme-theoretically with the union $\bigcup_{f'\in C_0(f)}\Pi_{f'}$; equivalently, $R/\bigcap_{f'\in C_0(f)}\cJ_{f'}$ is the coordinate ring of this union.
In the Frobenius split setting, the scheme-theoretic intersections among these strata are reduced, so
$\bigcup_{f'\in C_0(f)}\Pi_{f'}$ admits a formal inclusion--exclusion/M\"obius inversion expansion over the associated
intersection poset.
What is not immediate from this formalism is an explicit description of the nonzero terms and coefficients. On the other hand, our proof via
\Cref{thm:hilbrecur} shows that only $g\in \Lenart_0(f)\cap\bound(k,n)$ contribute, with coefficient
$(-1)^{\ell(g)-\ell(f)+1}$ (and \eqref{eqn:Jun1aaa}), which we obtain by comparison with the matrix Schubert case via
$f\mapsto \widetilde f$ (Definition~\ref{def:tilde}) and \Cref{cor:grothendieck-agree}.
\end{remark}

\subsection{Proof of Lemma~\ref{lemma:relateL} and Proposition~\ref{prop:chain}}
\label{subsec:proof}

Let $W_{A_{n-1}}\cong \symm_n $ be the maximal parabolic subgroup of $W_{\widehat{A}_{n-1}}$ generated by $\{s_1,\cdots,s_{n-1}\}$, and let $W_{P}\cong \symm_n\subset \symm_{2n}$ be the parabolic subgroup generated by the set of simple reflections $\{s_{k+1},\cdots, s_{n+k-1}\}$. We now extend the definition of $\widetilde{g}$ to include the case where $g\in W_{\widehat{A}_{n-1}}f$.
Let $g = wf\in W_{A_{n-1}}f$ and let $(i_1,\dots,i_\ell)$ be any reduced word of $w\in \symm_n$. Define $\widetilde{g} = w'\widetilde{f}$ where $w' = \prod_{j = 1}^\ell s_{k+i_j}\in W_P$. 
Consider the map 
\begin{align}\label{eqn:extendtilde}
    \begin{split}
        W_{A_{n-1}}f &\longrightarrow W_P\widetilde{f}\\
        g&\mapsto \widetilde{g}
    \end{split}.
\end{align}
This is an extension of the map $f\mapsto \widetilde{f}$ in Definition~\ref{def:tilde}. 

Our next proposition asserts that the map in \eqref{eqn:extendtilde} commutes with taking least upper bounds. 
\begin{prop}
    \label{prop:poset-iso}
     Let $f\in\bound(k,n)$. 
    Then the map $g\mapsto \tilde{g}$ in \eqref{eqn:extendtilde} is an isomorphism of posets:
    \[W_{A_{n-1}}f\cong W_P\widetilde{f}.\]
    Furthermore, for any $U\subseteq W_{A_{n-1}}f$, 
    \begin{equation}\label{eqn:lub}
        \widetilde{\lub_{W_{\widehat{A}_{n-1}}}(U)} = \lub_{\symm_{2n}}(\widetilde{U}).
    \end{equation}
\end{prop}

\begin{proof}
For the first claim, it is enough to show that if $g_1\gtrdot g_2\in W_{A_{n-1}}f$, then $\tilde{g_1}\gtrdot \tilde{g_2}$.
For any $g_1\gtrdot g_2$, we have $g_1 = t_{i,j}g_2$ for some $1\leq i<j\leq n$, and 
\begin{equation}\label{eqn:Dec1aaa}
    g_2^{-1}(m)\notin [g_2^{-1}(i),g_{2}^{-1}(j)]\text{ for all }i<m<j.
\end{equation}
Since $\widetilde{g_2}^{-1}(m+k) = g_2^{-1}(m)+n$ for all $m\in [n]$, 
\[\widetilde{g_2}^{-1}(m+k)\notin [\widetilde{g_2}^{-1}(i+k),\widetilde{g_{2}}^{-1}(j+k)]\text{ for all }i<m<j.\] 
Therefore $\widetilde{g_1} = t_{i+k,j+k}\, \widetilde{g_2}\gtrdot \widetilde{g_2}$ and the map $g\mapsto \widetilde{g}$ induces an injection in \eqref{eqn:Dec1aaa}. Since both $W_{A_{n-1}}f$ and $W_P\widetilde{f}$ are both isomorphic to $\symm_n$ as posets, the map in \eqref{eqn:extendtilde} is an isomorphism.

We now prove \eqref{eqn:lub}. Pick any subset $U\subseteq W_{A_{n-1}}f$. 
By Theorem~\ref{thm:lub}, for any $U\subseteq W_{A_{n-1}}f$ (resp. $W_P \tilde{f}$), the least upper bounds of $U$ in the poset $W_{\widehat{A}_{n-1}}$ (resp. $\symm_{2n}$) all lie inside the coset $W_{A_{n-1}}f$ (resp. $W_P \tilde{f}$). Therefore 
\[\lub_{W_{\widehat{A}_{n-1}}}(U) = \lub_{W_{A_{n-1}}f}(U)\text{ and }\lub_{\symm_{2n}}(\widetilde U) = \lub_{W_P\widetilde{f}}(\widetilde U).\]
\eqref{eqn:lub} then follows from the isomorphism of the two posets.
\end{proof}

To complete the proof of \Cref{prop:chain}, we need the following two technical lemmas.
\begin{lemma}
    \label{lem:Cf-in-coset}
      For $f\in \bound(k,n)$, $C_0(f)\subset W_{A_{n-1}}f$.
\end{lemma}
\begin{proof}
    Let $f'=ft_{ab}\gtrdot f$ where $b>0$ and $a\le 0$ and $f'\in \bound(k,n)$. Since $f$ is bounded and $f'\gtrdot f$, we must have $b\le n$ and $a\ge -n+1$. Since $f'$ is also bounded, we must have $f(b)\le n$ as well as $f(a)\ge 1$. Therefore $f'=t_{f(a)f(b)}f\in W_{A_{n-1}}f$. 
\end{proof}

\begin{lemma}\label{lemma:bounded}
     For $w\in \mathcal{L}_n(\widetilde{f})$, if $w\notin \widetilde{\bound(k,n)}$, then 
     \begin{enumerate}
         \item $w'\notin\widetilde{\bound(k,n)}$ for all $w'\in \mathcal{L}_n(\widetilde{f})$ such that $w'\geq w$;
         \item either $w = \widetilde{f}t_{ab}\gtrdot_k \widetilde{f}$ or there exists some $w'\in \mathcal{L}_{n}(\widetilde{f})$ such that $w'< w$ and $w'\notin\widetilde{\bound(k,n)}$.
     \end{enumerate}
\end{lemma}
\begin{proof}
    We assume $a<b$ for the reflection $t_{ab}$. 
    
    \noindent $(1)$:
    Notice that for any $w,w'\in \mathcal{L}_n(\widetilde{f})$, $w^{-1}(i) = w'^{-1}(i)$ for all $i\in [k]\cup [n{+}k{+}1,2n]$. 
    It is then enough to prove the Lemma when $w' = t_{ab}w$ where $a,b\in [k+1,n+k]$.
    If $w\notin \{\widetilde{g}:g\in \bound(k,n)\}$, then 
    $w^{-1}(i)\notin [i-k-n,i-k]$
    for some $i\in [k+1,n+k]$. If $a,b\neq i$, then $w'^{-1}(i) = w^{-1}(i)\notin [i-k,i-k+n]$. Hence $w'\notin \widetilde{\bound(k,n)}$. If $a=i<b$, then either
    \begin{equation}\label{eqn:Nov4aaa}
        w^{-1}(i)<i-k<b-k\text{ or } w^{-1}(b)>w^{-1}(i)>i-k+n.
    \end{equation}
    Since $w'^{-1}(b) = w^{-1}(i)$ and $w'^{-1}(i) = w^{-1}(b)$, by \eqref{eqn:Nov4aaa}, either 
    \[w'^{-1}(b)<b-k\text{ or }w'^{-1}(i)>i-k+n.\]
    Since $i,b\in [k+1,n+k]$, $w'\notin \widetilde{\bound(k,n)}$.
    Similarly, if $a<b=i$, then either 
    \begin{equation}\label{eqn:Nov4bbb}
        w^{-1}(a)<w^{-1}(i)<i-k \text{ or }w^{-1}(i)>i-k+n>a-k+n.
    \end{equation}
    Since $w'^{-1}(a) = w^{-1}(i)$ and $w'^{-1}(i) = w^{-1}(a)$, by \eqref{eqn:Nov4bbb}, either
    \[w'^{-1}(i)<i-k\text{ or }w'^{-1}(a)>a-k+n.\]
    We can then conclude that $w'\notin \widetilde{\bound(k,n)}$ for $w'\geq w \in \mathcal{L}_n(\widetilde{f})$. 
    
    \medskip

    \noindent $(2)$: Suppose $w\neq \widetilde{f}t_{ab}\gtrdot_k \widetilde{f}$, then by the definition of $\Lenart_n(\widetilde{f})$, there is some $u\in \Lenart_n(\widetilde{f})$ such that 
    \begin{equation}
        u\lessdot_k ut_{a_1b_1} \lessdot_k ut_{a_1b_1}t_{a_2b_2} = w
    \end{equation}
    where either ($b_1 = b_2$ and $a_1<a_2$)  or ($b_1>b_2$). If $u$ or $ut_{a_1b_1}$ is not bounded then we are done. Suppose they are both bounded, let $ut_{a'b'}$ be the unique permutation different from $ut_{a_1b_1}$ such that $u<ut_{a'b'}<w$. We are then done by the following claim.
    \begin{claim}\label{claim:Nov5aaa}
        $ut_{a'b'}\in \Lenart_n(\widetilde{f})$ and $ut_{a'b'}\notin \widetilde{\bound(k,n)}$. 
    \end{claim}
\noindent \emph{Proof of Claim~\ref{claim:Nov5aaa}:}
        We divide into three cases.
        \medskip
        
        \noindent \textbf{Case I} ($a_1,a_2,b_1,b_2$ are all distinct): In this case $t_{a_1b_1}$ and $t_{a_2b_2}$ commute. Hence
        $a' = a_2,b' = b_2$ and $ut_{a'b'}\in \Lenart_n(\widetilde{f})$. Since $ut_{a_1b_1}$ is bounded and $w$ is not, either $w(a_2)>a_2+k$ or $w(b_2)<b_2-n+k$. Since $ut_{a'b'}(a_2) = w(a_2)$ and $ut_{a'b'}(b_2) = w(b_2)$, $ut_{a'b'}\notin \widetilde{\bound(k,n)}$.

        \medskip
        \noindent \textbf{Case II} ($b_1 = b_2$ and $a_1<a_2$): In this case we have $u(a_2)<u(a_1)<u(b_1) = u(b_2)$. Then $a' = a_2,b' = b_1 = b_2$ and thus $ut_{a'b'}\in \Lenart_n(\widetilde{f})$. Since $ut_{a_1b_1}$ is bounded, $w$ is not bounded if and only if $w(b')< b'-n+k$. Since $ut_{a'b'}(b') = w(b')$, $ut_{a'b'}\notin \widetilde{\bound(k,n)}$.

        \medskip
        \noindent \textbf{Case III} ($a_1 = a_2$ and $b_1>b_2$): In this case we have $u(a_1) = u(a_2)< u(b_1)< u(b_2)$. Then $a' = a_1 = a_2, b' = b_2$ and thus $ut_{a'b'}\in \Lenart_n(\widetilde{f})$. Similar to Case II, since $ut_{a_1b_1}$ is bounded and $w$ is not, we have $w(a')> a'+k$. Since $ut_{a'b'}(a') = w(a')$, $ut_{a'b'}\notin \widetilde{\bound(k,n)}$.
\qed

\end{proof}

\begin{proof}[Proof of Lemma~\ref{lemma:relateL}]
We first observe that for $h,h'\in \bound(k,n)$, $h'\gtrdot_0 h$ if and only if $\widetilde{h'}\gtrdot_n \widetilde{h}$.
Let $f,g\in \bound(k,n)$. If $g\in \Lenart_0(f)$,  by \Cref{def:L_0} there is a  chain 
\[f\lessdot_0 ft_{a_1,b_1} \lessdot_0 \dots \lessdot_0 g\]
in $\Gamma_0(f)$.
Notice that if $h\in \bound(k,n)$ and $h\gtrdot_0 h'$, then $h'\in \bound(k,n)$. Since $g\in \bound(k,n)$, every element in the chain above lies in $\bound(k,n)$. 
We then have a saturated chain in $n$-Bruhat order in $\symm_{2n}$
\[\widetilde{f}\lessdot_n \widetilde{f}t_{a_1+n,b_1+n}\lessdot_n \dots \lessdot_n \widetilde{g}. \]
Equivalently, $\widetilde{g}\in \Lenart_n(\widetilde{f})$. 

    On the other hand, suppose $\widetilde{g}\in \Lenart_{n}(\widetilde{f})$, then by \Cref{def:L_k} there is a chain 
    \[\widetilde{f}\lessdot_n \widetilde{f}t_{a'_1,b'_1}\lessdot_n \dots \lessdot_n \widetilde{g} \]
    in $\Gamma_n(\widetilde{f})$. 
    Since $g\in \bound(k,n)$, by part (1) of \Cref{lemma:bounded}, every element in this chain is in $\widetilde{\bound(k,n)}$. By the  observation at the beginning of the proof, we recover a saturated $0$-Bruhat chain from $f$ to $g$ in $\Gamma_0(f)$. Therefore $g\in \Lenart_0(f)$.
\end{proof}

\begin{proof}[Proof of Proposition~\ref{prop:chain}] Let $V:=(V_1,\cdots, V_p)$ (resp. $U:=(U_1,\cdots,U_{q})$) be a sequence of sets of bounded affine permutations (resp. permutations) satisfying 
    \begin{itemize}
        \item $V_1\subseteq C_0(f),U_1\subseteq C_n(\widetilde{f})$;
        \item $V_{p} = \{f'\}, U_{q} = \{\widetilde{f'}\}$;
        \item $V_{i+1}\subseteq \lub(V_{i}),U_{i+1}\subseteq \lub(U_{i})$.
    \end{itemize}
We claim that the map sending $\{V_i\}_{i\in [m]}$ to $\{\widetilde{V_i}\}_{i\in [m]}$ is a bijection between $\mathcal{V}:=\{V\}$ and $\mathcal{U}:=\{U\}$, the sets of all such $V$ and $U$ respectively.

By Lemma~\ref{lem:Cf-in-coset}, $\widetilde{C_0(f)}$ is well-defined and thus $\widetilde{V_1}\subseteq \widetilde{C_0(f)}\subseteq C_n(\widetilde{f})$.
For $i\geq 1$, by \eqref{eqn:lub},
\[\widetilde{V_{i+1}} \subseteq \widetilde{\lub(V_i)} = \lub(\widetilde{V_i}).\]
Therefore $\{\widetilde{V_{i}}\}_{i\in [m]}$ satisfy all three conditions in Proposition~\ref{prop:chain}. The injectivity follows from the injectivity of the map $g\mapsto \widetilde{g}$ as in Proposition~\ref{prop:poset-iso}. For surjectivity, consider any $\{U_i\}_{i\in [m]}$ such that $U_m = \{\widetilde{f'}\}\subseteq \mathcal{L}_n(\widetilde{f})$. Since $U_{i+1}\subseteq \lub(U_{i})$, by (1) of Lemma~\ref{lemma:bounded}, 
\[U_{i+1}\subset \widetilde{\bound(k,n)}\implies U_{i}\subset \widetilde{\bound(k,n)}.\]
Since $U_m = \{\widetilde{f'}\}$, we have $U_i\subset \widetilde{\bound(k,n)}$ for all $i\in [m]$. Therefore the map is surjective. 

Let $c_{\widetilde{f'}},c_{f'}$ be as in \eqref{eqn:schubrecur} and \eqref{eqn:positroidrecur} respectively. Then
\[c_{\widetilde{f}'} = \sum_{U\in\; \mathcal{U}}(-1)^{\sum_{i}|U_i|-1}\text{ and }c_{f'} = \sum_{V\in\; \mathcal{V}}(-1)^{\sum_{i}|V_i|-1}.\]
By the bijection between $\mathcal{V}$ and $\mathcal{U}$ above, we conclude that $c_{\widetilde{f'}} = c_{f'}$.
\end{proof}

\appendix

\section{Least upper bounds in parabolic subgroups of Coxeter groups}

Let $W$ be a Coxeter group and $W_P$ be a parabolic subgroup. The main theorem in this appendix is the following.

\begin{theorem}\label{thm:lub}
Let \(W\) be a Coxeter group with Bruhat order \(\le\), and let \(W_P\subseteq W\) be a parabolic subgroup.
Fix \(f\in W\). For \(x,y\in W\), let \(\lub(x,y)\) denote the set of \emph{least} upper bounds of \(\{x,y\}\)
in Bruhat order (possibly empty, possibly containing multiple elements).
If \(f_1,f_2\in W_P f\), then every least upper bound of \(\{f_1,f_2\}\) lies in the same coset:
\[
\lub(f_1,f_2)\subseteq W_P f.
\]
\end{theorem}

% \begin{theorem}\label{thm:lub}
% If $f_1,f_2\in W_P f$, then $\lub(f_1,f_2)\subset W_P f$.
% \end{theorem}

We begin by recalling the following classical result on Coxeter groups.
Recall that \(D_L(w)\) (resp.\ \(D_R(w)\)) denotes the left (resp.\ right) descent set of \(w\) in \(W\) (see \cite{BB2005}*{Sec.~2.2}).

\begin{lemma}[Lifting property]\cite{BB2005}*{Proposition~2.2.7}\label{lemma:lift}
    Suppose $u< w$ and $s_i\in D_{L}(w)\setminus D_{L}(u)$, then $s_iu\leq w$ and $u\leq s_iw$. Similarly, if $s_j\in D_{R}(w)\setminus D_{R}(u)$, then $us_j\leq w$ and $u\leq ws_j$.
\end{lemma}
An immediate corollary of the lifting property is the following.
\begin{cor}\label{cor:lift}
    For simple reflections $s,s'$ such that $w\leq sw,ws'$, then either $sw,ws'\leq sws'$ or $w = sws'$. 
\end{cor}

Recall that the \textbf{Demazure product} $\ast$ on a Weyl group $W$ is defined, for a simple reflection $s$, by
\begin{equation}
    w\ast s = \begin{cases}
        ws & \text{ if }ws>w\\
        w & \text{ otherwise}
    \end{cases}.
\end{equation}
For any (not necessarily reduced) word $Q = s_{i_1}\dots s_{i_m}$, define 
\[\dem(Q):= ((s_{i_1}\ast s_{i_2})\ast \dots)\ast s_{i_m}.\]

We will use the Demazure product as a convenient way to produce an element above a given reduced subword in Bruhat order.

\begin{lemma}
\label{lem:dem}
    Let $Q$ be a not necessarily reduced word that contains a subword $Q_w$ that is a reduced word for $w\in W$. Then $\dem(Q)\ge w$.
\end{lemma}
\begin{proof}We proceed by induction on the length of $Q$. 
    Write $Q=Q' s_{i}$. If $s_{i}$ is not in $Q_w$, by the induction hypothesis we have $\dem(Q')\ge w$. Since $\dem(Q)\ge \dem(Q')$ the result follows. If $s_{i}$ is in $Q_w$, we have $ws_{i}<w$ and by the induction hypothesis $u:=\dem(Q')\ge ws_{i}$. We now claim that $\dem(Q)=\dem(Q's_i)\ge w$.  If $us_i>u$ then $\dem(Q)=us_i$. Since $us_i>ws_i$ and $s_i\in D_{R}(us_i)\setminus D_{R}(ws_i)$, by the lifting property, we have $us_i = \dem(Q)\ge w$.
    If $us_i<u$,
    since $s_i\in D_{R}(u)\setminus D_{R}(ws_i)$, again by the lifting property, $u = \dem(Q)\geq w$.
\end{proof}

\begin{proof}[Proof of Theorem~\ref{thm:lub}]
    It is enough to assume that $f\in W^P$ is a minimal coset representative. Let $w_1,w_2\in W_P$ be elements such that $f_1 = w_1f, f_2 = w_2f$. 
    
    We proceed by induction on $\ell(f)$. The base case is $f=e$, the identity. Assume $w\in \lub(w_1,w_2)$. We show that $w\in W_P$. Suppose to the contrary that $w\not\in W_P$ and assume $Q$ is a reduced word of $w$. Then there exists a simple reflection $s_k\not \in W_P$ which appears in the reduced word of $Q$. Let $Q'$ be the (not necessarily reduced) word obtained from $Q$ by crossing out this $s_k$ and let $w'=\dem(Q')$. By assumption, $Q'$ contains a reduced subword of $w_i$ for $i=1,2$. By Lemma~\ref{lem:dem}, we have $w\ge w_i$ for $i=1,2$. Since $w'<w$, it contradicts that $w\in \lub(w_1, w_2)$.

    Suppose the statement holds for all $f\in W^P$ with $\ell(f)< \ell$. We first consider the case where there exists $s_i \in D_{R}(g)\cap D_{R}(f)$. Notice that we also have $s_i \in D_{R}(w_1f),D_{R}(w_2f)$. By the lifting property, $gs_i\geq w_1fs_i,w_2fs_i$. If $gs_i\notin \lub(w_1fs_i,w_2fs_i)$, we can find $h$ where 
    \[w_1fs_i,w_2fs_i \leq h<gs_i.\]
    Since $s_i\in D_{R}(h\ast s_i)\setminus (D_{R}(w_1fs_i)\cup D_{R}(w_2fs_i))$, by the lifting property, we get 
    \[h\ast s_i \geq w_1f,w_2f.\]
    This contradicts the assumption that $g\in \lub(w_1f,w_2f)$ since $g>h\ast s_i$. Therefore $gs_i\in \lub(w_1fs_i,w_2fs_i)$.  Now by the inductive hypothesis, we have $gs_i = wfs_i$ for some $w\in W_P$ and thus $g = wf$ as desired.

    We are now left with the case where $D_{R}(g) \cap D_{R}(f) = \emptyset$. For any $s_i\in D_{R}(g)$, if $s_i\notin D_{R}(w_1f)\cup D_{R}(w_2f)$, then, by the lifting property, $g>gs_i\geq w_1f,w_2f$. This is impossible since $g\in \lub(w_1f,w_2f)$. As a result, we assume without loss that $s_i\in D_{R}(w_1f)$. Set
    \[f' = \begin{cases}
        w_2f & \text{ if }s_i\notin D_{R}(w_2f)\\
        w_2fs_i & \text{otherwise}
    \end{cases}.\]
    \begin{claim}\label{claim:fsi}
        $fs_i\in W_P f$, hence $w_1fs_i,f'\in W_Pf$.
    \end{claim}
    \begin{proof}
        Pick any $a_1\dots a_{\ell}\in \Red(f)$ and $b_1\dots b_m\in \Red(w_1)$, then since $w_1\in W_P$ and $f\in W^P$, $b_1\dots b_m a_1\dots a_\ell \in \Red(w_1f)$. Since $s_i\notin D_{R}(f)$, $a_1\dots a_\ell\,  i \in \Red(fs_i)$. Since $\ell(w_1fs_i) = \ell(w_1f) - 1$, there is a unique $j\in [m]$ such that 
        \[s_{b_j}\cdots s_{b_m}fs_i< s_{b_{j+1}} \cdots s_{b_m}fs_i.\]
        Set $u = s_{b_{j+1}} \cdots s_{b_m}f\in W$, we have 
        \[u\leq s_{b_j}u, us_i\text{ and }us_i>s_{b_j}us_i.\]
        By Corollary~\ref{cor:lift}, $u = s_{b_j}us_{i}$. Since $s_{b_i}\in W_P$ for all $i\in [m]$, we have $u\in W_P f$ and $s_{b_j}us_i\in W_P fs_i$. 
        Therefore $W_Pf=W_Pfs_i$ and thus 
        $w_1fs_i\in W_P f$.
        If $f' = w_2f$, it is clear that $f'\in W_Pf$. Otherwise, $f'\in W_Pf$ follows from the same reasoning as above.
    \end{proof}
    
    \begin{claim}\label{claim:gsi}
         $gs_i\in \lub(w_1fs_i,f')$. 
    \end{claim}
    \begin{proof}
        We first show that $gs_i\geq w_1fs_i,f'$. Since $s_i\in D_{R}(g)\cap D_{R}(w_1f)$ and $g\geq w_1f$, by the lifting property we have $gs_i\geq w_1fs_i$. Since $f'\leq w_2f\leq g$ and $s_i\in D_{R}(g)\setminus D_{R}(f')$, $f'<g$.
        Again by the lifting property, $gs_i\geq f'$.

        Suppose that $gs_i$ is not a least upper bound, and pick $h<gs_i$ such that $h\geq w_1fs_i,f'$. Consider $h' = h\ast s_i$. Since $g>gs_i>h$, by the lifting property we have $g>h'$. Since $s_i\in D_{R}(h')\setminus (D_{R}(w_1fs_i)\cup D_{R}(f'))$, again by the lifting property we have $h'\geq w_1f,w_2f$. This contradicts the assumption that $g\in \lub(w_1f,w_2f)$. Therefore $gs_i\in \lub(w_1fs_i,f')$. 
    \end{proof}

    We can now finish the proof by induction on $\ell(w_1)$ and $\ell(w_2)$. If either $\ell(w_1)$ or $\ell(w_2)$ is $0$, then $g = \max(w_1f,w_2f)$ and the theorem statement holds. By Claim~\ref{claim:fsi}, we write $$w_1fs_i = w_1'f\text{ and }f' = w_2'f$$
    where 
    $$w_1',w_2'\in W_P\text{ and }\ell(w_1')<\ell(w_1),\ell(w_2')\leq \ell(w_2).$$ 
    By Claim~\ref{claim:gsi}, $gs_i\in \lub(w_1'f,w_2'f)$ and by the inductive hypothesis, 
    $gs_i = w'f$ for some $w'\in W_P$. We are then done by Claim~\ref{claim:fsi} since
    \[
g = gs_is_i = w'fs_i \in W_Pf.\qedhere
    \]
\end{proof}

\section{Distributivity from concatenating Gr\"obner bases}\label{appendix: distributivity}
In this section we give a Gr\"obner-theoretic criterion ensuring that the sublattice generated by a finite family of ideals in a polynomial ring is distributive under sum and intersection. Using the characterization of distributive families via exactness of an associated \v{C}ech-type complex, we show that if initial ideals commute with sums for a fixed term order, then the family is distributive.\footnote{Thanks to Keller VandeBogert for pointing out this criterion for distributivity.} As a consequence, we obtain inclusion--exclusion formulas for multigraded Hilbert series in settings where such Gr\"obner compatibility holds.

Throughout this section, let $S=\k[x_1,\dots,x_n]$ be a polynomial ring in $n$ variables over a field $\k$.
Let $\cI=\{I_1,\dots,I_r\}$ be a finite family of ideals in $S$, and write $[r]=\{1,\dots,r\}$.
Let $\mathsf{L}(\cI)$ denote the \emph{sublattice} of the lattice of ideals of $S$ generated by $\{I_1,\dots,I_r\}$ under the operations
\[
A\wedge B:=A\cap B,\qquad A\vee B:=A+B.
\]

\begin{definition}\label{def:distributive-family}
We say that $\cI$ is \textbf{distributive} if the lattice $\mathsf{L}(\cI)$ is distributive, i.e.\ for all $A,B,C\in\mathsf{L}(\cI)$,
\begin{equation}\label{eq:distributive}
(A+B)\cap C = (A\cap C) + (B\cap C).
\end{equation}
\end{definition}

\begin{remark}\label{rem:monomial-distributive}
If $J,K\subseteq S$ are monomial ideals, let $\Mon(J)$ denote the set of monomials contained in $J$.
Then $\Mon(J)$ is an upset in the divisibility poset on monomials, and one has
\[
\Mon(J+K)=\Mon(J)\cup \Mon(K),\qquad \Mon(J\cap K)=\Mon(J)\cap \Mon(K).
\]
Consequently the lattice of monomial ideals is distributive (it identifies with a sublattice of the
Boolean lattice of subsets of monomials). In particular, any family of monomial ideals is distributive,
so \eqref{eq:distributive} holds for all monomial ideals $A,B,C$.
\end{remark}

For $U\subseteq [r]$ write
\[
I_U := \sum_{u\in U} I_u \quad (\text{with } I_\emptyset := 0),\qquad
I^U := \bigcap_{u\in U} I_u \quad (\text{with } I^\emptyset := S).
\]

\begin{definition}\label{def:cech-complexes}
Define the cochain complex $C^\bullet(S;\cI)$ by
\[
C^p(S;\cI) := \bigoplus_{\substack{U\subseteq [r]\\ |U|=p}} \frac{S}{I_U}.
\]
For a fixed $U=\{u_1<\cdots<u_p\}$ and $j\notin U$, let
\[
\rho_{U,j}: \frac{S}{I_U}\longrightarrow \frac{S}{I_{U\cup\{j\}}}
\]
be the natural quotient map. Define $d^p:C^p\to C^{p+1}$ on the $U$-summand by
\[
d^p|_{S/I_U} := \sum_{j\notin U} (-1)^{\#\{u\in U: u<j\}}\ \rho_{U,j}.
\]
Similarly, define the chain complex $C_\bullet(S;\cI)$ by
\[
C_p(S;\cI) := \bigoplus_{\substack{U\subseteq [r]\\ |U|=p}} I^U.
\]
For $U=\{u_1<\cdots<u_p\}$ and $1\le \ell\le p$, let $\iota_{U,\ell}:I^U\hookrightarrow I^{U\setminus\{u_\ell\}}$
be the inclusion map. Define $\partial_p:C_p\to C_{p-1}$ on the $U$-summand by
\[
\partial_p|_{I^U}:=\sum_{\ell=1}^p (-1)^{\ell-1}\,\iota_{U,\ell}.
\]

\end{definition}

\begin{theorem}[{\cite{polishchuk2005quadratic}*{Proposition~1.7.2}}]\label{thm:pp-equivalence}
Let $\cI$ be a family of ideals in $S$. The following are equivalent:
\begin{enumerate}
\item $\cI$ is a distributive family.
\item $C^\bullet(S;\cI)$ is exact in all positive cohomological degrees.
\item $C_\bullet(S;\cI)$ is exact in all positive homological degrees.
\end{enumerate}
\end{theorem}

\begin{prop}\label{prop:hilbert-series-inclusion-exclusion}
Assume $S$ is $\NN^d$-graded with $\dim_\k(S_\aa)<\infty$ for all multidegrees $\aa\in\NN^d$, and that
$I_1,\dots,I_r$ are homogeneous with respect to this grading.
If $\cI$ is distributive, then
\begin{equation}\label{eq:hilb-inclusion-exclusion}
\Hilb\!\left(\frac{S}{\bigcap_{i=1}^r I_i}; t_1,\dots, t_d\right)
=\sum_{\emptyset\neq U\subseteq [r]} (-1)^{|U|+1}\ \Hilb\!\left(\frac{S}{\sum_{u\in U} I_u}; t_1,\dots, t_d\right).
\end{equation}
\end{prop}

\begin{proof}
Since the $I_i$ are homogeneous, each $C^p(S;\cI)$ inherits an $\NN^d$-grading from $S$, and all differentials are homogeneous of multidegree $0$.
By Theorem~\ref{thm:pp-equivalence}, $H^p(C^\bullet(S;\cI))=0$ for $p>0$.

Note that $C^0(S;\cI)=S$ and $d^0:S\to \bigoplus_{i=1}^r S/I_i$ is the diagonal quotient map, so
\[
H^0(C^\bullet(S;\cI))=\ker(d^0)=\bigcap_{i=1}^r I_i \subseteq S.
\]
Fix a multidegree $\aa\in\NN^d$. Taking the alternating sum of $\k$-dimensions in degree $\aa$ (Euler characteristic) gives
\[
\dim_\k\!\left(\bigcap_{i=1}^r I_i\right)_\aa
= \sum_{p=0}^r (-1)^p \sum_{\substack{U\subseteq [r]\\ |U|=p}}
\dim_\k\!\left(\frac{S}{I_U}\right)_\aa.
\]
Convert $\dim_\k(\bigcap_{i=1}^r I_i)_\aa$ to $\dim_\k(S/(\bigcap_{i=1}^r I_i))_\aa$ using the short exact sequence
$0\to \bigcap_{i=1}^r I_i \to S \to S/(\bigcap_{i=1}^r I_i)\to 0$, i.e.
\[
\dim_\k\!\left(\frac{S}{\bigcap_{i=1}^r I_i}\right)_\aa
=\dim_\k(S_\aa)-\dim_\k\!\left(\bigcap_{i=1}^r I_i\right)_\aa.
\]
After cancellation of the $p=0$ term (which is $\dim_\k(S_\aa)$), the remaining alternating sum is exactly \eqref{eq:hilb-inclusion-exclusion} degreewise. Summing over $\aa$ yields the Hilbert series identity.
\end{proof}

Fix a monomial order $\prec$ on $S$. For $0\neq f\in S$, write $\inn_\prec(f)$ for its leading monomial
and $\inn_\prec(I)$ for the monomial ideal generated by $\{\inn_\prec(f):f\in I\}$.

\begin{theorem}\label{thm:gb-concatenate-distributive}
Let $\prec$ be a monomial order on $S$, and let $\cI=\{I_1,\dots,I_r\}$ be homogeneous ideals.
Assume that for every subset $U\subseteq [r]$,
\begin{equation}\label{eq:initial-commutes-with-sums}
\inn_\prec\!\left(\sum_{u\in U} I_u\right) \;=\; \sum_{u\in U}\inn_\prec(I_u).
\end{equation}
Then $\cI$ is a distributive family.
\end{theorem}

\begin{proof}[Proof using filtrations]
By Theorem~\ref{thm:pp-equivalence}, it suffices to prove that 
$C^\bullet(S;\cI)$ is exact in positive degrees.

Fix an integer $q\ge 0$ (total degree in the standard $\NN$-grading).
Let $\Mon_q=\{m_{q,1}\prec \cdots \prec m_{q,N_q}\}$ 
be the degree $q$ monomials ordered increasingly by $\prec$.
Define a filtration on $S_q$ by
\[
F_p S_q := \Span_\k\{m_{q,1},\dots,m_{q,p}\}
\qquad (0\le p\le N_q),
\]
and let $F_p(S/J)_q$ denote the image of $F_pS_q$ in $(S/J)_q$ for any homogeneous ideal $J$.

For each $k\ge 0$, the degree $q$ strand of the complex $C^\bullet(S; \cI)$ is
\[
C^k(S;\cI)_q
=
\bigoplus_{\substack{U\subseteq [r]\\ |U|=k}}
(S/I_U)_q,
\]
equipped with the direct-sum filtration
\[
F_p C^k(S;\cI)_q
:=
\bigoplus_{\substack{U\subseteq [r]\\ |U|=k}}
F_p(S/I_U)_q.
\]
Since the differentials are induced by $\id_S$, they preserve this filtration.

For any homogeneous ideal $J$, there is a canonical isomorphism
\begin{equation}\label{eq:gr-quotient-initial-tight}
\gr_F\!\bigl((S/J)_q\bigr)
\;\cong\;
(S/\inn_\prec(J))_q.
\end{equation}
This follows because the class of $m_{q,p}$ in 
$F_p(S/J)_q/F_{p-1}(S/J)_q$ vanishes precisely when 
$m_{q,p}\in \inn_\prec(J)$, so the associated graded has a basis given by the degree $q$ monomials not in $\inn_\prec(J)$.

Applying \eqref{eq:gr-quotient-initial-tight} with $J=I_U$ and using
\eqref{eq:initial-commutes-with-sums}, we obtain degreewise identifications
\[
\gr_F\!\bigl((S/I_U)_q\bigr)
\;\cong\;
\bigl(S/\sum_{u\in U}\inn_\prec(I_u)\bigr)_q,
\]
compatible with the differentials.
Hence the associated graded of the degree $q$ strand of
$C^\bullet(S;\cI)$ identifies with the degree $q$ strand of
\[
C^\bullet\!\bigl(S;\{\inn_\prec(I_1),\dots,\inn_\prec(I_r)\}\bigr).
\]

Each $\inn_\prec(I_i)$ is a monomial ideal, so by Remark~\ref{rem:monomial-distributive} the family
$\{\inn_\prec(I_1),\dots,\inn_\prec(I_r)\}$ is distributive. Hence, by Theorem~\ref{thm:pp-equivalence},
the complex $C^\bullet\!\bigl(S;\{\inn_\prec(I_1),\dots,\inn_\prec(I_r)\}\bigr)$ is exact in all positive degrees.
Thus $\gr_F\bigl(C^\bullet(S;\cI)_q\bigr)$ is exact in positive degrees.

By the standard filtered-complex argument (e.g.\ \cite{weibel1994introduction}*{Theorem~5.5.1}),
exactness of the associated graded in positive degrees implies exactness of
$C^\bullet(S;\cI)_q$ in positive degrees.
Since this holds for every $q$, the full complex $C^\bullet(S;\cI)$
is exact in positive degrees.

By Theorem~\ref{thm:pp-equivalence}, $\cI$ is distributive.
\end{proof}

\begin{example}[Matrix Schubert determinantal ideals]
Let $S=\k[z_{i,j}]_{i,j\in[n]}$ with $\deg(z_{i,j})=t_i$, and let $I_1,\dots,I_r$ be Schubert determinantal ideals.
For any antidiagonal term order, the compatibility \eqref{eq:initial-commutes-with-sums} is proved in
\cite{knutson2009frobenius}*{Theorem~4, Corollary~2, and Sec.~7.2}.
Moreover, the relevant initial ideals are monomial, so \Cref{thm:gb-concatenate-distributive} applies and yields the Hilbert series identity \Cref{eq:hilb-inclusion-exclusion}.
\end{example}

\begin{example}[Positroid ideals]
Let $S=R(k,n)$ with $\deg(P_J)=t^J$, and let $I_1,\dots,I_r$ be positroid ideals.
Under the reverse lexicographic term order, \Cref{thm:std-mon-intersect} establishes \eqref{eq:initial-commutes-with-sums}.
Hence \Cref{thm:gb-concatenate-distributive} applies, and we obtain the Hilbert series identity \Cref{eq:hilb-inclusion-exclusion}.
\end{example}

\bibliographystyle{amsplain}
\bibliography{biblio}
% \addcontentsline{toc}{section}{Bibliography}

\end{document}